\RequirePackage[l2tabu]{nag}		% Warns for incorrect (obsolete) LaTeX usage
\documentclass[a4paper,11pt,leqno,openbib]{memoir} %add 'draft' to turn draft option on (see below)
\usepackage{datetime}
\usepackage{ifpdf}
\ifpdf
\fi
\ifdraftdoc 
	\usepackage{draftwatermark}				%Sets watermarks up.
	\SetWatermarkScale{0.3}
	\SetWatermarkText{\bf Draft: \today}
\fi
\newsubfloat{figure}
\newsubfloat{table}
\settrimmedsize{297mm}{210mm}{*}
\settypeblocksize{634pt}{448.13pt}{*} 
\setulmargins{4cm}{*}{*} 
\setlrmargins{*}{*}{1.5} 
\setmarginnotes{17pt}{51pt}{\onelineskip} 
\setheadfoot{\onelineskip}{2\onelineskip} 
\setheaderspaces{*}{2\onelineskip}{*} 
\checkandfixthelayout
\usepackage{fouriernc}
\usepackage[T1]{fontenc}
\OnehalfSpacing 
\setsecnumdepth{subsection} 
\maxsecnumdepth{subsubsection}
\usepackage{calc,soul,fourier}
\makeatletter 
\newlength\dlf@normtxtw 
\newsavebox{\feline@chapter} 
\newcommand\feline@chapter@marker[1][4cm]{%
	\sbox\feline@chapter{% 
		\resizebox{!}{#1}{\fboxsep=1pt%
			\colorbox{gray}{\color{white}\thechapter}% 
		}}%
		\rotatebox{90}{% 
			\resizebox{%
				\heightof{\usebox{\feline@chapter}}+\depthof{\usebox{\feline@chapter}}}% 
			{!}{\scshape\so\@chapapp}}\quad%
		\raisebox{\depthof{\usebox{\feline@chapter}}}{\usebox{\feline@chapter}}%
} 
\newcommand\feline@chm[1][4cm]{%
	\sbox\feline@chapter{\feline@chapter@marker[#1]}% 
	\makebox[0pt][c]{% aka \rlap
		\makebox[1cm][r]{\usebox\feline@chapter}%
	}}
\makechapterstyle{daleifmodif}{

	\renewcommand\printchapternum{\null\hfill\feline@chm[2.5cm]\par}

} 
\makeatother 
\chapterstyle{daleifmodif}

\makepagestyle{myvf} 
\makeoddfoot{myvf}{}{\thepage}{} 
\makeevenfoot{myvf}{}{\thepage}{} 
\makeheadrule{myvf}{\textwidth}{\normalrulethickness} 
\makeevenhead{myvf}{\small\textsc{\leftmark}}{}{} 
\makeoddhead{myvf}{}{}{\small\textsc{\rightmark}}
\newcommand{\clearemptydoublepage}{\newpage{\thispagestyle{empty}\cleardoublepage}}
\makeindex
\usepackage{import}

\usepackage{lipsum}					%Needed to create dummy text
\usepackage{amsfonts} 					%Calls Amer. Math. Soc. (AMS) fonts
\usepackage[centertags]{amsmath}			%Writes maths centred down
\usepackage{stmaryrd}					%New AMS symbols
\usepackage{amssymb}					%Calls AMS symbols
\usepackage{amsthm}					%Calls AMS theorem environment
\usepackage{newlfont}					%Helpful package for fonts and symbols
\usepackage{layouts}					%Layout diagrams
\usepackage{graphicx}					%Calls figure environment
\usepackage{longtable,rotating}			%Long tab environments including rotation. 
\usepackage[utf8]{inputenc}				%directly for mac
\DeclareUnicodeCharacter{2010}{-}% support older LaTeX versions some hack that was suddenly necessary on my pc
\usepackage{colortbl}					%Makes coloured tables
\usepackage{wasysym}					%More math symbols
\usepackage{mathrsfs}					%Even more math symbols
\usepackage{float}						%Helps to place figures, tables, etc. 
\usepackage{verbatim}					%Permits pre-formated text insertion
\usepackage{upgreek }					%Calls other kind of greek alphabet
\usepackage{latexsym}					%Extra symbols
\usepackage[square,numbers,
		     sort&compress]{natbib}		%Calls bibliography commands 
\usepackage{url}						%Supports url commands
\usepackage{etex}						%eTeXÕs extended support for counters
\usepackage{fixltx2e}					%Eliminates some in felicities of the 
\usepackage[spanish,ngerman,french,main=english]{babel}		%For languages characters and hyphenation
\usepackage{lipsum}%\usepackage{blindtext}					% lore ipsum filler
\usepackage{color}                    				%Creates coloured text and background
\usepackage[colorlinks=true,
		     allcolors=black]{hyperref}              %Creates hyperlinks in cross references
\usepackage{memhfixc}					%Must be used on memoir document 
\usepackage{enumerate}					%For enumeration counter
\usepackage{footnote}					%For footnotes
\usepackage{microtype}					%Makes pdf look better.
\usepackage{rotfloat}					%For rotating and float environments as tables, 
\usepackage{flafter}				% put figure after reference
\usepackage{alltt}						%LaTeX commands are not disabled in 
\usepackage[version=0.96]{pgf}			%PGF/TikZ is a tandem of languages for producing vector graphics from a 
\usepackage{tikz}						%geometric/algebraic description. 
\usetikzlibrary{arrows,shapes,decorations,
		       automata,backgrounds,
		       petri,topaths,decorations.markings}				%To use diverse features from tikz		
\newcommand{\pgftextcircled}[1]{                                                                    %Defines encircled text
    \setbox0=\hbox{#1}%
    \dimen0\wd0%
    \divide\dimen0 by 2%
    \begin{tikzpicture}[baseline=(a.base)]%
        \useasboundingbox (-\the\dimen0,0pt) rectangle (\the\dimen0,1pt);
        \node[circle,draw,outer sep=0pt,inner sep=0.1ex] (a) {#1};
    \end{tikzpicture}
}
\newcommand{\blackged}{\hfill$\square$}
\newcommand{\whiteged}{\hfill$\square$}
\newcounter{proofcount}
\renewenvironment{proof}[1][\proofname.]{\par
 \ifnum \theproofcount>0 \pushQED{\whiteged} \else \pushQED{\blackged} \fi%
 \refstepcounter{proofcount}
 \normalfont 
 \trivlist
 \item[\hskip\labelsep
       \itshape
   {\bf\em #1}]\ignorespaces
}{%
 \addtocounter{proofcount}{-1}
 \popQED\endtrivlist
}
\let\oldsqrt\sqrt
\def\sqrt{\mathpalette\DHLhksqrt}
\def\DHLhksqrt#1#2{%
\setbox0=\hbox{$#1\oldsqrt{#2\,}$}\dimen0=\ht0
\advance\dimen0-0.2\ht0
\setbox2=\hbox{\vrule height\ht0 depth -\dimen0}%
{\box0\lower0.4pt\box2}}
\newcommand{\mycaption}[2][\@empty]{
	\captionnamefont{\scshape} 
	\changecaptionwidth
	\captionwidth{0.9\linewidth}
	\captiondelim{.\:} 
	\indentcaption{0.75cm}
	\captionstyle[\centering]{}
	\setlength{\belowcaptionskip}{10pt}
	\ifx \@empty#1 \caption{#2}\else \caption[#1]{#2}
}
\newcommand{\mysubcaption}[2][\@empty]{
	\subcaptionsize{\small}
	\hangsubcaption
	\subcaptionlabelfont{\rmfamily}
	\sidecapstyle{\raggedright}
	\setlength{\belowcaptionskip}{10pt}
	\ifx \@empty#1 \subcaption{#2}\else \subcaption[#1]{#2}
}
\usepackage{lettrine}
\newcommand{\initial}[1]{%
	\lettrine[lines=3,lhang=0.33,nindent=0em]{
		\color{gray}
     		{\textsc{#1}}}{}}
            
\usepackage{mdframed}

\theoremstyle{definition}
\newmdtheoremenv[backgroundcolor=gray!10]{theorem}{Theorem}
\AtBeginEnvironment{theorem}{\begin{minipage}{\textwidth}}
\AtEndEnvironment{theorem}{\end{minipage}}
\theoremstyle{definition}
\newmdtheoremenv[backgroundcolor=gray!10]{definition}[theorem]{Definition}
\AtBeginEnvironment{definition}{\begin{minipage}{\textwidth}}
\AtEndEnvironment{definition}{\end{minipage}}
\theoremstyle{definition}
\newmdtheoremenv[backgroundcolor=gray!10]{lemma}[theorem]{Lemma}
\AtBeginEnvironment{lemma}{\begin{minipage}{\textwidth}}
\AtEndEnvironment{lemma}{\end{minipage}}
\theoremstyle{definition}
\newmdtheoremenv[backgroundcolor=gray!10]{proposition}[theorem]{Proposition}
\AtBeginEnvironment{proposition}{\begin{minipage}{\textwidth}}
\AtEndEnvironment{proposition}{\end{minipage}}
\newmdtheoremenv[backgroundcolor=gray!10]{corollary}[theorem]{Corollary}
\AtBeginEnvironment{corollary}{\begin{minipage}{\textwidth}}
\AtEndEnvironment{corollary}{\end{minipage}}
\theoremstyle{definition}
\newmdtheoremenv[backgroundcolor=gray!0]{remark}{Remark}
\AtBeginEnvironment{remark}{\begin{minipage}{\textwidth}}
\AtEndEnvironment{remark}{\end{minipage}}

\usepackage{tikz}
\usepackage{pgfplots}
\usepackage[section]{placeins} % to place figures after mentioning
\usepackage{algpseudocode}
\usepackage{algorithm}
\usepackage{nomencl}
\makenomenclature
\nomenclature{}{\textbf{List of Symbols and Abbreviations}}

\nomenclature[A,01]{$\vu$}{The velocity field of the fluid, unless specified differently.}
\nomenclature[A,02]{$\vw$}{The vorticity field of the fluid, unless specified differently.}

\nomenclature[B,01]{$\Delta$}{The Laplace operator}
\nomenclature[B,02]{$\nabla$}{The nabla operator}
\nomenclature[B,03]{$\grado f$}{The gradient operator applied to a scalar function $f$}
\nomenclature[B,04]{$\cdot$}{Scalar vector multiplication operator}
\nomenclature[B,05]{$\diveo$}{The divergence operator}
\nomenclature[B,06]{$\times$}{Cross product operator}
\nomenclature[B,07]{$\curlo$}{The curl operator}

\nomenclature[C,00]{EC}{Exterior Calculus}
\nomenclature[C,01]{DEC}{Discrete Exterior Calculus}
\nomenclature[C,02]{DDG}{Discrete Differential Geometry}
\nomenclature[C,03]{VFX}{Visual Effects}
\nomenclature[C,03]{VEX}{The name of the main programming language used in Houdini}
\nomenclature[C,04]{VR}{Virtual Reality}

\newcommand{\R}{\mathbb{R}}
\newcommand{\N}{\mathbb{N}}

\newcommand{\Sp}{\mathbb{S}^2}
\newcommand{\w}{\omega}

\newcommand{\vw}{\vec{\w}}
\newcommand{\vu}{\vec{u}}
\newcommand{\vv}{\vec{v}}
\newcommand{\partio}[1]{\frac{\partial }{\partial #1}}
\newcommand{\parti}[2]{\frac{\partial #1}{\partial #2}}
\newcommand{\material}[1]{\frac{D}{D #1}}
\newcommand{\vx}{\vec{x}}
\newcommand{\vy}{\vec{y}}

\newcommand{\vp}{\vec{p}}

\newcommand{\vg}{\vec{g}}

\newcommand{\vn}{\vec{n}}

\newcommand{\vF}{\vec{F}}
\newcommand{\vnu}{\vec{\nu}}
\newcommand{\tf}{\tilde{f}}
\newcommand{\tg}{\tilde{g}}
\newcommand{\tsigma}{\tilde{\sigma}}
\newcommand{\tH}{\tilde{H}}
\newcommand{\tX}{\tilde{X}}
\newcommand{\tY}{\tilde{Y}}
\newcommand{\tl}{\tilde{l}}

\newcommand{\bigO}{\mathcal{O}}

\newcommand{\vpsi}{\vec{\psi}}
\newcommand{\ddv}[2]{\begin{pmatrix} #1 \\  #2  \end{pmatrix}}
\newcommand{\dddv}[3]{\begin{pmatrix} #1 \\  #2 \\ #3 \end{pmatrix}}
\newcommand{\crossprod}[6]{\begin{pmatrix} #2#6 - #3#5\\ #3#4-#1#6 \\ #1#5-#2#4 \end{pmatrix}}

\newcommand{\diveo}{\nabla\cdot}
\newcommand{\curl}{\operatorname{curl}}
\newcommand{\curlo}{\nabla\times}
\newcommand{\grad}{\operatorname{grad}}
\newcommand{\grado}{\nabla}
\newcommand{\sgrad}{\operatorname{sgrad}}
\newcommand{\tsgrad}{{\operatorname{sgrad^*}}}
\newcommand{\lap}{\Delta}

\newcommand{\acos}{\operatorname{acos}}
\newcommand{\dime}{\operatorname{dim}}
\newcommand{\Id}{\operatorname{Id}}

\newcommand{\thesis}{thesis {}}

\newcommand{\planarvd}{Planar Point Vortex Dynamics {}}
\newcommand{\sphericalvd}{Spherical Point Vortex Dynamics {}}
\newcommand{\generalvd}{Closed Surfaces Point Vortex Dynamics {}}

\begin{document}
% UoB guidlines:
%
% Preliminary pages
% 
% The five preliminary pages must be the Title Page, Abstract, Dedication
% and Acknowledgements, Author's Declaration and Table of Contents.
% These should be single-sided.
% 
% Table of contents, list of tables and illustrative material
% 
% The table of contents must list, with page numbers, all chapters,
 % sections and subsections, the list of references, bibliography, list of
% abbreviations and appendices. The list of tables and illustrations
% should follow the table of contents, listing with page numbers the
% tables, photographs, diagrams, etc., in the order in which they appear
% in the text.
%

\frontmatter
\pagenumbering{roman}

%
%
% File: Title.tex
% Author: V?ctor Bre?a-Medina
% Description: Contains the title page
%
% UoB guidelines:
% 
% At the top of the title page, within the margins, the dissertation should give the title and, if 
% necessary, sub-title and volume number. If the dissertation is in a language other than English, the 
% title must be given in that language and in English. The full name of the author should be in the 
% centre of the page. At the bottom centre should be the words ?A dissertation submitted to the 
% University of Bristol in accordance with the requirements for award of the degree of ? in the 
% Faculty of ...?, with the name of the school and month and year of submission. The word count of 
% the dissertation (text only) should be entered at the bottom right-hand side of the page.
%
%
\begin{titlingpage}
\begin{SingleSpace}
\calccentering{\unitlength} 
\begin{adjustwidth*}{\unitlength}{-\unitlength}
\vspace*{16mm}
\begin{center}
\rule[0.5ex]{\linewidth}{2pt}\vspace*{-\baselineskip}\vspace*{3.2pt}
\rule[0.5ex]{\linewidth}{1pt}\\[\baselineskip]
{\HUGE Point Vortex Dynamics\\[1ex] on Closed Surfaces}\\[6mm]
%{\Large \textit{from a Discrete Differential Geometry Perspective}}\\
\rule[0.5ex]{\linewidth}{1pt}\vspace*{-\baselineskip}\vspace{3.2pt}
\rule[0.5ex]{\linewidth}{2pt}\\
\vspace{6.5mm}
{\large By}\\
\vspace{6.5mm}
{\large\textsc{Marcel Padilla}}\\
{\scriptsize\textsc{Immatriculation number: 351206}}\\
\vspace{11mm}
\includegraphics[height=3cm]{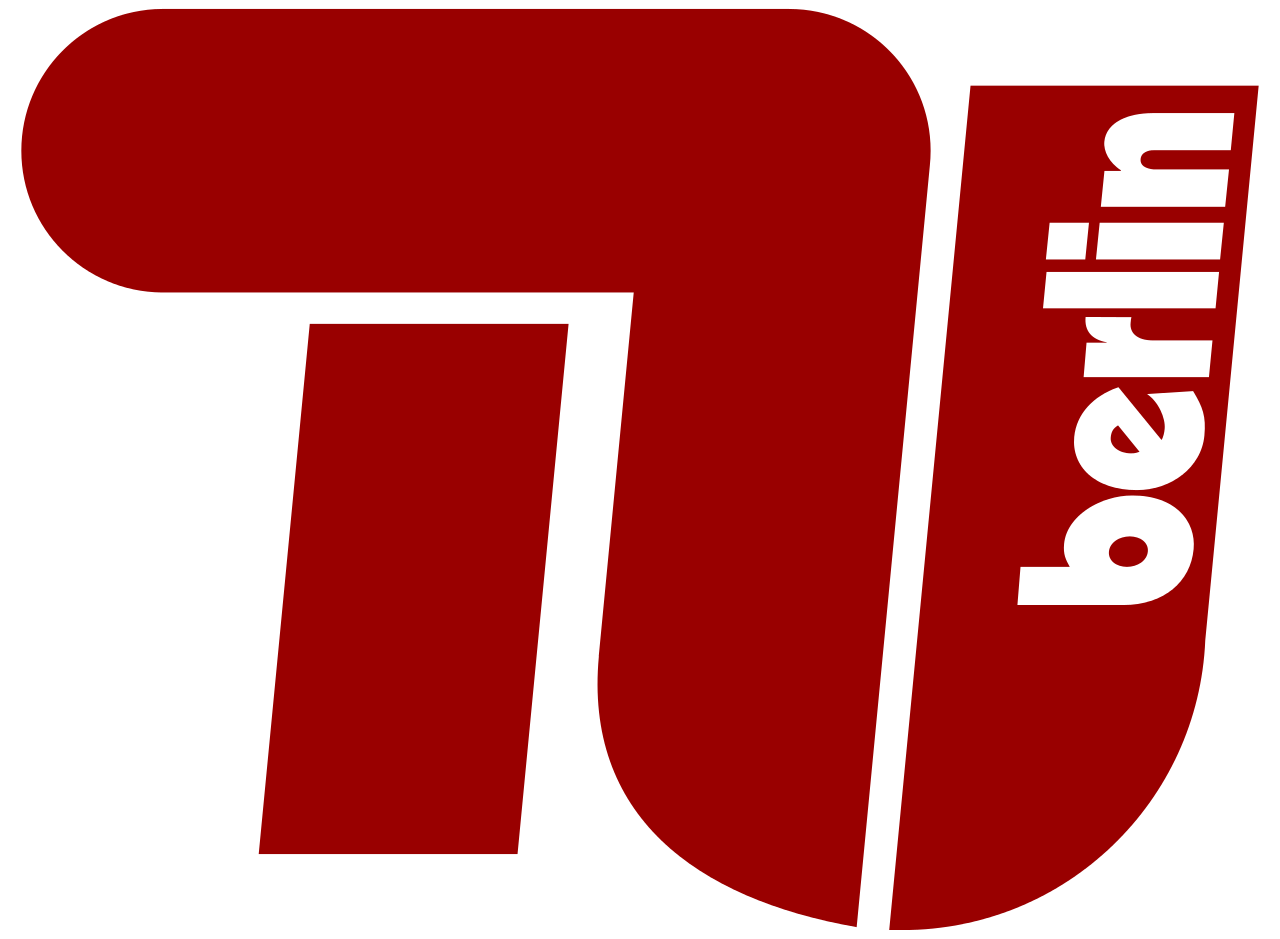}\includegraphics[height=3cm]{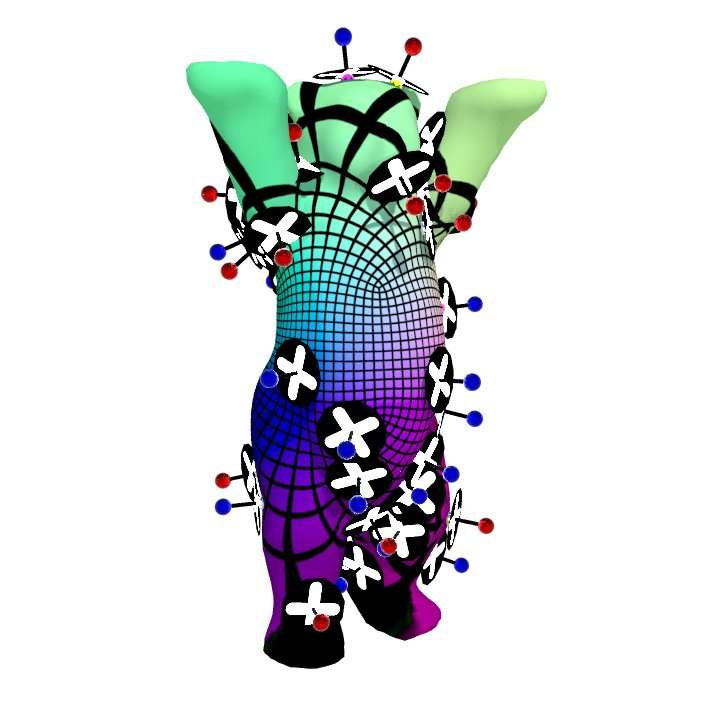}\\
\vspace{6mm}
{\large Department of Mathematics\\
\textsc{Technical University of Berlin}}\\
\vspace{11mm}
\begin{minipage}{10cm}
A dissertation submitted to the \emph{Technical University of Berlin} in accordance with the requirements of the degree of \textsc{Master of Science}.
\end{minipage}\\
\vspace{9mm}
{\large\textsc{APRIL 22, 2018}}
\vspace{5mm}
\end{center}
% \begin{flushright}
%{\small Word count: ten thousand and four}
% \end{flushright}

\vspace{3\baselineskip}%
\begin{center}
\begin{tabular}[h]{lr}%
Supervisor: & Prof. Dr. Ulrich Pinkall\\%
Secondary referee: & Prof. Dr. Yuri B. Suris\\%
Day of submission: & 02.05.2018%
\end{tabular}%
\end{center}
\clearpage%

\end{adjustwidth*}
\end{SingleSpace}
\end{titlingpage}
%\clearemptydoublepage
%
%
% File: abstract.tex
% Description: Contains the text for thesis abstract
%

\chapter*{Abstract}

\vfill

\begin{SingleSpace}
\textbf{English}

The theory of point vortex dynamics has existed since Kirchhoff's proposal in 1891 and is still under development with connections to many fields in mathematics. As a strong simplification of the concept of vorticity it excels in computational speed for vorticity based fluid simulations at the cost of accuracy. Recent finding by Stefanella Boatto and Jair Koiller allowed the extension of this theory on to closed surfaces. A comprehensive guide to point vortex dynamics on closed surfaces with genus zero and vanishing total vorticity is presented here. Additionally fundamental knowledge of fluid dynamics and surfaces are explained in a way to unify the theory of point vortex dynamics of the plane, the sphere and closed surfaces together with implementation details and supplement material.

\vfill
	
\textbf{Deutsch}

Die Theorie der Punkt Vortex Dynamik existiert seit Kirchhoff's Einführung in 1891 und wurde seitdem weiterentwickelt mit Verbindungen zu vielen Bereichen in der Mathematik. Als starke simplifikation von Vortizität überragt das Modell durch seine Rechengeschwindigkeit für Vortizität basierende Fluidsimulationen auf Kosten der Genauigkeit. Neue Entdeckungen von Stefanella Boatto und Jair Koiller erlauben die Erweiterung dieser Theorie auf geschlossene Oberflächen. Es wird eine verständliche Anleitung für Punkt Vortex Dynamik auf geschlossenen Oberflächen mit genus null und verschwindender Gesamtvortizität präsentiert. Zusätzlich werden Grundbegriffe von Fluiddynamik und Oberflächen erläutert auf eine Art die die Theorie der Punkt Vortex Dynamik für die Ebene, die Sphäre und geschlossene Flächen vereint zusammen mit Implementierungsdetails und Ergänzungsmaterial.

%\lipsum[1]

\end{SingleSpace}

\vfill

\clearpage
%\clearemptydoublepage

%\input{frontmatter/declaration}
% Declaration page removed for arXiv submission (signature pages cannot be signed digitally)
% The original declaration contained the "Eidesstattliche Erklärung" which requires a physical signature.
\clearpage
%\clearemptydoublepage
%

%
% file: dedication.tex
% description: Contains the text for thesis dedication
%

\chapter*{Dedication and Acknowledgements}
\begin{SingleSpace}
\vfill
\begin{center}
\emph{
Dedicated to\\
my Family for their continuous support,\\
my closest friends and my dear Natalia\\ for giving me so much joy.}

\end{center}
\vfill
\end{SingleSpace}
\clearpage
%\clearemptydoublepage

%\makeglossary

%==================================================================================================================================================================================================================================================================================================

% FRONT WORK

\renewcommand{\contentsname}{Table of Contents}
\maxtocdepth{subsection}
\tableofcontents*
\addtocontents{toc}{\par\nobreak \mbox{}\hfill{\bf Page}\par\nobreak}
\clearemptydoublepage
%\listoftables
%\addtocontents{lot}{\par\nobreak\textbf{{\scshape Table} \hfill Page}\par\nobreak}
%\clearemptydoublepage

%\listoffigures
%\addtocontents{lof}{\par\nobreak\textbf{{\scshape Figure} \hfill Page}\par\nobreak}
%\clearemptydoublepage

\printnomenclature
\clearemptydoublepage

%============================================================================================================================================================================================================================================================================================================================================================

% ACTUAL TEXT START

% The bulk of the document is delegated to these chapter files in
% subdirectories.
\mainmatter

\let\textcircled=\pgftextcircled

%+ + + + + + + + + + + + + + + + + + + + + + + + + + + + + + + + + + + + + + + + + + + + + + + + + + + + + + + + + + + + + + + + + + + + + + + + + + + + + + + + + + + + + + + + + + + + + + + + + + + + + + + + + + + + + + + + + + + + + + + + +    NEW CHAPTER     + + + + + + + + + + + + + + + + + + + + + + + + + + + + + + + + + + + + + + + + + + + + + + + + + + + + + + + + + + + + + + + + + + + +

\chapter{Introduction}
\label{chap:Introduction}

\initial{T}he main goal of this thesis is to document and derive rigorously the theory of point vortex motion on closed surfaces while also explaining it's implementation. Therefore we will establish all the details of fluid dynamics necessary for our model of surface fluids at the beginning and later introduce important concepts of discrete geometry. The final results will be a real time fluid point vortex motion algorithm that takes advantage of recent finding by Boatello \cite{Dritschel-2015} on point vortex dynamics together with recent results in discrete geometry processing algorithms to create conformal maps \cite{Kazhdan:2012:MFM:2346796.2346809}. As of the time when this thesis was written, no similar implementation has been published that presents such a vortex dynamics fluid simulation on closed surfaces this streamlined.

After a short recollection of the history (chapter \ref{chap:Introduction}) and the basics of fluid dynamics (chapter \ref{chap:Fluid Dynamics and the Vorticity Equation}) we will explain the concept of vorticity and its use in fluid simulations and how these concepts can be used on closed surfaces (chapter \ref{chap:Surface Fluids}). Next we will introduce the reduction of vorticity through point vortex dynamics (chapter \ref{chap:Point Vortices}) and work our way up from the planar case (chapter \ref{chap:\planarvd}), to the spherical case (chapter \ref{chap:\sphericalvd}) and later to the arbitrary genus zero closed surfaces case (chapter \ref{chap:\generalvd}). 

All results archived in this thesis will be presented in the supplement material implemented in the cinematic software called Houdini with full access to all the documented code with direct execution options. A comprehensible path of the implementation steps and tricks is also given in chapter \ref{chap:Implementation}.

%=============================================================================
\section{History of Fluid Dynamics}
\label{sec:History of Fluid Dynamics}

Fluid mechanics might be one of the oldest topics in mathematics that are still today not fully understood. The birth hour of this discipline came with Archimedes's investigation on buoyancy in his work \emph{Floating Bodies} around around 250 BC. Much later famous researches such as Leonardo da Vinci (1452 - 1519), Isaac Newton (1642 - 1726), Blaise Pascal (1623 - 1662), Daniel Bernoulli (1700 - 1782), Leonhard Euler (1707 - 1783), Lord Kelvin (1824 - 1907), Hermann von Helmholtz (1821 - 1894), Claude-Louis Navier (1785 - 1836) and George Gabriel Stokes (1819 - 1903) and Simeon Denis Poisson (1781 -  1840) would all further pioneer into related topics with many of them making contributions that are honored with their names in many theorems, equations and operators through this thesis. It is astonishing too see how a one mathematical discipline can receive such an incredible amount of attention and yet remain such a difficult problem to solve with many publications every month attempting to shed more light on possible solutions. By no means does this thesis claim to archived a major breakthrough, instead it is designed to shed a clear light on the sub topic of point vortex dynamics on closed surfaces in such a way that anyone can pick up the topic from here and have a guide on the simulation implementation as well.

\begin{figure}
\centering
\includegraphics[width=7cm]{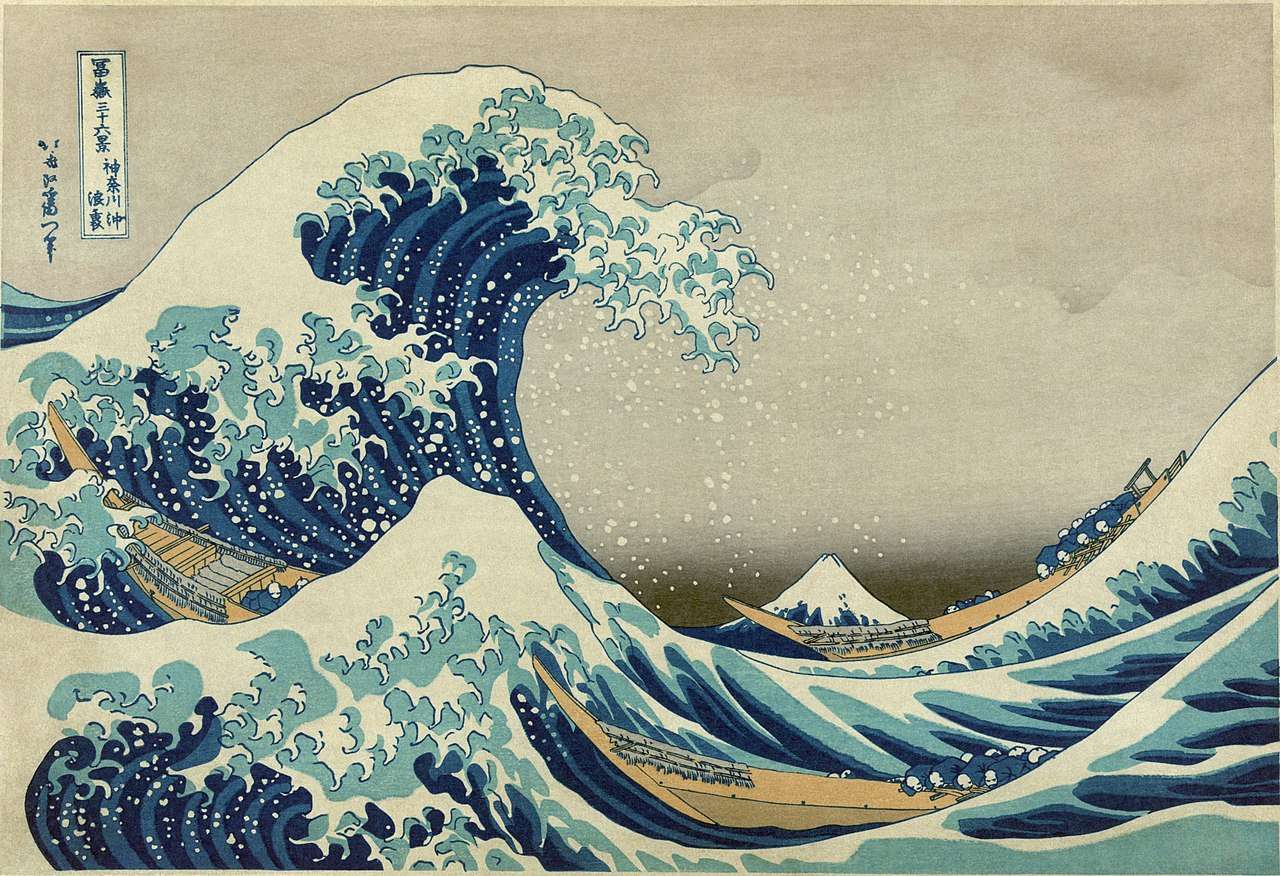}
\caption{\emph{The great wave of Kanagawa} by Katsushika Hokusai. Around 1829 in Japan. Public domain.}
\label{fig:Kanagawa}
\end{figure}

Perhaps fluid mechanics and dynamics are such a popular topic due to its every day relevance and thus familiar nature. The air we breath, the water we drink, the wind whistling past our ears, the shower splashing on our skin, the rain drops covering our paths with puddles or the great wave of Kanagawa 
(figure \ref{fig:Kanagawa}), those are all very familiar interactions with fluids that we all know about. We develop an intuition for the laws of motion that fluids follow but this intuition is often misleading and the closer we observe fluids the more complex they seem to become. This complexity becomes apparent in the almost unpredictable nature of fluid motion due to its high sensitivity to the initial conditions. Next time you place a teabag into a hot mug of water or when you pour in milk into your coffee take a close look on the motion of the colors and you may loose all hope in ever finding a method that can predict these complex patterns. It might take a god such as Poseidon to understand fluids thoroughly. %\figref{fig:Neptunbrunnen}

%removed image
%\begin{figure}
%\centering
%\includegraphics[width=7cm]{images/general/Neptunbrunnen}
%\caption{The mighty Poseidon in the \emph{Neptunbrunnen} build in 1891 in Berlin Mitte. Photo courtesy by Jorge Royan}
%\label{fig:Neptunbrunnen}
%\end{figure}

Nevertheless, Leonhard Euler's \emph{Principes generaux du mouvement des fluides} \cite{Euler1757} published in 1757 made a great leap forward by coining down the Euler equation (section \ref{sec:Euler Equations}  definition \ref{def:Euler equations}). This however is only a simplification on the Navier-Stokes equations \eqref{eq:Navier-Stokes equations 1} by Claude-Louis Navier and George Gabriel Stokes that we will explain with more detail in section \ref{sec:The Navier-Stokes Equations}. 

\begin{equation}
\label{eq:Navier-Stokes equations 1}
\partio{t}\vu + (\vu \cdot \grado)\vu+\dfrac{1}{\rho}\grado p = \vg +\nu\lap\vu
\end{equation}
\begin{equation*}
\diveo\vu =0
\end{equation*}

In essence, the study of fluid dynamics is the study of the Navier-Stokes equation. That is plenty of work for generations to come and solving the \emph{Navier-Stokes existence and smoothness problem}  could earn you a million US dollars by the Clay Mathematics Institute.

The development of mathematical tools together with many advances in computational mathematics give us plenty of new opportunities to handle fluid dynamics. Essential to this thesis will be Kirchhoff's description of vorticity through point vortex dynamics made back in 1891 \citep{kirchhoff1891vorlesungen}.

%with the use of the Laplace operator in the Poisson equation and using Helmholtz's decomposition theorem already gives us a set of tools that allow us to dissect the numerics of the fluid equations for the efficient use inside computers. %To view the historic advances made in the field of fluids methods, focused on the implementation inside computer, please refer to section \ref{sec:Common Fluid Methods}.

\begin{figure}
\centering
\includegraphics[width=7cm]{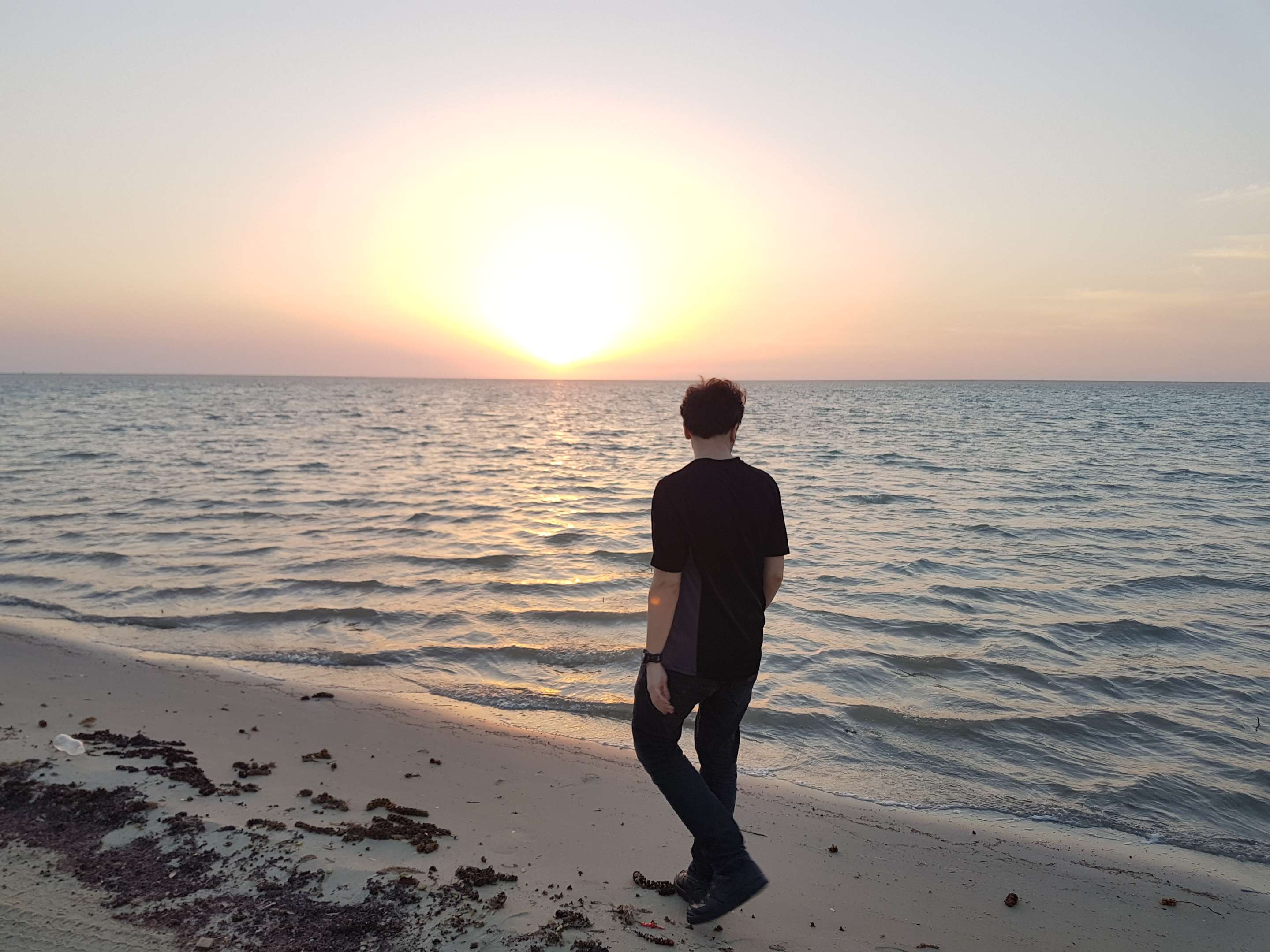}
\caption{Me trying to understand fluids by looking at the Red Sea.}
\label{fig:Redsea}
\end{figure}

\section{Modern Fluid Dynamics}

Luckily the age of computing has arrived and we can gain insights into fluid dynamics theories by displaying their results. The bad news is that for the purpose of fluid simulations the demands for memory and processing power are very high, thus dividing the theory into multiple applications. A rigorous scientific fluid solver may work really well but at an incredible cost in computation that makes it unusable for most activities but which nevertheless are essential when simulating tests for turbines or planes. When asking our selfs what fluid methods to use we must always weight the options according to their demands. An ocean surface for a movie will require more rigorous methods when letting waves splash against cliffs, but If the only thing the water needs to do is to shake a little with a boat then a surface based fluid will provide much higher quality results in the same amount of computation time. There is no scale from high fidelity to high performance in fluid methods. We can only hope to choose the right method for each demand. Section \ref{sec:Common Fluid Methods} of this thesis will list a few fluid methods. In a simplified way you can think of some multi-dimensional scale space where visual quality, low computational cost, and numerical stability compete with each other and there might not be a way to bring them all together as in figure \ref{fig:triangle of balance}.

\begin{figure}
\centering
\def\svgwidth{0.5\textwidth}
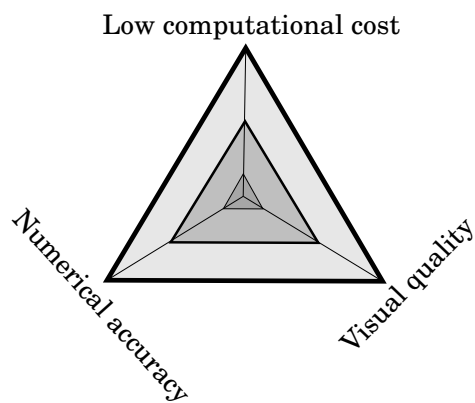
\caption{A simplification to illustrate the complexity of the demands made by fluid dynamics applications.}
\label{fig:triangle of balance}
\end{figure}

Yet sometimes we just want to have fluids and smoke for visual candy only and are ready to sacrifice realism such as in video games or animations. Especially in video games a fast performance is crucial and until up to this day many video games rely on 2d transparent texture based fake smoke that often clips through walls. Absolutely disgraceful (figure \ref{fig:fake smoke}). Of course, the games such as \emph{Tom Clancy's Rainbow Six Siege} aim to create a smoke screen that is easy to synchronized across the machines of each online player without dragging down too much performance. It is kind of funny how this unrealistic element still persists in a video game where the kinetic physics and destructibility are extremely impressive.

\begin{figure}
\centering
\includegraphics[width=7cm]{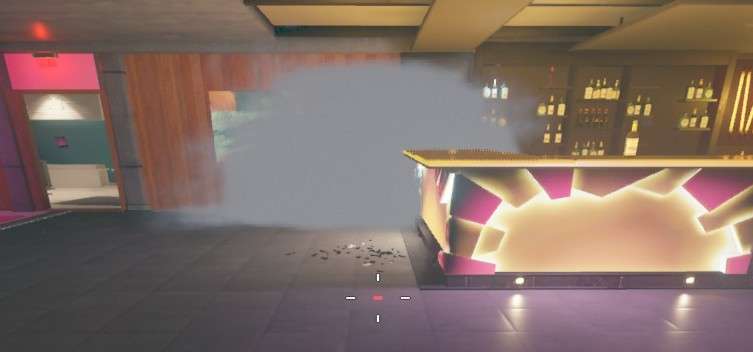}\\
\includegraphics[width=7cm]{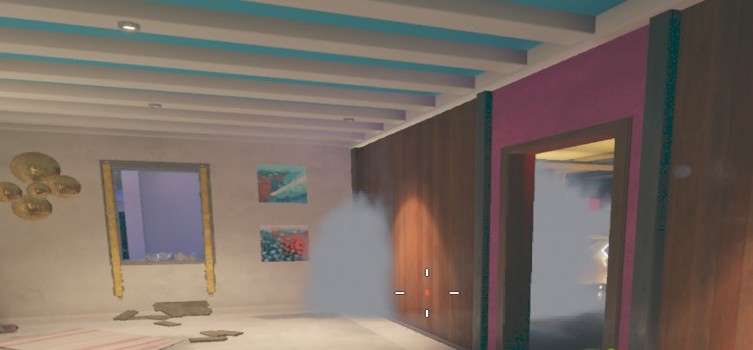}
\caption{The effects of a smoke grenade next to a wall in the game \emph{Rainbow Six Siege} seen from two perspectives. The smoke clearly penetrates through the surface.}
\label{fig:fake smoke}
\end{figure}

In common video games real time fluids have a hard time catching up with the real time kinetic physics that archived a mayor break through with the introduction of the source game engine by valve in 2004. Kinetic physics have been integrated into gameplay quite often and perhaps one day a functional smoke simulation will spice up the gameplay just as much. The current two dimensional smoke texture method won't hold up with the 3D demands especially with VR headsets. A real time smoke simulation of the future has to be able to run fast enough to please the high frame rates needed in VR to avoid motion sickness.

We mentioned the brief note of smoke in video games due to the assumptions that we are ready to make in this thesis in order to archive a fast performance with point vortex dynamics. In essence point vortices are a model to express the motion of a fluids by sampling the vorticity and only focusing on the whirls of the fluid and their motion. The aim of this thesis is to explain this model in detail and to implement it onto closed surfaces to then observe its value for visual effects. The highlighted question is if this way of simulating fluids still looks good.

%=============================================================================
\section{History of Vortex Dynamics}
\label{sec:History of Vortex Dynamics}

When seeing the fluid motion as a velocity vector field, vorticity describes the circulation of that said velocity field and is defined by the equation

$$\vw=\curlo\vu$$.

The properties of vorticity have gained increasing attention in the past century due to its nature to highlight turbulence where the most complex motion happens especially as they were revealed by simulations and more accurate experiments. It is precisely the area of high vorticity where the fluid twists the most that is the most difficult to describe and these areas have significant effects on the motion of any fluid as the fluid around circulations moves much faster. A detailed description of vorticity will be provided in section \ref{sec:Vorticity of Fluids}.

Kirchhoff's introduction of point vortices equation in a plane in his 1891 lecture \cite{kirchhoff1891vorlesungen} started what we will work on here. We will explain its details rigorously in chapter \ref{chap:Point Vortices}.

Especially the introduction of the vorticity equation as derived from the Navier-Stokes equations gives us reasons to believe that we can model fluids by their vorticity alone. This vorticity equation will be explained in section \ref{sec:The Vorticity Equation}. Jean-Baptiste Biot (1774 - 1862) and Felix Savart (1791 - 1841) introduced the Biot-Savart integral and thus paved the way to approximate the velocity field responsible for a given vorticity field which we will explain in section \ref{sec:Getting Velocity from the Vorticity}.

%Just like the velocity the field, the vorticity field is continuous too and also requires discretization for any applications. However, the (lucky) nature of vorticity makes discretization much more convenient as we will explain in section \ref{sec:Vorticity of Fluids}.

In 1992 P.G. Saffman published \emph{Vortex Dynamics} collecting the essentials of different approaches to describing vorticity \cite{saffman1992vortex} (in chapter. 2). The two most important descriptions are the one of vortex sheets,  and line vortices. Vortex sheets are an attempt to model the vorticity with a two dimensional sheets while line vortices concentrate the fluid vorticity inside a single curve. These methods and the additional usage of vortex tubes, filaments and points can be found in \cite{book:494471} and are still active areas of research. 

In 2007 Elcott et. al \cite{Elcott-2007} published a discrete exterior calculus (DEC) surface fluid model that preserves discrete circulation thus avoiding numerical diffusion of vorticity. This one is especially interesting as we are too taking advantage of DEC methods in our implementation. Azencot et. al, 2014 \cite{azencot2014functional} also published a DEC vorticity based method to simulate surface fluids on triangulated meshes. Their method is efficient, time-reversible and conserves vorticity while giving back natural phenomena. We do not challenge this method as it requires to track the vorticity on every static point of the mesh which we can avoid using selected point vortices that can be anywhere on the surface for the benefit of computation speed.

The topic of point vortex dynamics on closed surfaces today is pioneered by Stefanella Boatto, Jair Kolier and David Gerard Dritschel. In 2008 their publication of \emph{Vortices on closed surfaces} \cite{0802.4313} is a solid source to get into the topic and in 2015 a paper with the same name was published \cite{Boatto2015} which contains the essential theory that allows us to implement this thesis' point vortex dynamics and which has also been presented at \cite{Dritschel-2015}. In 2016 Boatto et. al moved on to establish the relations between N-point vortex dynamics and N-point gravity dynamics \cite{Boatto-2016}. We also note that in 2011 A. Regis published vortex dynamics for the triaxial ellipsoid \cite{RegisThesis} as his Ph.D. thesis.

\section{Supplement Material Software}
\label{sec:Supplement Material Software}

All results archived in this thesis will be presented in the supplement material implemented in Houdini with full access to all the documented code with instant execution options (after making sure scipy is installed). A comprehensible path of the implementation steps (figure \ref{fig:implementation eyecandy}) and tricks is also given in chapter \ref{chap:Implementation}. The files can be found on the following git site:

\begin{center}
\href{https://github.com/marcelpadilla/Point-Vortex-Dynamics-on-Closed-Surfaces.git}{\color{blue}https://github.com/marcelpadilla/Point-Vortex-Dynamics-on-Closed-Surfaces.git}.
\end{center}

%===============================================================
%
%add some attractive images of the houdini stuff.
\begin{figure}
\centering
\def\ll{0.24\textwidth}
\includegraphics[width=\ll]{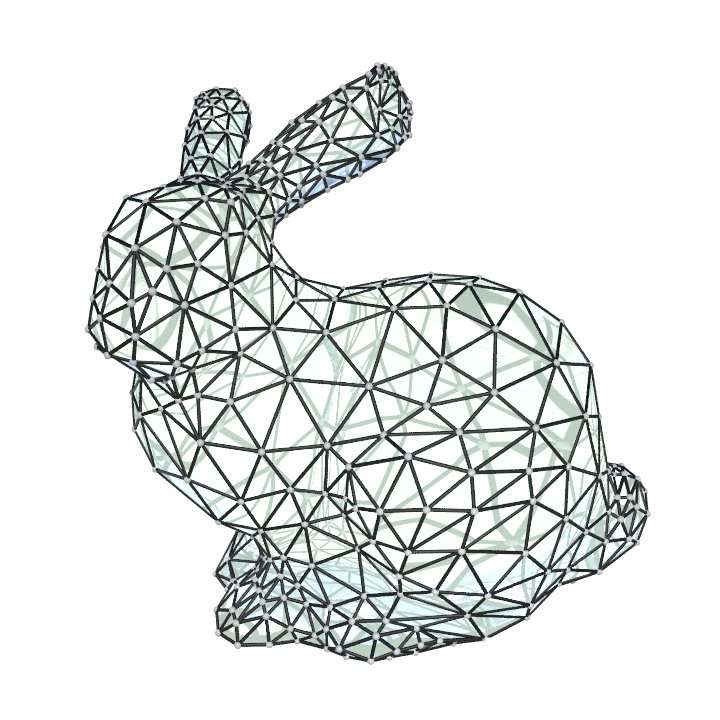}
\includegraphics[width=\ll]{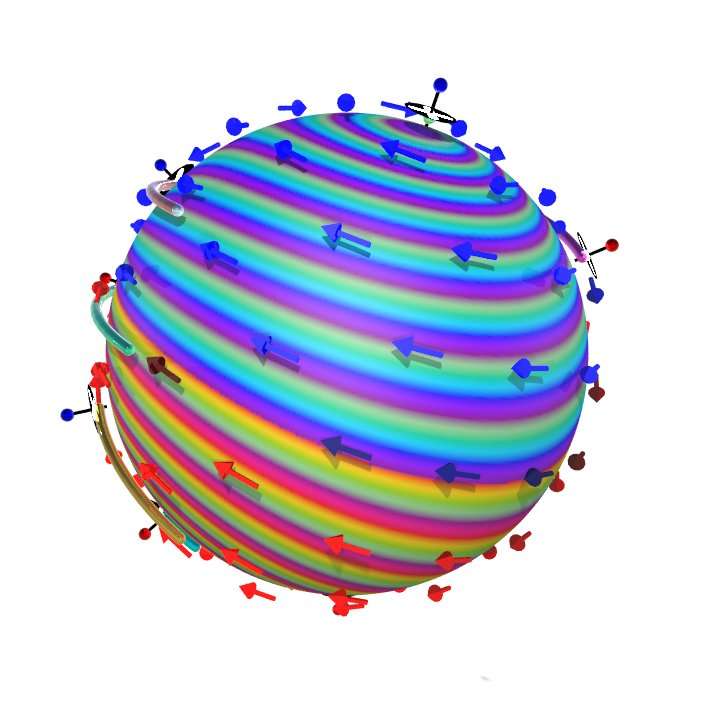}
\includegraphics[width=\ll]{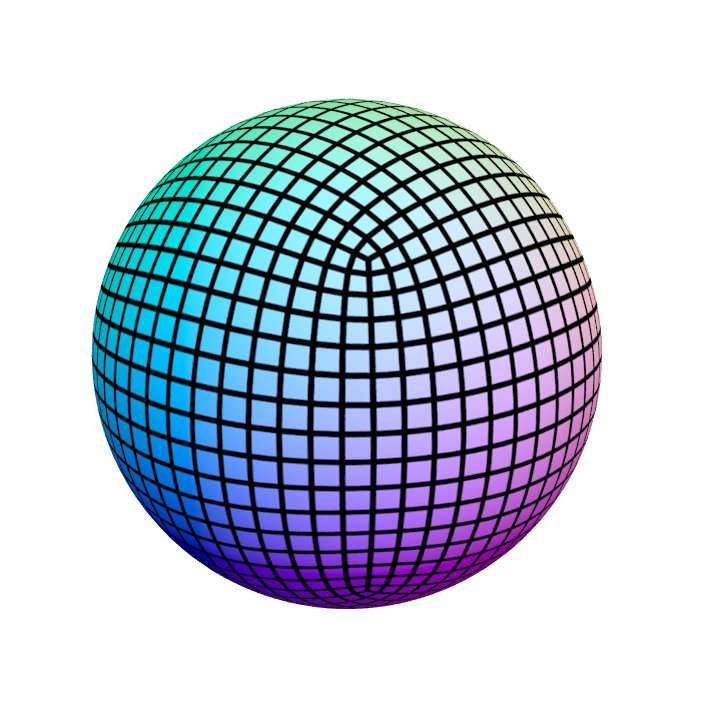}
\includegraphics[width=\ll]{images/houdini/bear_eyecandy}

\caption{Images created using the supplement material.}
\label{fig:implementation eyecandy}
\end{figure}

%=============================================================================
%\import{chapters/chapter01/}{chap01.tex}
%\clearemptydoublepage

%+ + + + + + + + + + + + + + + + + + + + + + + + + + + + + + + + + + + + + + + + + + + + + + + + + + + + + + + + + + + + + + + + + + + + + + + + + + + + + + + + + + + + + + + + + + + + + + + + + + + + + + + + + + + + + + + + + + + + + + + + +    NEW CHAPTER     + + + + + + + + + + + + + + + + + + + + + + + + + + + + + + + + + + + + + + + + + + + + + + + + + + + + + + + + + + + + + + + + + + + +

\chapter{Fluid Dynamics and the Vorticity Equation}
\label{chap:Fluid Dynamics and the Vorticity Equation}

%\initial{T}
This chapter will establish the fundamentals of fluid dynamics to create a solid mathematical playground for the reader to follow this thesis.

%=============================================================================
\section{The Navier-Stokes Equations}
\label{sec:The Navier-Stokes Equations}

You can't talk about fluid dynamics without it's number one differential equation that has supplied mathematicians with countless jobs all around the world. This is of course the incompressible Navier-Stokes equation that we will quickly define in here in its most general form

\begin{definition}[Incompressible Navier-Stokes equations]
\label{def:Navier-Stokes equations}
The flow of a fluid inside a given container is governed by the solutions to the following differential equations:
\begin{align*}
\partio{t}\vu + (\vu \cdot \grado)\vu+\dfrac{1}{\rho}\grado p &= \vg +\nu\lap\vu \\
\diveo\vu &= 0
\label{eq:Navier-Stokes equations}
\end{align*}

Here $t$ denotes time, $\vu$ the velocity field, $\rho$ the density field, $p$ the pressure field, $\vg$ the external forces such as gravity, $\nu$ the viscosity, $\grado$ the gradient operator, $\diveo$ the divergence operator and $\lap=\diveo\grado$ the Laplacian. Note that these are 4 differential equations at once. The first line is called the \emph{momentum equation} while the second one is called the \emph{incompressibility condition}.
\end{definition}

Mathematicians can't do everything (surprisingly) and solving these equations analytically is insane and numerically/computationally costly and difficult. Especially the incorporation of boundaries in the fluid will make everything much worse. However, boundary conditions are something we can avoid talking about since closed surfaces, the sphere and the plane have no boundaries.

Let us not forget to define what \emph{incompressible flow} actually means. Physically it describes the local property of density conversation in small parcels of fluid. Here comes the formal definition:
%We will then establish its connection to the 

\begin{definition}[Incompressible flow]
\label{def:Incompressible flow}
A velocity vector field $\vu$ with a density scalar field $\rho$ of a fluid is said to be \emph{incompressible} if the material derivative of $\rho$ vanishes, i.e.
\begin{equation}
\material{t}\rho=\partio{t}\rho + \vu\cdot\grado\rho = 0
\end{equation}
\end{definition}

%Why does the Navier-Stokes equation (definition \ref{def:Navier-Stokes equations}) look like this?

To get a feel for the Navier-Stokes equations we will have a quick look at their derivation. Surprisingly this only depends on basic laws of physics that we will quickly recall from \cite{bridson2015fluid}. Note that the main statement of this thesis does not need all of the details that we will now establish but it really helps to go through this in order to comprehend the assumptions that we make to reach our goal with the Euler equations (section \ref{sec:Euler Equations}).

% Forces derivation, momentum equation
Newton's countlessly cited first equation of motion will be the starting point of it all. Here we will now imagine a blob of fluid with mass $m$ and volume $V$ inside our container (figure \ref{fig:fluid blob}) and ask ourselfs: what forces act on that piece of fluid? For our blob, the change in momentum is proportional to the total force and our mass is constant which is why we can write:

\begin{equation}
m\material{t}\vu =\vF
\label{eq:newton first law}
\end{equation}

\begin{figure}
\label{fig:fluid blob}
\centering
\def\svgwidth{0.3\textwidth}
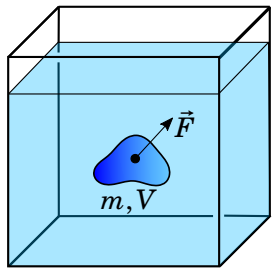
\caption{Let's image a piece of fluid with mass $m$ and volume $V$ inside a fluid. What forces $\vF$ apply to it? We ignore any boundary effects and pretend to live in an infinite water container.}
\end{figure}

How can we express $\vF$? The basic external force that we all submit to is \emph{gravity} $m\vg$ and we know that it should appear in $\vF$.

%(Definition \ref{def:Navier-Stokes equations})
Next we should think about \emph{pressure} $p$. Pressure is very non-trivial and will be explained later as a volume preserving force connected to the divergence free condition of the Navier-Stokes equation. Pressure is a scalar value that will consistently create a force to balance its values across the fluid. We know that the gradient of a scalar function points in the direction of the steepest ascent, which is why our pressure force must look something like $-\grado p$ in order to fight it's deviation from zero. This is proportional to the Volume size which is why we multiply gradient $\grado p$ with the blob volume $V$ to have our force $-V\grado p$.

Next want to mention the \emph{viscous force}. Viscosity describes the amount of friction within the fluid i.e. the fluid blobs have with each other. It can be seen as a force that attempts the balance out the velocities of one fluid blob to the velocities in its surroundings. In fact, it tries to minimize the differences in velocity similarly like the heat flow attempts to balance out the temperature in a solid. This way the mathematical expression involving viscosity comes with our beloved \emph{Laplacian} $\Delta$ that will receive a lot of attention in the later parts of this \thesis (chapter \ref{chap:Surface Fluids}, not in the context of viscosity). The viscous force also needs to be multiplied by the blob volume $V$ and by the dynamic viscosity coefficient $\mu$ that describes the amount of friction. The Resulting force is $V\mu\Delta\vu$.

%\begin{figure}
\begin{center}
  \begin{tabular}{ | l | c | c | c | }
   \hline
    \textbf{Fluid force:} & Gravity (external) & Pressure & Viscous \\ \hline
    \textbf{Expression:} & $+m\vg$ & $-V\grado p$ & $+V\mu\Delta\vu$ \\
    \hline
  \end{tabular}
\end{center}

%\caption{The collection of forces acting on the fluid blob of volume $V$.}
%\end{figure}

Collecting all the forces from the above table gives us:

\begin{equation}
m\material{t}\vu =m\vg - V\grado p + V\mu\Delta\vu
\end{equation}

Now we can finally get to jump from the blob case into the continuous case by dividing by the mass and taking the limit. This will introduce the \emph{density} $\rho$ by the limit $m/V \longrightarrow \rho$.

\begin{equation}
\material{t}\vu =\vg - \frac{1}{\rho}\grado p + \frac{\mu}{\rho}\Delta\vu
\end{equation}

Note that the material derivative can be written as

\begin{equation}
\material{t}\vu=\partio{t}\vu + (\vu \cdot \grado)\vu
\label{eq:material derivative}
\end{equation}

and that the kinematic viscosity $\nu$ equals $\frac{\mu}{\rho}$. Now we just need a tiny little rearranging work to get to the bigger one of the Navier-Stokes equations.

\begin{equation}
\partio{t}\vu + (\vu \cdot \grado)\vu+\dfrac{1}{\rho}\grado p = \vg +\nu\lap\vu
\end{equation}

% incompressibility condition 
But that is only half of the story and now we want to show the origin of the incompressibility condition $\diveo\vu=0$ and how it relates to the incompressible flow definition other than just by its name. As with many other derivations in physics we start out with properties of mass and momentum. The density function $\rho$  determines the mass given a volume and the $\rho\vu$ its momentum. The integral of these functions over the a sub domain $\Omega$ of the fluid results in the total mass and total momentum values in $\Omega$:

\begin{equation}
\label{eq:mass and momentum}
M=\iiint_{\Omega}\rho, \ \ \ \vec{P}=\iiint_{\Omega}\rho\vu
\end{equation}

It is important to see that the equations \eqref{eq:mass and momentum} compute the mass and momentum of any subset $\Omega$ of the entire fluid domain. Integrating on the whole fluid would give us constant values corresponding to the conservation of mass and momentum principles. On any subset $\Omega$ the amount of mass inside changes by the amount of fluid moving in or out of the boundaries of $\omega$. In mathematical terms this is expressed as:

\begin{equation}
\label{eq:derivative of mass}
 \partio{t} M = \partio{t}\iiint_{\Omega}\rho = \iiint_{\Omega}\partio{t}\rho = -\iint_{\partial\Omega}\rho\vu\cdot\vn
\end{equation}

where $\vn$ represents the outward pointing normal of the boundary $\partial\Omega$. Equations like this instantly trigger mathematicians to try out Stoke's theorem resulting in:

\begin{equation}
\label{eq:apply stokes theorem}
\iiint_{\Omega}\partio{t}\rho = -\iiint_{\Omega}\diveo(\rho\vu)  \ \ \Leftrightarrow \ \ 
\iiint_{\Omega}\partio{t}\rho + \diveo(\rho\vu) = 0
\end{equation}

Equation \eqref{eq:apply stokes theorem} has been established using $\Omega$ without placing any demands on the set. We conclude from this the equation

\begin{equation}
\label{eq:continuity equation}
\partio{t}\rho + \diveo(\rho\vu) = 0
\end{equation}

We refer to equation \eqref{eq:continuity equation} as the \emph{continuity equation}. This is the spot where we can smoothly smuggle in our definition of an incompressible flow, namely the disappearance of the material derivative of the density $\rho$, $\material{t}\rho=\partio{t}\rho + \vu\cdot\grado\rho = 0$. Subtracting this from equation \eqref{eq:continuity equation} leads to the incompressibility condition that we all know and love:

%http://www.continuummechanics.org/materialderivative.html Comes from uniform density maybe?. NO, from incompressible flow condition.

\begin{align*}
\label{eq:material derivative}
\partio{t}\rho + \diveo(\rho\vu) - \left( \partio{t}\rho + \vu\cdot\grado\rho \right) &= 0\\
\diveo(\rho\vu) - \vu\cdot\grado\rho &= 0\\
\vu\cdot\grado\rho + \rho\diveo\vu - \vu\cdot\grado\rho &= 0\\
\rho\diveo\vu &= 0\\
\end{align*}

Now we can just divide by the non-zero density to see the last bit that we needed for the full Navier-Stokes equation.

\begin{equation}
\label{eq:incompressibility equation}
\diveo\vu=0
\end{equation}

The incompressibility equation \eqref{eq:incompressibility equation} is a very important condition that many fluid simulations attempt to preserve. As seen in this derivation, a violation of the incompressibility leads to a change in mass of the fluid. Most of the time we will refer to this as the \emph{incompressibility condition} and $\vu$ is then called \emph{divergence free}. We should really honor it with its own corollary.

\begin{corollary}[Divergence free condition]
\label{cor:Divergence free condition}
A vector field $\vu$ of a fluid flow satisfying the incompressibility condition $\material{t}\rho=\partio{t}\rho + \vu\cdot\grado\rho = 0$ for it's density $\rho$ consequently has to fulfill the divergence free condition:
\begin{equation}
\diveo\vu=0
\end{equation}
If so, $\vu$ is called \emph{divergence free}.
\end{corollary}

Fluid methods that sacrifice accuracy of the incompressibility condition can be funny. Imagine stirring a resting fluid and suddenly having more fluid than before. We will explain more about the incompressibility assumption in section \ref{sec:Euler Equations}.

%=============================================================================
\section{The Euler Equations}
\label{sec:Euler Equations}

% sweet source of euler derivation https://math.stackexchange.com/questions/1446796/derivation-of-euler-equation-of-inviscid-flow

Let's have a closer look at what we call the Euler equation. It is less general than the Navier-Stokes equation but comes with assumptions that we too make to our fluid. In the nutshell, the Euler equation is just the Navier-stokes equation without viscosity which turns the second-order partial differential equation into a first-order partial differential equation. In this section we will explain the Euler equations and our assumptions on inviscidity, mass uniformity and incompressibility. The Euler equations are more important to us in this \thesis but having established the Navier-Stokes equation is needed to understand the assumptions we make on the fluid.% \emph{Principes generaux du mouvement des fluides} \cite{Euler1757} published in 1757 spawned the following differential equations:

% NOTE:
%THE true euler equation is simply navier stokes without viscocity
\begin{definition}[Incompressible Euler Equations]
\label{def:Euler equations}
The motion of an incompressible, inviscid flow with uniform density inside a container without external forces is determined by the following equations:
\begin{align*}
\label{eq:Euler equation}
\material{t} \vu  &= -\grado p \\
\diveo\vu &= 0
\end{align*}
Here $\vu$ denotes the velocity field and $p$ the pressure.
\end{definition}

Note that these are again four partial differential equations at once as the first line is expressed as a vector. We have already mentioned the importance of \emph{incompressibility} in definition \ref{def:Incompressible flow} on page \pageref{def:Incompressible flow}. 

No extra effort is needed to derive these equations since their derivation is equivalent to the one of the Navier-Stokes equation done in section \ref{sec:The Navier-Stokes Equations}. The \emph{divergence free} condition $\diveo\vu=0$ is the same as in the Navier-Stokes equations and has the exact same interpretations of mass preservation. Later we will see that the pressure force $\grado p$ is precisely the forced needed inside the fluid  in order to maintain it divergence free and thus preserve its volume. This last bit is a purely mathematical interpretation of pressure.

We have not yet discussed the relevance of \emph{incompressibility}. As the word suggests, it describes that a fluid is unable to compress it's density. This means that you can not squish a container with a certain fluid together. However, common knowledge of fluids state that gases can be compressed into fire extinguishers and that in the water and air sounds are transmitted through compression waves. Does this make working with incompressible fluids a pointless endeavor?

Of course not. The great news is that the effects of fluid compression is minimal in reality and have an negligible on the overall motion of the fluid. True, fluid compression is needed to simulate the acoustics but that is not what we are after and an extremely expensive goal to reach. Having incompressible fluids makes our life much easier and focuses our resources more efficiently for the fluid motion. Even if we are not only after visuals, dropping compressibility is a very common procedure.

And as a side note, water and air are both very difficult to compress. It would require very fast motion such as in blasts or strong pumps to get some notable effect compared to incompressible fluid flows. The study of fluids with these effects is called \emph{compressible flow}.  

Ok, now that we talked through \emph{incompressibility} we can now talk about our assumption of \emph{uniform constant density}. This again is related to the incompressibility we set up earlier. We assume our fluids, for example air and water, to enter our simulation without any biases in the initial compression, meaning that we have the same density everywhere and that we want that density to remain constant. For water this is naturally negligible anyway but for gases this is less simple. However, as we are dealing with small scale gases we can still drop this requirement.

Let us now talk about the assumption on our fluid to be \emph{invicid} next. What are the consequences of dropping viscosity? Luckily, water has pretty low viscosity and air even less. The effects however depend very much on the scale of the operation. The funny thing about viscosity in particular is that even if you remove it in theory from the fluid equations, many numerical errors can later be interpreted as a viscous effect. This however won't be the case with point vortices as we will see later.

%However, please note that without viscosity, i.e. with a frictionless fluid, no outside objects passing through them could receive any drag or lift. Airplanes would fall of the sky instantly and we would not be able to mix milk into a coffee using a spoon. NOT SURE IF THIS IS THAT TRUE

%=============================================================================
\section{Notes on Common Fluid Methods}
\label{sec:Common Fluid Methods}

The art of computational fluid mechanics is broad and diverse and all basically attempt to solve the same differential equations. We have briefly mentioned the history of fluid dynamics and will now mention some main computational techniques that have been developed since the invention of computers. %This section serves as a continuation of section \ref{sec:History of Fluid Dynamics}.

% https://en.wikipedia.org/wiki/Computational_fluid_dynamics

In 1969 P. McCormack developed a fluid method that is now named the McCormack method which is still popular as it is easy to understand and to implement \cite{mccormack2012physical} and still represented in recent literature for its ease to explain and for benchmarking \cite{bridson2015fluid}, \cite{Chern:2016:SS:2897824.2925868}. It is a second-order finite difference method that dissects the fluid container into a grid (figure \ref{fig:macgrid}).

\begin{figure}
\centering
\def\svgwidth{0.3\textwidth}
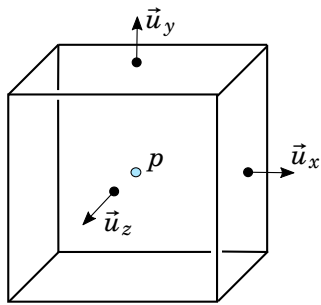
\caption{One cell of MacCormack's grid. Pressure is in the cell center while velocity is noted on the surfaces of neighboring cells.}
\label{fig:macgrid}
\end{figure}

Common fluid methods used today are collected inside a survey made in 2002 \cite{langtangen2002numerical} from Hans Petter Langtangen (1962 - 2016) et. al. Common numerical approaches in computational fluid dynamics typically discretize the fluid equations through \emph{Finite Volumes} (FVM), \emph{Finite Elements} (FEM), \emph{Finite Difference} (FDM) or \emph{Discrete Exterior Calculus} (DEC) \cite{Elcott-2007} methods to then compute the fluid or vortex flow directly. Please be advised not to mix up \emph{point vortices} with \emph{vortex particles} as the later one is for 3D fluid simulations \cite{Selle:2005:VPM:1073204.1073282}.

Through the above mentioned survey you will notice that the assumption made in this thesis, modeling vorticity as points, is a very strong (rough) assumption. The upside is that it discretizes vorticity rather than discretizing space itself as in the common methods. With many many point vortices you can approximate any vorticity fields. Apart from that point vortices have been named \emph{a classical mathematics playground} \cite{doi:10.1063/1.2425103} for linking a large number of areas of classical mathematics.

%=============================================================================
\section{Vorticity of Fluids}
\label{sec:Vorticity of Fluids}

%explain vorticity in detail with images if possible

Vorticity is express as the curl of the velocity field.

\begin{equation}
\label{eq:vorticity}
\vw = \curlo\vu
\end{equation}

But what does this mean? In this section we want to give a better understanding of the idea behind vorticity. This will help to better understand the rest of this thesis. We will illustrate vorticity and the \emph{curl} operator in 2- and 3-dimensions. Let us first clarify the basic definitions.

\begin{definition}[Cross Product]
\label{def:cross product}
Let $u$ and $v$ be vectors in $\R^3$. The \emph{cross product between $u$ and $v$} is defined as: 
\begin{equation}
\vu \times \vv = \dddv{u_x}{u_y}{u_z} \times \dddv{v_x}{v_y}{v_z} = \crossprod{u_x}{u_y}{u_z}{v_x}{v_y}{v_z}
\label{eq:cross product}
\end{equation}
\end{definition}

\begin{definition}[Curl]
\label{def:curl}
Let $u$ be a vector field on the domain $M\subset \R^3$. The \emph{curl of the vector field $u$} is defined as: 
\begin{equation}
\curl\vu= \curlo\vu = \dddv{\partio{x}}{\partio{y}}{\partio{z}}\times\dddv{u_x}{u_y}{u_z}{u_z} = \crossprod{\partio{x}}{\partio{y}}{\partio{z}}{u_x}{u_y}{u_z}
\label{eq:curl}
\end{equation}
\end{definition}

As we can see the curl's definition is tied to its dimension being three. In two dimensions the curl operator must be redefined. In order to reasonably define this we need to understand the curl operator first.
%Let's first try to understand what this curl thing is about.

%===========================================================

For the sake of a thought experiment let us imagine a 3D vector field $\vu=(u_x,u_y,u_z)$ where we assume that $u_z=0$, i.e. all motion remains in the x-y-planes (this just helps to visualize, the math won't need $u_z=0$). One such plane could for example be the flow of fluid on the surface of a lake. Let us walk to that lake on a beautiful day and throw a wooden log into the flow. Now we want to observe this log's behavior in the flow as shown in figure \ref{fig:log in the flow}.

\begin{figure}
\centering
\def\svgwidth{0.5\textwidth}
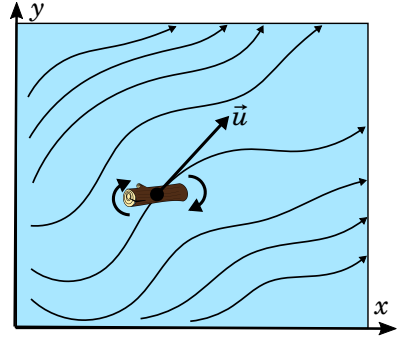
\caption{A infinitesimally small log flowing happily in a lake. The curved arrows show the flow lines of the lake surface. How will this log rotate?}
\label{fig:log in the flow}
\end{figure}

We know that the log will move along the flow lines, but that only captures the translational change. In real life our log will rotate as well and vorticity will precisely capture this rotation. To understand this we need to become one with the wooden log. From the frame of perspective of the log we let $\vnu\in\R^3$ represent the rotation vector. $\vnu$ exists for sure but we have not yet expressed it mathematically. When switching from the inertial frame velocity $\vu_I$ to the rotational frame velocity $\vu_R$ the following equation must hold:

%source https://www.quora.com/Fluid-Dynamics-Why-is-vorticity-equal-to-2x-the-angular-velocity-for-rigid-rotation

\begin{equation}
\vu_I = \vu_R + \vnu\times\vx
\end{equation}

Let's smash the curl operator against this system under the assumption that the wooden log is at rest on a spot in the flow i.e. $\vu_R=0, \ \grad \vnu = 0$ and use this lemma:

\begin{lemma}
For vector fields $a$ and $b$ we have
\begin{equation}
\curlo (a\times b) = (\diveo b)a+(b\cdot\grado)a - (\diveo a)b -(a\cdot\grado)b 
\end{equation}
\end{lemma}

This is what happens:

\begin{align*}
\curlo \vu_I &= \curlo (\vnu\times\vx)\\
&= (\diveo\vx)\vnu+(\vx\cdot\grado)\vnu - (\diveo\vnu)\vx -(\vnu\cdot\grado)\vx\\
&= (3)\vnu+ 0  - (0)\vx -(\nu_x\partio{x} + \nu_y\partio{y} + \nu_z\partio{z} )\vx\\
&=3\vnu -\vnu\\
&=2\vnu
\end{align*}

This shows that the curl operator has the ability to expose the rotation $\vnu$ of the object in the flow at a given spot. Not only are we given the axis of rotation, but also the rotational velocity can be extracted with this curl. Physically, this rotation occurs due to the different rate of changes of the velocity along each side of a flow line.

The transition to the 2D surface case with tangential velocity field is now very simple. The only possible axis of rotation is the surface normal $\vn$, thus we know that there must be some scalar $\w$ such that 

\begin{equation}
\label{eq:scalar vorticity vector relation}
\vw = \curlo\vu = 2\vnu= \w\vn
\end{equation}

Since $\langle \vn,\vn \rangle = 1$ we can now define the 2D curl using the above relation.

%Without loss of generality we can transform the vector field such that the $\vn$ is mapped to the z-axis. Then $z$ has no influence on $u_x,u_y$, $u_z$=0 and $\partio{z}u_x=\partio{z}u_y=0$.
%
%\begin{equation}
%\curl\vu = \dddv{\partio{x}}{\partio{y}}{\partio{z}}\times\dddv{u_x}{u_y}{0} = \crossprod{\partio{x}}{\partio{y}}{\partio{z}}{u_x}{u_y}{0} = \dddv{0}{0}{\parti{u_y}{x}-\parti{u_x}{y}}
%\label{eq:2d curl justification}
%\end{equation}

%Equation \eqref{eq:2d curl justification}'s last column vector suggests a way to define the 2D curl.

\begin{definition}[2D Curl]
\label{def:2d curl}
Let $\vu$ be a two dimensional vector field on a surface $M$ with normal vectors $\vn$. The \emph{curl of $\vu$} is defined as:
\begin{equation}
\curl\vu= \langle \vn , \vw \rangle = \langle \vn , \curlo \vu \rangle
\label{eq:2d curl}
\end{equation}
with the euclidean product $\langle \ , \ \rangle$.
\end{definition}

As you can see, the 2D curl from definition \ref{def:2d curl} is a scalar function, which might be a bit surprising given that the three dimensional curl returns a 3D vector. This is fine because in 2D there is only one possible axis of rotation whereas in 3D you posses three.

%Let us now use this 2D curl to attempt to understand vorticity. Equation \eqref{eq:2d curl} suggests that vorticity describes the rotation of the change in the fluid velocity. The change is described 

%\begin{equation}
%\ddmv{0}{-1}{1}{0}\ddv{\parti{u_x}{y}}{\parti{u_y}{x}}=\ddv{-\parti{u_y}{x}}{\parti{u_x}{y}}
%\text{if} \curl\vu=0 \Leftrightarrow \parti{u_y}{x}=\parti{u_x}{y}
%\end{equation}

The rotation is proportional to the scalar value of the vorticity and clockwise for positive values. Vector fields in 2D and 3D without vorticity, i.e. $\curl\vu\equiv 0$,  are called \emph{irrotational}. In the plane $\R^2$ embedded in $\R^3$ for example we have $\vn=(0,0,1)^T$ and thus

$$\curl \vu = \parti{u_y}{x} - \parti{u_x}{y}$$

Take a look at the very basic vector fields in figures \ref{fig:2d constant vorticity}, \ref{fig:2d zero vorticity} and \ref{fig:2d mixed vorticity} to familiarize yourself with them.

\begin{figure}
\centering
\def\svgwidth{0.5\textwidth}
\def \xfieldfunc{-y}
\def \yfieldfunc{x}

% leave the bottom part as untouched as possible. manipulate the function from above
\def\veclength{sqrt((\xfieldfunc)^2+(\yfieldfunc)^2)}
\begin{tikzpicture}
\begin{axis}[domain=-3:3, view={0}{90}]
\addplot3[
blue, quiver={
u={\xfieldfunc/(\veclength)},
v={\yfieldfunc/(\veclength)},
scale arrows=0.15},
-stealth,samples=11
] {0};
\end{axis}
\end{tikzpicture}

%this example works fine
\caption{The vector field $(-y,x)$ has vorticity $\partio{x}x - \partio{y}(-y)=2$ everywhere. An object flowing along this field would rotate with a constant angular velocity in this flow.}
\label{fig:2d constant vorticity}
\end{figure}
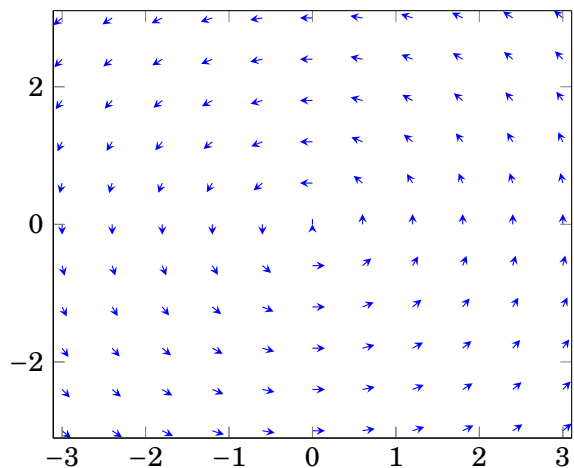

\begin{figure}
\centering
\def\svgwidth{0.5\textwidth}
\def \radiusfieldfunc{-0.5*(x^2+y^2)^(-3/2)}
\def \xfieldfunc{\radiusfieldfunc*2*x}
\def \yfieldfunc{\radiusfieldfunc*2*y}

% leave the bottom part as untouched as possible. manipulate the function from above
\def\veclength{sqrt((\xfieldfunc)^2+(\yfieldfunc)^2)}
\begin{tikzpicture}
\begin{axis}[domain=-3:3, view={0}{90}]
\addplot3[
blue, quiver={
u={\xfieldfunc/(\veclength)},
v={\yfieldfunc/(\veclength)},
scale arrows=0.15},
-stealth,samples=11
] {0};
\end{axis}
\end{tikzpicture}

%this example works fine
\caption{Here we see the gravitational potential gradient field $\grado f= \grado \frac{1}{r}$. Since $\curl\grado f\equiv 0$ any gradient field is irrotational. Objects in this field will not rotate.}
\label{fig:2d zero vorticity}
\end{figure}
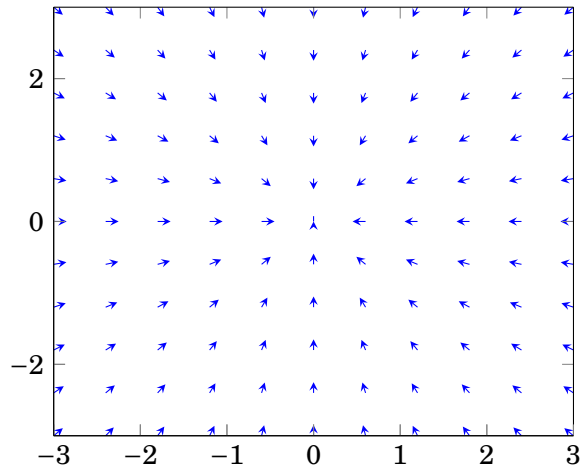

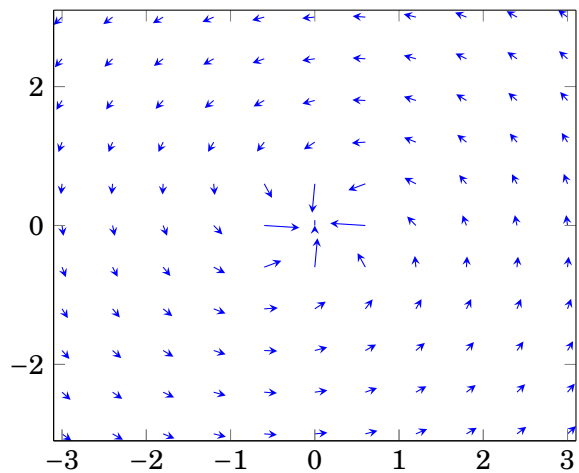
\begin{figure}
\centering
\def\svgwidth{0.5\textwidth}
\def \radiusfieldfunc{-0.5*(x^2+y^2)^(-3/2)}
\def \xfieldfunc{\radiusfieldfunc*2*x -y}
\def \yfieldfunc{\radiusfieldfunc*2*y +x}

% leave the bottom part as untouched as possible. manipulate the function from above
\def\veclength{sqrt((\xfieldfunc)^2+(\yfieldfunc)^2)}
\begin{tikzpicture}
\begin{axis}[domain=-3:3, view={0}{90}]
\addplot3[
blue, quiver={
u={\xfieldfunc/(\veclength)},
v={\yfieldfunc/(\veclength)},
scale arrows=0.15},
-stealth,samples=11
] {0};
\end{axis}
\end{tikzpicture}

%this example works fine
\caption{Here we see the sum of $\grado \frac{1}{r}$ and $(-y,x)$. The total vorticity is the sum of the separate vorticities and thus constant with $\w=2$.}
\label{fig:2d mixed vorticity}
\end{figure}

%In three dimension, now using the 3D $\curl$ operator as in definition \ref{def:curl}, the interpretations of the vorticity are similar. However, instead of always having the same axis of rotation pointing upward of the 2D plane the curl now defines the axis of this rotation in 3D space.

%If the center of mass of the log is at $(a,b)\in\R^3$ at $t=0$ then the infinitesimal motion during the time step $\delta t>0$ is given b % derivation of 3d curl: http://mathworld.wolfram.com/InfinitesimalRotation.html
%\begin{equation}
%\ddv{a}{b} + \ddv{\delta t \parti{x}{u_x} + \delta t \parti{y}{u_x}}{ \delta t \parti{x}{u_y} + \delta t \parti{y}{u_y}}
% \end{equation}

%=============================================================================
\section{The Vorticity Equation}
\label{sec:The Vorticity Equation}
% maybe call it vorticity transport equation

This \thesis is all about point vortices and thus vortex methods, but all the fluid equations we have seen so far are only about velocity fields and don't even mention our beloved vorticity $\vw$. The time has come to change this. We can play around with the differential equations until we end up with the \emph{vorticity equation} that focuses on $\vw$ instead of $\vu$. For this we will orient ourself according to \cite{pozrikidis2011introduction}.
%page 256 pozrikidis book

One useful lemma right now is:

\begin{lemma}
For vector fields $\vu$ and $\vw := \curlo\vu$ we have
\begin{equation}
\curlo( \ (\vu \cdot \grado)\vu \ )=(\vu\cdot\grado)\vw-(\vw\cdot\grado)\vu
\end{equation}
\end{lemma}

Right now we just want to drop the curl operator on the entire incompressible Euler equations and hope to catch a few useful terms involving vorticity. So the main part of the incompressible Euler equations goes like this:

\begin{align*}
\material{t} \vu=  \partio t \vu + (\vu \cdot \grado) \vu &= -\grado p 
\end{align*}

Remember that the curl operator acts on $x, \ y, \ z$ and not on time $t$. This is why we can interchange the curl operator with the time differentiation. If we unleash the curl operator this happens:

\begin{align*}
\curlo \left(\partio t \vu + \vu\cdot\grado \vu \right) &= \curlo  (-\grado p)\\
\curlo \partio t \vu + \curlo(\vu\cdot\grado \vu ) &= -\curlo  \grado p \\
\partio t \curlo \vu + \vu\cdot\grado\vw-\vw\cdot\grado\vu &= - 0\\
\partio t \vw + \vu\cdot\grado\vw-\vw\cdot\grado\vu &= 0 \\
\material{t} \vw -\vw\cdot\grado\vu &= 0 \\
\end{align*}

Which leaves the following important message to take home:

\begin{equation}
\label{eq:material derivative of vorticity}
\material{t} \vw = (\vw\cdot\grado)\vu
\end{equation}

%Equation \eqref{eq:material derivative of vorticity} essentially tells us how vorticity is transported with the flow of the fluid.

This equation can also be re-written using the \emph{Lie derivative} $\mathscr{L}_u\vw := \vu\cdot\grado\vw - \vw\cdot\grado\vu$ as

\begin{equation}
\partio t \vw + \vu\cdot\grado\vw - \vw\cdot\grado\vu = \partio t \vw + \mathscr{L}_u\vw =0
\end{equation}

According to \cite{Elcott-2007} this equation then states that \emph{vorticity is simply advected along the fluid flow}. This is an extremely important statement as it is fundamental to the justification of the point vortex dynamics. In essence we will use the point vortices to approximate vorticity to compute the velocity field and then use this velocity field to advect the point vortices (see figure \ref{fig:vorticity_advection}.

\begin{figure}
\centering
\def\svgwidth{0.7\textwidth}
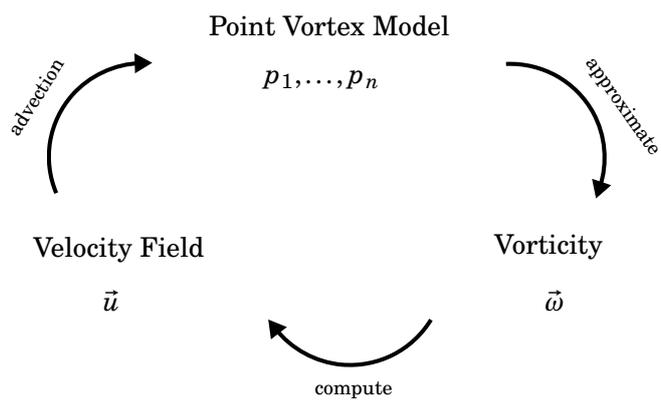
\caption{The path the point vortex model will take in the simulations.}
\label{fig:vorticity_advection}
\end{figure}

%For a 2D planar domain for example we have a constant normal pointing upwards. Since $\vw=\w\vn$ and $\vu$ and $\grado\vu$ are tied to the x-y plane we get
%For the 2D plane we have a constant normal vector $\vn$ pointing away from the surface. Since $\vw=\w\vn$ and $\vu$ is bound to the plane
%
%\begin{align*}
%\material{t} \vw &= (\w\vn\cdot\grado)\vu\\
%&= \w \left(\partio{z} \vu\right)\\
%&= 0
%\end{align*}
%
%
%meaning that in the planar-2D case vorticity is convected by the flow of $\vu$. When $\w(\vn\cdot\grado)\vu\neq0$ then \emph{vortex stretching} occurs.

%=============================================================================
%\import{chapters/chapter02/}{chap02.tex}
%\clearemptydoublepage

%+ + + + + + + + + + + + + + + + + + + + + + + + + + + + + + + + + + + + + + + + + + + + + + + + + + + + + + + + + + + + + + + + + + + + + + + + + + + + + + + + + + + + + + + + + + + + + + + + + + + + + + + + + + + + + + + + + + + + + + + + +    NEW CHAPTER     + + + + + + + + + + + + + + + + + + + + + + + + + + + + + + + + + + + + + + + + + + + + + + + + + + + + + + + + + + + + + + + + + + + +

\chapter{Surface Fluids}
\label{chap:Surface Fluids}

By now we established a basic understanding of fluids but there is still a necessity to be more specific about the term \emph{surface fluids}. What kind of surfaces are we exactly talking about and what basic knowledge of them do we need to recall to complete the purpose of this thesis? Apart from explaining how to constrain fluids to closed surfaces we will also clarify what the Green's functions and the stream functions are.

%=============================================================================
\section{Surface Fluids}
\label{sec:Surface Fluids}

As the name of this \thesis suggests, we only work with closed surfaces, in our final result meaning \textbf{genus zero} surfaces in $\R^3$ \textbf{without any boundary}. Most of the theory here still applies to surfaces of any genus without boundary but the final theorem works only on genus zero surfaces. Take a look at figure \ref{fig:closed_surfaces__many} first.

\begin{figure}
\centering
\def\svgwidth{0.8\textwidth}
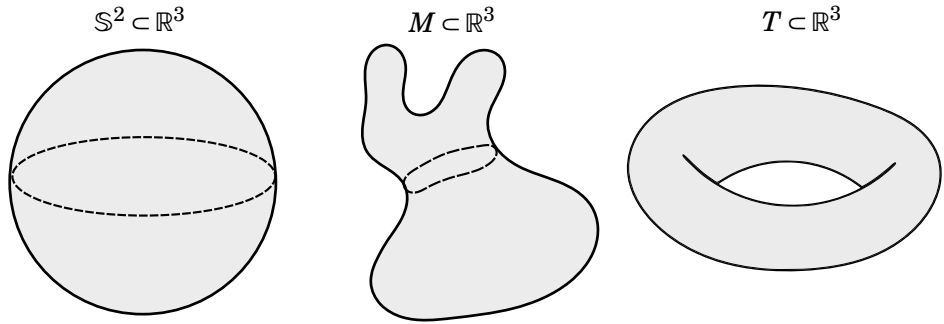
\caption{In our model, we will pay close attention to the sphere $\Sp$ (left) and on any arbitrary closed surface $M\subset\R^3$ of genus zero (center). The torus (right) has genus 1 and will thus not be used in the final result.}
\label{fig:closed_surfaces__many}
\end{figure}

\begin{definition}[Closed Surface in $\R^3$]
A smooth \emph{closed surface} is a 2-dimensional manifold in $\R^3$ without boundary. It is a set $M\subset\R^3$ such that for every point $p\in M$ there exists an $\epsilon>0$ such that the intersection of the $\epsilon$-sphere $B_{\epsilon}(p):=\{x\in\R^3:|x-p|<\epsilon\}$ around $p$ with $M$ can be smoothly mapped to an open set in $\R^2$.
\end{definition}

We only care about smooth surfaces here. Let's specify genus zero to highlight what we are working with in the final result.

\begin{definition}[Closed Surface in $\R^3$ with Genus Zero]
A \emph{closed surface} has genus zero if it can be smoothly deformed into a sphere. \textit{i.e.} it has no handles and no boundary.
\end{definition}

The boundarylessness is particularly important as boundaries can turn the relation between a mathematician and his fluid simulations into a sour one quickly. The effects a boundary has on the fluid motion is enormous unless the container is extremely huge. To respect these effects one would have to add a bunch of extra work into a solver. A taste of what this means for the planar point vortices in $\R^2$ can be found in \cite{pozrikidis2011introduction}( section 11.2.5 ). It involves ghost point vortices behind the boundary. Indeed scary.

Closed surfaces don't have boundaries and so we can skip a chapter that would just be about boundary conditions. Closed surfaces of genus zero by the way are the next best boundaryless thing apart from the infinite plane and periodic plane segments. An example of a surface fluid can be the atmosphere. Lets go more into detail and define the tangent planes.

% Sadly that system is much more complex, otherwise we would have much more reliable weather forecasts.

\begin{definition}[Tangent Space]
For a surface $M\subset\R^3$, $T_pM$ denotes the \emph{tangent space of $M$ at point $p\in M$}. It is the set of all derivatives at $p$ of smooth curves on the surface passing through $p$.
$$T_pM:=\{ \ \gamma ' (0) \ : \ \gamma : (-1,1) \ \longrightarrow \ M \ \  \text{smooth curve with } \  \gamma(0)=p \ \}$$
\end{definition}

Define $TM:=\{T_pM:p\in M\}$. Figure \ref{fig:tangent field} shows $T_pM$ in an example. It is also important to explicitly mention how our vector fields are affected by the transition to surfaces. We expect to not spill any fluid off the surface so it becomes essential that at any point $p$ on the surface $M$ the velocity vector $\vu_p$ must be in the tangent plane $T_pM$ of that point. 

\begin{figure}[H]
\centering
\def\svgwidth{0.5\textwidth}
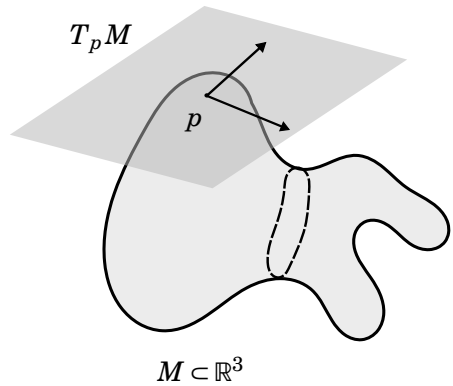
\caption{Tangent field $T_pM$ on a point $p\in M$}
\label{fig:tangent field}
\end{figure}

\begin{definition}[Tangent Vector fields on surfaces]
A vector field $\vu$ on a surface $M$ is \emph{tangential to $M$} if each vector lies is the tangent plane of its point.
%$$\vu : M \longrightarrow TM$$
$$p \mapsto \vu(p)\in T_pM$$

\end{definition}

Good news is that we do not have to redefine the divergence, gradient and curl for these surfaces. A surface fluid is now simply a tangent vector field $\vu$ on a surface $M$ that solves the Navier-Stokes equations on $M$. Once $\vu$ is in that tangent space, its flow will remain trapped in the tangent space forever. No escape. When ever we talk about $\vu$ being the velocity field of the fluid we will  always talk about a tangent vector field. The normal vectors $\vn:M\longrightarrow \R^3$ are an example of a non tangent vector field on $M$.

%Gather the most basic surface definitions.ink

%=============================================================================
\section{Green's Functions on Surfaces}
\label{sec:Greens Functions on Surfaces}

In section \ref{sec:Getting Velocity from the Vorticity} we will make use of Green's functions in order to get a velocity field $\vu$ with a given vorticity $\w$. Here we want to establish an exact definition and exact expressions for Green's functions for the two most important cases that we will cover in the thesis: for the plane and for the unit sphere $\Sp$. As a reminder:

\begin{definition}[Green's function]
\label{def:Green's function}
Let $\Delta$ be the Laplacian of the manifold $M$ acting on $x$ and $\delta$ the Dirac-delta function on $M$, then we call

$$G \ : \ M^2\setminus \{(x,x):x\in M\}\ \longrightarrow \ \R $$
$$ (x,y) \ \mapsto \  G(x,y)$$

the \emph{Green's function on $M$} if it solves the following equation:
\begin{equation}
\Delta_x G(x,y)=\delta(x-y)
\end{equation}
\end{definition}

There is not just one Green's function because whenever we pick a differently shaped $M$ the Laplacian changes too. In general it is difficult to compute Green's function for arbitrary shapes but we will see how much we can deduce by only knowing the Green's functions for the plane and for the sphere as take from \cite{Kimura245}.

\begin{proposition}[Green's function in the 2D plane]
\label{prop:Green's function in the 2D plane}
The Green's function in the 2D plane is given by
\begin{equation}
G(x,y)=-\frac{1}{2\pi}\ln(|x-y|)
\end{equation}
\end{proposition}

%\begin{proof}
%Without loss of generality $y=0$.
%\begin{align*}
%\lap_x G(x,0) &= \diveo_x \grado_x -\frac{1}{2\pi}\ln(|x|) \\
%&= -\frac{1}{2\pi} \diveo_x \grado_x \ln(|x|) \\
%\end{align*}
%\end{proof}

\begin{proposition}[Green's function on $\Sp$]
\label{prop:Green's function on the sphere}
The Green's function on $\Sp$ is given by
\begin{equation}
G(x,y)=-\frac{1}{2\pi}\ln \left(\ \sin\left(\frac{1}{2}d(x,y) \right)\ \right )
\end{equation}
where $d(x,y)=\acos(x\cdot y)$ is the euclidean distance between $x$ and $y$ on $\Sp$.
\end{proposition}

To verify the correctness of the above equation one only has to make a few integrations around small neighborhoods to see if $\lap G(x,y)$ behaves just like $\delta$ does.

%For general closed surfaces of genus 0 we will later see that the vortex dynamics will rely on the spherical 

%=============================================================================
\section{Stream Functions}
\label{sec:Stream Functions}

Knowing the stream function of your velocity field makes describing the flow of the fluid much easier. We introduce stream functions as they are one of the many great things we can compute from the point vortex model. Finding stream functions is usually a hard endeavor but we will later reveal how the point vortex model creates a trivial link between Green's functions and stream functions. %Normally these stream functions are very hard to obtain but we will later see that the point vortices model is able to compute the stream function.

\begin{definition}[3D Stream Function]
\label{def:3dstream}
For a fluid flow velocity field $\vu$ the vector field $\vpsi : \ M \ \longrightarrow \ \R^3$ is called the \emph{stream function} if:
\begin{equation}
\label{eq:3dstream}
\vu = \curlo \vpsi
\end{equation}
\end{definition}

For 2D the stream function is a scalar potential and thus
\begin{definition}[2D Stream Function]
\label{def:2dstream}
For a fluid flow velocity field $\vu$ on the euclidean surface $M \subset \R^3$ the scalar function $\psi : \ M \ \longrightarrow \ \R$ is called the \emph{stream function} if:
\begin{equation}
\label{eq:2dstream}
\vu = \vn \times \grado\psi
\end{equation}
where $\vn$ is the surface normal.
\end{definition}

This means that solving for the stream function of the flow also gives us the velocity field $\vu$ and thus solves the fluid dynamic problem just as well. Equation \eqref{eq:2dstream} is a simplification of the vector potential version since in 2D we only have one orientation to care about. Note that stream functions come with inconsistent sign conventions and that the 3D and 2D stream function definition actually have opposite sign when trying to link them. More about this later in section \ref{sec:Computing the Stream Function in 2D} remark \ref{remark:sign convention}.

In the 2D case, the gradient $\grado\psi$ points towards the direction of greatest ascend and applying a 90 degree rotation to that direction in the 2D setting using $\vn\times (\ )$ means that $\vu$ follows the lines where $\psi$ remains constant. In other words: a passive points following the vector field $\vu$ will move along the level lines of $\psi$. Figure \ref{fig:stream function lines} illustrates the stream lines in an example vector field. Such a neat interpretation is not possible for 3D stream functions. Section \ref{sec:Symplectic Manifolds and Symplectic Gradients} \emph{Symplectic Manifolds and Symplectic Gradients} will describe this in much more detail.

\begin{figure}
\centering
\def\svgwidth{0.5\textwidth}
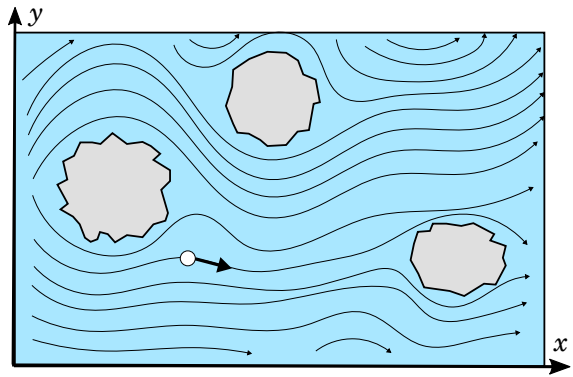
\caption{The level lines of the stream function of a fluid flow section with objects. A white table tennis ball dropped into the stream moves along one of these lines if the flow is constant.}
\label{fig:stream function lines}
\end{figure}

%=============================================================================
\section{Getting Velocity from Vorticity}
\label{sec:Getting Velocity from the Vorticity}

How exactly can we estimate the velocity from the vorticity?. The following section will explain how to find a corresponding $\vu$ given a $\vw$ and will play a key role in the later derivation of the results of this thesis to move the point vortices. Vorticity and velocity are related by the following equation

$$\vw = \curlo \vu$$\label{eq:fd vorticity}

It is not so obvious that given any vorticity field $\vw$ that we can find a corresponding $\vu$ such that this relation holds true. Even worse, a common differential identity is that $\curlo \grado f = 0$ for any smooth scalar field $f$, consequently meaning that given one solution $\vu$ to equation \ref{eq:fd vorticity} we can add a gradient $\grado f$ of any scalar field $f$ and obtain a new solution.

\begin{align*}
\ f \ : M \longrightarrow \R \ \ , \ \curlo \vu = \vw \ \ \  \\ \Rightarrow \ \ \ \curlo (\vu +\grado f) = \curlo\vu + \overbrace{\curlo\grado f}^{=0} = \vw
\end{align*}

If we want to gain serious results from fluids dynamics based on the vorticity we obviously need to come up with a good answer on how to gain a sensible velocity from the vorticity that is unique. For this we need to derive the existence of a solution first and then we need to eliminate the non-uniqueness of the velocity field.

One way to find a suitable velocity field $\vu$ given a vorticity field $\vw$ is through searching for a minimizer of the difference.

$$ \arg\min_{\vu}\int\int\int||\curlo \vu -\vw||^2$$

Let us try to solve this like an energy minimization problem. Lets take $\vp$ to be a vector field that decays to zero the further away you go from the origin and $s$ a parameter that controls the amount of influence $\vp$ has on the output. The energy of $\vu + s\vp$ then is 

$$g(s):=\int\int\int||\curlo (\vu + s\vp)-\vw||^2$$

If we assume a given choice of $\vu$ to fulfill $\curlo\vu=\vw$ then it would have to be a minimizer of this energy at $s = 0$, i.e. $g'(s)=0$. This argument can be used to start deriving conditions that $\vu$ must complete. Let us first write the $g$ into a more useful form.

\begin{align*}
g(s) &= \iiint(\curlo(\vu+s\vp)-\vw)\cdot(\curlo(\vu+s\vp)-\vw)\\
&= \iiint \ ||\curlo\vu-\vw||^2 + 2(\curlo\vu-\vw)\cdot(\curlo\vp)s + ||\curlo\vp||^2s^2
\end{align*}

As we can see, the god of mathematics blessed us with a simple quadratic function. For this quadratic the minima condition for $s=0$ s.t. $g'(s)=0$ is fulfilled if

\begin{align*}
0 &= \ \ \ \ \iiint 2(\curlo \vu -\vw)\cdot(\curlo\vp)\\
&= -2\iiint\curlo(\curlo\vu-\vw)\cdot\vp=0 
\end{align*}

Where in the last bit we used integration by parts with the curl operators.

\begin{lemma}[Curl Integration by Parts]
\label{lemma:Curl integration by parts}
For vector fields $\vx,\vy$ on a manifold without boundary:
\begin{equation*}
\int \vx \cdot \curlo\vy = - \int \curlo \vx \ \cdot \vy
\end{equation*}
\end{lemma}

%This leads us to to the minimization condition

If $\vu$ fulfills $\curlo\vu=\vw$ then $\iiint\curlo(\curlo\vu-\vw)\cdot\vp=0$ is true for any function $\vp$ (that decays to zero away from the origin). Thus we can conclude that the following expression has to equal zero

\begin{equation}
\curlo(\curlo\vu-\vw)=0 \Leftrightarrow \curlo\curlo\vu=\curlo\vw
\end{equation}

This relation seems rather obvious but remember that we had wished to start the derivation using a minimization of an energy. We now know that solving $\curlo\curlo\vu=\curlo\vw$ gives us a solution. Another differential identity states that

\begin{lemma}[Double Curl]
\label{lemma:double curl}
For any smooth vector field $\vu$
\begin{equation}
\curlo\curlo\vu= -\diveo\grado\vu+\grado(\diveo\vu)
\end{equation}
\end{lemma}

Thus we established that for $\curlo\vu=\vw$ to be true we require the condition

\begin{equation}
\label{eq:apply divergence free}
-\diveo\grado\vu+\grado(\diveo\vu)=\curlo\vw
\end{equation}

This is where the magic happens. Remember our incompressibility condition (definition \ref{def:Incompressible flow})? In a nutshell it is the volume preservation condition that turns out to be just the divergence-free condition as we have derived earlier in section \ref{sec:The Navier-Stokes Equations}. If our solution $\vu$ derived from the given vorticity $\vw$ is supposed to model an incompressible fluid we must require it also to solve the incompressibility equation $\diveo\vu=0$. This means that right now we can also use $\diveo\vu=0$ in our derivation of a solution $\vu$, however we must not forget to later check if this condition is actually true with the solution we end up with (done in section \ref{sec:Computing the Stream Function}). Using this simplifies our equation to

\begin{equation}
\diveo\grado\vu = \lap\vu=-\curlo\vw
\end{equation}

Surprise surprise! The Poisson equation sneaked into our problem! We welcome this very much as there is a lot of theory related to solving Poisson equations. One specific way of solving the Poisson equation involves the use of the Green's functions and this one will be the most helpful in our context.

The Green's function is defined section \ref{def:Green's function} such that:

\begin{equation}
G :M^2\setminus\{(x,x):x\in M\} \rightarrow \R, \ \ (x,y)\mapsto G(x,y), \ \ \lap_x G(x,y)=\delta(x-y)
\end{equation}

This allows us to express solutions of the Poisson equation in terms of combinations of functions $G(x,p)$ with as many $p$ from the surface $M$ as we need.

\begin{proposition}[Biot-Savart Law]
\label{prop:Biot-Savart law}
Given a vorticity $\vw$, a solution to the equation $\lap\vu=-\curlo\vw$ can be given as
\begin{equation}\label{eq:Biot-Savart law}
\vu(x)=  \int_M \grado_p G(x,p)  \times \vw(p) \ dp
\end{equation}
\end{proposition}

\begin{proof}

First we apply basic integration by parts to the right hand side to replace %(lemma \ref{lemma:Curl integration by parts})

\begin{equation*}
\grado_p G(x,p) \times \vw(p)  = - G(x,p)\ \grado_p  \times \vw(p) 
\end{equation*}

Now let us just verify the solution by smacking $\lap_x$ to it.
\begin{align*}
\lap_x\vu(x) &=   \lap_x \int_M \grado_p G(x,p) \times \vw(p) \ dp \\
&=  \lap_x - \int_M G(x,p)\ \grado_p  \times \vw(p) \ dp\\
&= -\int_M \lap_x G(x,p)\ \grado_p  \times \vw(p) \ dp\\
&= -\int_M \delta(x-p)\ \grado_p  \times \vw(p) \ dp\\
&= -\grado  \times \vw(x)
\end{align*}
\end{proof}

Thus we have shown a method to compute the velocity through vorticity using the so called \emph{Biot-Savart law}\citep{bridson2015fluid}. The demand for the divergence free condition removed some of the non-uniqueness. For a scalar function $f$

$$
\curlo(\vu + \grado f) = \vw
$$
$$
\diveo(\vu + \grado f) = 0 \Longrightarrow \lap f = 0
$$

Thus only harmonic functions $f$ (definition: $\lap f=0$) remain as a degree of freedom for velocity fields. In the the next section however we will express $\vu=\curlo \vpsi$ in terms of the stream function only.%, thus fixing all degrees of freedom as a consequence of the Helmholtz-Hodge decomposition. 

%=============================================================================
\section{Computing the Stream Function}
\label{sec:Computing the Stream Function}

Section \ref{sec:Getting Velocity from the Vorticity} provided us with enough tricks to construct the stream function for this fluid. Doing so will also allow us to prove that the velocity field $\vu$ established by the Biot-Savart law (proposition \ref{prop:Biot-Savart law}) is also divergence free.

\begin{proposition}[Stream Function for $\vu$]
\label{prop:stream function}
Given vector field $\vu(x)=  \int_M \grado_p G(x,p)  \times \vw(p) \ dp$, the following is a solution to the equation $\vu = \curlo \vpsi$ and thus a stream function:
\begin{equation*}
\vpsi \ : M \ \longrightarrow \ \R^3 \\
\end{equation*}
\begin{equation*}
\vpsi(x) := \int_M G(x,p)\ \vw(p)\ dp
\end{equation*}
\end{proposition}

\begin{proof}
\cite{bridson2015fluid} By construction, Green's functions are always symmetric $G(x,p)=G(p,x)$. Thus we also have 

\begin{equation*}
\grado_p G(x,p) = \grado_x G(p,x)
\end{equation*}

We use this together with integration by parts (lemma \ref{lemma:Curl integration by parts}) to apply the following magic:

\begin{align*}
\vu(x) &=  \int_M \grado_p ( G(x,p) )  \times \vw(p) \ dp\\
&=  \int_M \grado_x G(x,p)  \times \vw(p) \ dp\\
&=  \int_M \grado_x \times \left ( G(x,p) \vw(p) \right )  dp\\
&= \grado_x \times \underbrace{ \left (  \int_M   G(x,p) \vw(p) \ dp \right )}_{\vpsi:=}
\end{align*}

The last equation matches the stream function definition $\vu = \curlo \vpsi$ if we define the integral term to be our stream function.
\end{proof}

Ok! Don't forget that in equation \eqref{eq:apply divergence free} we applied the divergence free condition to reach the Poisson problem. We must now check that our solution for $\vu$ is actually divergent free. Thanks to the stream function $\vpsi$ this is easy to see since the divergence of the curl is always zero.

\begin{equation}
\diveo \vu = \diveo \ \curlo \vpsi = 0%\dive \curl \vpsi = 0
\end{equation}

This concludes that with the knowledge of $\vpsi$ we can easily compute a (not unique) divergence free velocity field $\vu$ given a specified vorticity field $\vw$. However, it is hard to find an expression for the function $\vpsi$ because it very much depends on on the Green's function $G$ and the shape of the domain $M$. We still need to solve the Poisson problem $\lap G=\delta$ first so our fight is not over yet.

Before we finish this section we want to point out that there is still another interesting relation to observe.

\begin{theorem}
For a fluid with vorticity $\vw$ and stream function $\vpsi=\int_M G(x,p)\ \vw(p)\ dp$ we have

$$\vw=-\lap\vpsi$$

Thus $\vpsi$ and $\vw$ also related as a Poisson problem.
\end{theorem}

\begin{proof}

Notice first that the divergence of $\vpsi$ vanishes:

\begin{align*}
\nabla_x\cdot\vpsi(x) &= \nabla_x\cdot \int_M G(x,p)\ \vw(p)\ dp\\
&=  \int_M (\nabla_x\cdot G(x,p))\ \vw(p)\ dp\\
&=  \int_M (\nabla_p\cdot G(x,p))\ \vw(p)\ dp\\
&=  -\int_M  G(x,p)\ \nabla_p\cdot\vw(p)\ dp\\
&=  -\int_M  G(x,p)\ \underbrace{\nabla_p\cdot\curlo\vu(p)}_{=0}\ dp = 0
\end{align*}

Where we used that the divergence of the curl is always zero. Now just release the curl on the equation $\vu = \curlo \vpsi $ and use lemma \ref{lemma:double curl}\emph{double curl}:

\begin{align*}
\w &= \curlo \vu \\
&= \curlo \curlo \vpsi\\
&= -\diveo\grado\vpsi+\underbrace{\grado(\diveo\vpsi)}_{=0} = -\lap\vpsi
\end{align*}

\end{proof}

%=============================================================================
\section{Computing the Stream Function in 2D}
\label{sec:Computing the Stream Function in 2D}

Until here in this section we have only covered the 3D theory of the stream function until now. In order to apply this knowledge to surfaces we do not require much work. In fact, the following theorem makes the transition from 3D to 2D seem easy. Take note that now the velocity field will be tangential to the surface and the vorticity will be pointing towards the normal direction i.e. $\vw=\w\vn$.

\begin{remark}[Sign Convention Confusion Warning]\label{remark:sign convention}
The definition of the stream function for 3D and 2D was established by the equations $\vu=\curlo\vpsi$ and $\vu=\vn\times\grado\psi$ respectively. One could alternatively add a minus sign to the right hand sides and to $\vpsi,\psi$ simultaneously without consequences. Throughout the literature both conventions where used and the two equations used here actually result in opposite signs. This becomes evident at the equations $\lap\vpsi=-\vw$, $\lap\psi=\w$. In order to stay in touch with the resources this sign convention is taken into account here too.
\textbf{This is why the stream functions for 3D and 2D will have different signs}.
\end{remark}

\begin{proposition}[2D Stream Function for $\vu$]
\label{def:2D Stream function}
For a surface $M$ with normal vectors $\vn$ and a velocity vector field $\vu$ of an incompressible flow with vorticity $\w=\vn\cdot \curlo\vu$, the scalar stream function can be established by %given vector field $\vu(x)=-\int_M\grado_pG(x,p)\times\vw(p)\ dp$, the following is a solution to the equation $\vu = \vn \times \grado \psi$ and thus a stream function: %=\grado_{2D} \times \psi 

\begin{equation*}
\psi \ : M \ \longrightarrow \ \R 
\end{equation*}
\begin{equation*}
\psi(x) = \int_M G(x,p)\ \w(p)\ dp
\end{equation*}

where $G:\ M^2\setminus \{(x,x):x\in M\}\longrightarrow\R$ is the Green's function of the surface $M$.
\end{proposition}

\begin{proof}

By definition of the stream function we demand

$$\vu = \vn \times \grado \psi$$

As shown in \cite{Dritschel-2015} plugin this velocity expression into the formula for the vorticity leads to

$$\w=\lap\psi$$

As done before in section \ref{sec:Getting Velocity from the Vorticity} \emph{getting velocity from the vorticity} we will make use of the Green's function $G$ of the surface $M$ in order to solve the Poisson problem, thus leading us to the formula

$$\psi(x) = \int_M G(x,p)\ \w(p)\ dp$$

since

\begin{align*}
\lap_x\psi(x) &= \lap_x\int_M G(x,p)\ \w(p)\ dp\\
&= \int_M \lap_xG(x,p)\ \w(p)\ dp\\
&= \int_M \delta(x-p)\ \w(p)\ dp\\
&= \w(x)
\end{align*}

\end{proof}

\section{Conformal Mappings}
\label{sec:Conformal Mappings}

We will now turn our attention to the theory of conformal mappings. A review can be found here \citep{book:812122}. We will also recall the necessary basics to follow this thesis. First of we recall what a Riemannian metric is \cite{book:1183137}.

\begin{definition}[Metric]
A \emph{metric} $g$ for the manifold $M$ is a collection of inner products on $TM$. For any $p\in M$ there is a inner product

$$ g_p : T_pM \times T_pM \longrightarrow \R $$

The tangent vectors $X,Y\in T_pM$ are assigned to $g_p(X,Y)$.
\end{definition}

The most intuitive metric is the euclidean one $g_p(X,Y)=\langle X,Y \rangle=X^TY=X\cdot Y$. Metrics define what \emph{distances} and \emph{area} mean on a manifold such as a surface. A manifold with a metric is called \emph{Riemannian}.

In order to study our closed surface of genus zero $M$ in $\R^3$ we will need a conformal map from the surface to the sphere $\Sp$. The idea is to study the point vortex dynamics on the sphere with a modified metric rather than to study it on the arbitrary surface $M$ with its euclidean metric.

Before going into conformal maps we also need to recall the transport.

\begin{definition}[Transport]
Let $f: M\longrightarrow N$ be a map between two manifolds and $p\in M$ and $X\in T_pM$. For a smooth curve $\gamma : (-1,1) \longrightarrow M$ with $\gamma(0)=p$ and $X=\gamma'(0)\in T_pM$ we define

$$d_fX:=(f\circ\gamma)'(0)$$

We call $d_fX$ the \emph{transport of $X$ from $M$ to $N$ through $f$}. It is well defined. Note that $X,Y\in T_pM$ while $d_fX,d_fY\in T_{f(p)}N$.
\end{definition}

The transport essentially describes how tangent vector are being carried over from $T_pM$ to $T_pN$ by $f$. We now want to classify maps that transport vectors without distorting the angles between them.

\begin{definition}[Conformal Map]
Let $M, N\subset \R^3$ be a Riemannian manifolds with metrics $\tg,g$ respectively. We call $f \ : \ M \ \longrightarrow \ N$ a \emph{conformal map} if it is angle preserving anywhere.

%That is, for any $p\in M$ and smooth curves $\gamma_1, \gamma_2 : (-1,1) \longrightarrow M$ with $\gamma_1(0)=\gamma_2(0)=p$ we define $X:=\gamma_1'(0)$, $Y:=\gamma_2'(0)$, $d_fX:=(f\circ\gamma_1)'(0)$, $d_fY:=(f\circ\gamma_2)'(0)$ and require:

%$$ \frac{\gamma_1'(0)\cdot\gamma_2'(0)}{|\gamma_1'(0)||\gamma_2'(0)|} = \frac{(f\circ\gamma_1)'(0)\cdot (f\circ\gamma_2)'(0)}{|(f\circ\gamma_1)'(0)||(f\circ\gamma_2)'(0)|} $$

That is, for any $p\in M$ and $X,Y\in T_pM$ we have 

$$ \frac{\tg_p(X,Y)}{\sqrt{\tg_p(X,X)\tg_p(Y,Y)}} = \frac{g_{f(p)}(d_fX,d_fY)}{\sqrt{g_{f(p)}(d_fX,d_fX)g_{f(p)}(d_fY,d_fY)}} $$

\end{definition}

The angel preservation can be seen in the equation above. With the euclidean metrics only, the definition says:

$$ \cos(\theta_M)=\frac{\langle X,Y \rangle }{|X||Y|} = \frac{\langle d_fX,d_fY \rangle }{|d_fX||d_fY|} = \cos(\theta_N)$$

Where $\theta_M, \ \theta_N$ represent the angles between $X,Y$ in $M$ and $d_fX,d_fY$ in $N$ respectively.

\begin{theorem}[Conformal Maps on Closed Surfaces of Genus Zero]
Let $M\subset \R^3$ be a closed surface of genus zero. Then there exists a conformal map
$$f \ : \ M \ \longrightarrow \ \Sp$$
\end{theorem}

\begin{proof}
This is a direct result from the \emph{uniformization theorem} which states that any Riemanian metric on any surface is conformally equivalent to a metric with constant curvature. Closed surfaces of genus 0 are conformally equivalent to a surfaces with constant curvature 1, which has to be $\Sp$.
\end{proof}

So of course we need a conformal map and instead of beying lazy mathematicians and saying \emph{``It exists. Not my problem how you find it!''} we will introduce Kazhdan et al.'s method \cite{Kazhdan:2012:MFM:2346796.2346809} through the implementation chapter \ref{chap:Implementation}. Figure \ref{fig:easy conformal map} shows one example of a conformal map acquired this way. Note how all angles on the sphere are preserved on the closed surface by looking at the right angles on the textures. If the conformal map $f$ is bijective, then the inverse $f^{-1}$ will also be a conformal map.

\begin{figure}
\centering
\includegraphics[width=11cm]{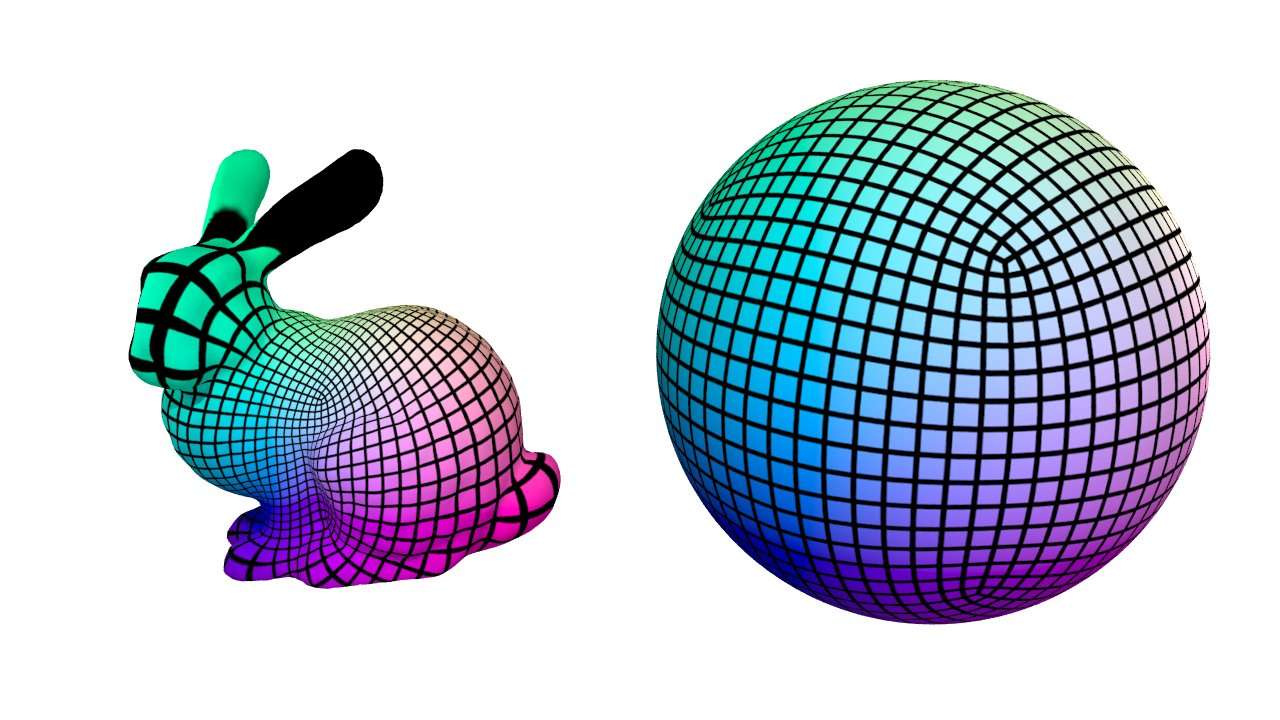}
\caption{A texturing with many right angles is being transported through a conformal map. The colors are determined by the normals on the sphere. The right angles are preserved.}
\label{fig:easy conformal map}
\end{figure}

And of course we must not forget the conformal factor. Since conformal maps preserve angles they can still sacrifice the preservation of scales. The conformal factor catches this precisely.

\begin{definition}[Conformal factor]
Given a conformal map $f \ : \ M \ \longrightarrow \ \Sp$ in $\R^3$ the conformal factor is defined as the function $ h \ : \ M \ \longrightarrow \ \R $ that fulfills 

%$$ \gamma_1'(0)\cdot\gamma_2'(0) = h(p)^2 (f\circ\gamma_1)'(0)\cdot (f\circ\gamma_2)'(0)$$
$$ \tg_p(X,Y) = h(p)^2 g_{f(p)}(d_fX,d_fY)$$

with $X,Y,d_fX,d_fY$ defined like in the definition of conformal maps above.
%for any $p\in M$ and smooth curves $\gamma_1, \gamma_2 : (-1,1) \longrightarrow M$ with $\gamma_1(0)=\gamma_2(0)=p$.

\end{definition}

The existence of a conformal factor follows from the existence of a conformal map. The observant reader will quickly notice from the definition of conformal maps that 
%Two surfaces are called \emph{conformally equivalent} if there exists a conformal map. 
$$ \frac{\tg(X,Y)}{\sqrt{\tg(X,X)\tg(Y,Y)}} = \frac{g(d_fX,d_fY)}{\sqrt{g(d_fX,d_fX)g(d_fY,d_fY)}}  $$

$$\Rightarrow \ \ \ \tg(X,Y) =  \underbrace{ \sqrt{\frac{\tg(X,X)\tg(Y,Y)}{g(d_fX,d_fX)g(d_fY,d_fY)}} }_{h^2:=} g(d_fX,d_fY) $$

The important thing to keep in mind is that from a mathematical perspective it makes no difference if the point vortices are...

\begin{enumerate}
\item living on the surface $M$ with the euclidean metric or...
\item living on $\Sp$ with an adjusted metric $\tg$ that emerges from the conformal mapping $f$ from $M$ to $\Sp$.
\end{enumerate}

So for the sake of being able to apply a theorem later on let's just adapt to the second perspective mentioned above.

\begin{remark}[Conformal Equivalence]
We talked much of $M$ as being our closed surface with genus zero and euclidean metric. We can also talk about $M$ with it's conformal relation to $N$, meaning that we actually talk about $N$ with a modified metric. So in a way $M$=$N$ if it is obvious that we are using the metric obtained from the conformal mapping, not the euclidean one.
\end{remark}

And the true magic comes from the following corollary:

\begin{corollary}[New Metric after Conformal Mapping]
\label{cor:New Metric after Conformal Mapping}
Let $f:M\longrightarrow N$ be a bijective conformal map from the surface $M$ with euclidean metric to $N$. It follows that $\tf:=f^{-1}$ is also a conformal map and $\langle \tX,\tY\rangle = h(p)^2g_{f(p)}(d_f\tX,d_f\tY)$ for $\tX,\tY\in T_pM$.

If we are given the conformal factor $h:M\longrightarrow N$ we can directly specify the metric $g$ of $N$ using:

$$ g_q(X,Y) = \frac{1}{h(\tf(q))^2}\langle d_{\tf}X , d_{\tf}Y \rangle $$

where $q\in N$, $X,\ Y \in T_qN$ and $ d_{\tf}X, \ d_{\tf}Y \in T_{\tf(q)}M$.
\end{corollary}

This corollary will provide us access to operate on the target metric after a conformal mapping which in turn will allow us to use Boatto's results \citep{Boatto2015} for point vortex dynamics. How exactly do we gain access to the conformal factor? We will show one such method in the implementation chapter \ref{chap:Implementation} in the section \ref{sec:Discrete Conformal Factor} \emph{discrete conformal factor}. See figure \ref{fig:conformal mapping}.

\begin{figure}
\centering
\def\svgwidth{0.6\textwidth}
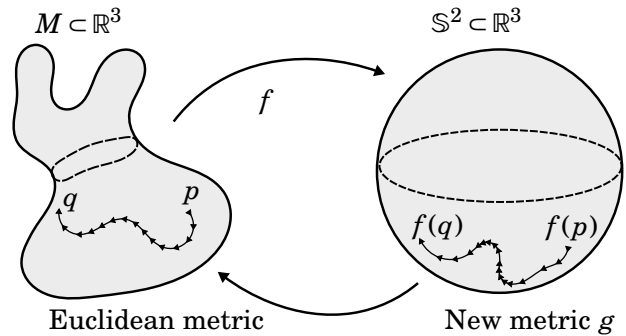
\caption{A conformal Map $f:M\rightarrow \Sp$ mapping a curve from $p$ to $q$. Notice that the curve is now distorted, yet it remains directly identifiable using the conformal maps $f, \ f^{-1}$.}
\label{fig:conformal mapping}
\end{figure}

\section{Symplectic Manifolds and Symplectic Gradients}
\label{sec:Symplectic Manifolds and Symplectic Gradients}

We must also recall the theory of symplectic manifolds as it will play a key role in unlocking the path to point vortex dynamics on surfaces. The very ambitious readers are recommended to take a look at V.I. Arnold's book \cite{book:13762}. However, since our point vortex model does not require the full uncensored theory we will only recapture what we really need in this chapter in an attempt to spare the reader. Special care is being made into making this theory relatable geometrically.

The idea of conservation of energy will be a key argument in this thesis to derive the vortex motion. Energy is given as a scalar function $H$ called the Hamiltonian on the space of possible configurations of the system.

Without having yet defined what point vortices are we can make a small thought experiment. Imagine having $n\in \N$ point vortices $p_i\in M$ on the surface. Together they define a total kinetic energy $H(p_1,...,p_n)$. As usual we demand $H$ to remain constant by the conservation of energy and know that $H$ resides on the space $M^n$. We can already tell that the set of possible configurations

$$\{ \ (p_1,...,p_n)\in M^n \ : \ H(p_1,...,p_n)=const. \  \}$$

is a manifold in $M^n$. Thus we can describe our system of point vortices by closely studying this manifold which was attempted to conceptualize in figure \ref{fig:n_point_hamiltonian}. Our tool of choice to do so are symplectic manifolds and symplectic gradients.

\begin{figure}
\centering
\def\svgwidth{0.8\textwidth}
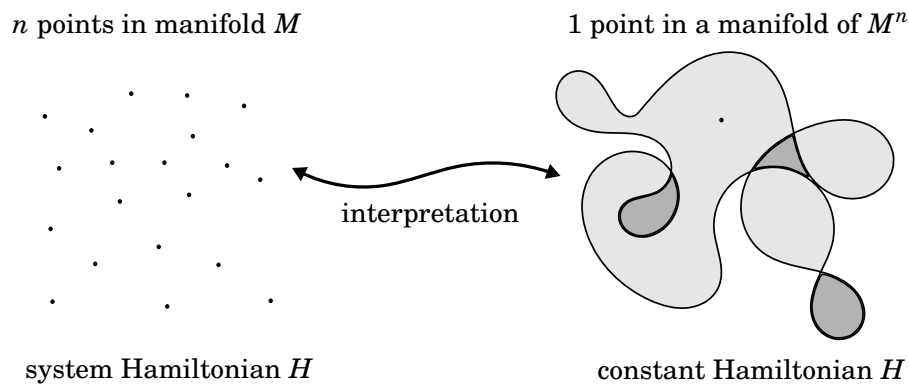
\caption{An illustration of the concept of reinterpreting n-point dynamics as 1-point on a high-dimensional manifold.}
\label{fig:n_point_hamiltonian}
\end{figure}

What is the symplectic gradient $\sgrad H$? In general we know that the gradient $\grad H$ of a function $H$ points into the direction of the steepest accent, but since we plan to conserve a energy $H$ we want to move in a direction that maintain $H$ constant.

On two dimensional surfaces with euclidean metric it becomes evident that the tangent vectors in who's directions $H$ remains constant have to be perpendicular to $\grad H$. Why? Because any directional derivative can be expressed using a unit tangent vector $X\in T_pM$ using $dH(X)=\langle \grad H, X \rangle$, thus $dH(X)=0\Leftrightarrow X\perp \grad H$. The two dimensions on the surface make it really easy to define perpendicular vectors in $T_pM$ using the normal vector $\vn$ and the cross product. The operator 

$$\vn\times \ ( \ \cdot \ ) \ : T_pM \longrightarrow T_pM$$

$$ X \mapsto \vn(p) \times X$$

precisely does the anti-clockwise 90 degree rotation on the surface. The requirement for normal vectors will not be a problem since we always work on closed surfaces of genus 0, which are always orientable. Figure \ref{fig:symplectic_gradient} shows this.

\begin{figure}
\centering
\def\svgwidth{0.4\textwidth}
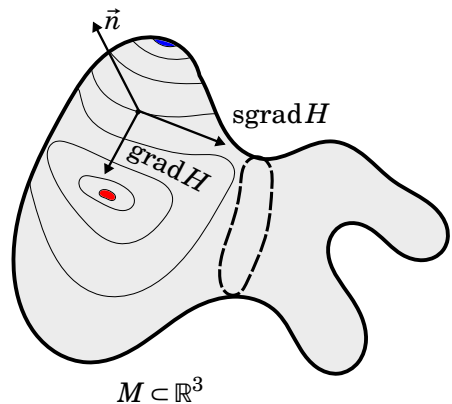
\caption{The symplectic gradient illustrated for euclidean surfaces $M\subset\R^3$ with level lines of $H$. Red and blue indicate maxima and minima values of $H$.}
\label{fig:symplectic_gradient}
\end{figure}

However, the above geometric description only works well for the euclidean metric case. In order to later derive the change of the symplectic gradient after transport through a conformal map we must sadly go into more detail.

\begin{definition}[Symplectic Manifold]
A differentiable manifold $M$ is called \emph{symplectic} if it is equipped with a closed, nondegenerate 2-form $\forall p\in M$,  $\sigma_p: T_pM\times T_pM \longrightarrow \R$. $\sigma$ is called the \emph{symplectic form}.
\end{definition}

\begin{remark}[Symplectic Manifold]
$\sigma$ what? What is a \emph{closed, nondegenerate 2-form}? In short, it a alternating non-degenerate bilinear map on $T_pM$ with vanishing exterior derivative. what is an \emph{exterior derivative}? The restless readers are advised to take a look at \cite{Desbrun:2006:DDF:1185657.1185665} because for our cases in $\R^3$ we can work with a simpler framework. %This reminds me a friend's quote who passed school with the absolute minimum required grades:

%\begin{center}
%\emph{``A smart sheep jumps only as high as it needs to."}
%\end{center}

All of our cases deal with 2D surfaces $M\subset \R^3$ with normals $\vn$ that are conformal to a surface with euclidean metric. Thus we can be satisfied with simply using

$$ \sigma_p(X,Y) := g_p(X,\vn(p) \times Y) $$
\end{remark}

On a manifold $M$ the gradient of a function $H:M \longrightarrow \R$ can be defined using the metric $g$, precisely through

\begin{equation}
Y = \grad H \in T_pM \Leftrightarrow \forall X\in T_pM : g_p(X,Y)=dH(X)
\end{equation}

Where $dH(X)$ is the directional derivative of the scalar function $H$ along $X$. This equation will be used to illustrate the analogy of the symplectic gradient to the regular gradient.

\begin{definition}[Symplectic Gradient]
For a 2D oriented manifold $M\subset \R^3$ with normal vectors $\vn:M\longrightarrow\R^3$ and smooth $H:M\longrightarrow\R$ and the symplectic gradient $\sgrad H$ is defined as the vector $Y$ that satisfies

$$ Y = \sgrad H \in T_pM \Leftrightarrow \forall X\in T_pM : \sigma_p(X,Y)=dH(X) $$
\end{definition}

This is not just orthogonal to the gradient, it respects the metric $g$ as well in the definition of $\sigma$. Now we finally have enough in order to see how the symplectic gradient adjusts to transportation through a conformal map.

We will need $f:M \longrightarrow N$ to be a bijective conformal map with it's conformal factor $h:M\longrightarrow \R$. As established, on a 2D orientable surfaces $M,N$ with the symplectic form can already be specified. $M$ is euclidean and $N$ has the metric $g$, thus $\sigma_{p\in M} = \langle \tX,\vn_M(p) \times \tY \rangle$ and $\sigma_{q\in N} = g(X,\vn_N(q)\times Y)$ respectively with $p\in M,q\in N, \tX,\tY\in T_pM, X,Y\in T_qN$. $H:N\longrightarrow\R$ is the smooth function that we want to compute the gradients from.

Since $M$ is euclidean, the symplectic gradiant for $p\in M$ is just

$$ \sgrad_{p\in M} H = \vn_M(p) \times \grad_p H(p) $$

where $\vn_M(p)$ is the normal vector at $p$ on the surface $M$. The magic question now is how we can define the symplectic gradient $\sgrad_{q\in N}$ of $N$?

\begin{theorem}[Symplectic Gradient after Conformal Mapping]
\label{thr:Symplectic Gradient after Conformal Mapping}
Let $M$ be a 2D surface with euclidean metric $\tg = \langle \ , \ \rangle$ and $N$ be a 2D surface with metric $g$. $f:M\longrightarrow N$ is a bijective conformal map with conformal factor $h:M\longrightarrow \R$, $\tf:=f^{-1}$. $H:N\longrightarrow\R$ is a smooth function and $\tH:=H\circ \tf$ it's push forward to $N$. Then:

$$ \sgrad_{q\in N}\tH(q) = h(\tf(q))^2 d_f \sgrad_{\tf(q)\in M} H = h(\tf(q))^2 \vn_N()\times d_f\grad_{\tf(q)} H $$

\end{theorem}

\begin{proof}

\begin{figure}
\centering
\def\svgwidth{0.7\textwidth}
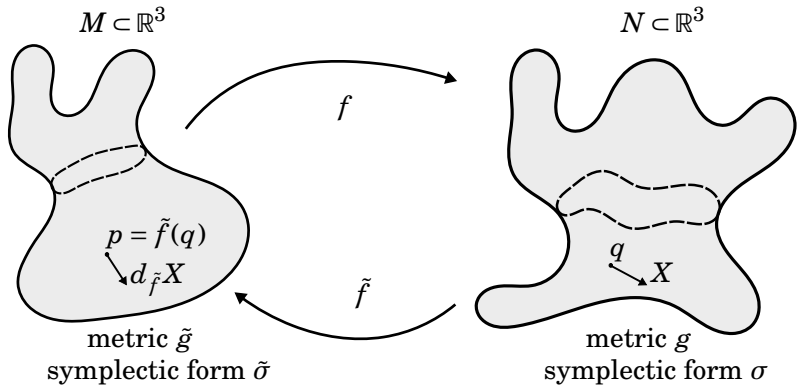
\caption{An overview of our mathematical objects that we deal with in this theorem.}
\label{fig:conformal_transport}
\end{figure}

Figure \ref{fig:conformal_transport} shows what we are working with. We first need to recall corollary \ref{cor:New Metric after Conformal Mapping} which stated that the metric $g$ at $q$ on $N$ can be described using the inverse conformal map $\tf=f^{-1}$.

$$ g_q(X,Y) = \frac{1}{h(\tf(q))^2}\langle d_{\tf}X , d_{\tf}Y \rangle $$

Thus $g_q(X,Y) = 0 \Leftrightarrow \langle d_{\tf}X , d_{\tf}Y \rangle = 0$ and from the euclidean metric we know that the last equation is the case when $d_{\tf}X$ and $d_{\tf}Y$ are perpendicular to each other. Now since conformal maps preserve angles we instantly know that the transport $d_{\tf}$ did not change the angle, thus 

\begin{align*}
g_q(X,Y) = 0 &\Leftrightarrow \langle d_{\tf}X , d_{\tf}Y \rangle = 0 \\
&\Leftrightarrow d_{\tf}X , d_{\tf}Y \hfill \text{ right angled}\\
&\Leftrightarrow X , Y \text{ right angled}\\
&\Leftrightarrow \langle X , Y \rangle = 0 \\
\end{align*} 
%&\Leftrightarrow aY=b\vn\times X, \ \ \text{for some } a,b\R

This in particular means that

$$ d_{f}(\vn_M(p)\times \tX) = \vn_N(f(p))\times d_f\tX $$

This nice result will now help us to express the symplectic gradient on $N$ equipped with the metric $g$. Let $\tsigma, \sigma$ be the symplectic forms on $M$ and $N$ respectively, then the conformal map establishes

\begin{align*}
\tsigma_p(X,Y) &= \tg_p( \ X, \ \vn_M(p)\times Y \ )\\
&= h(p)^2 g_{f(p)}( \ d_fX, \ d_f(\vn_M(p)\times Y) \ )\\
&= h(p)^2 g_{f(p)}( \ d_fX, \ \vn_N(f(p))\times d_fY \ )\\
&= h(p)^2 \sigma_{f(p)}( \ d_fX, \ d_fY \ )
\end{align*}

This result links both symplectic forms nicely together. Of course we can reverse the statement by using $\tf$ instead of $f$. We are searching for an expression for $Y\in T_qN$ that fulfills our symplectic gradient definition.

\begin{align*}
d\tH(X) &= \sigma_q(X,Y)\\
d(H\circ\tf)(X) &= \frac{1}{h(\tf(q))^2} \tsigma( d_{\tf}X , d_{\tf}Y)\\
dH(d_{\tf}X) &= \tsigma( d_{\tf}X , \frac{1}{h(\tf(q))^2} d_{\tf}Y)\\
dH(\tX) &= \tsigma( \tX , \frac{1}{h(\tf(q))^2} d_{\tf}Y)\\
\end{align*}

The renaming of $\tX:=d_{\tf}X\in T_{\tf(q)}M$ makes the connection to the symplectic gradient on $M$ evident. From its definition it follows that

\begin{align*}
\frac{1}{h(\tf(q))^2} d_{\tf}Y &= \sgrad_{ \tf(q) }H\\
d_{\tf}Y &= h(\tf(q))^2\sgrad_{ \tf(q) }H\\
d_fd_{\tf}Y &= d_f ( h(\tf(q))^2\sgrad_{ \tf(q) }H )\\
Y &=  h(\tf(q))^2d_f\sgrad_{ \tf(q) }H \\
\sgrad_q\tH &=  h(\tf(q))^2 d_f ( \ \vn_M(\tf(q))\times \grad_{ \tf(q) }H \ )\\
\sgrad_q\tH &=  h(\tf(q))^2   \vn_N(q)\times d_f( \grad_{ \tf(q) }H )
\end{align*}

This concludes our proof.

%\begin{align*}
%d\tH(X) &= \sigma_q(X,Y)\\
%dH(d_{\tf}X) &= g_q(X,\vn_N(q)\times Y)\\
%&= \frac{1}{h(\tf(q))^2}\tg_{\tf(q)}( d_{\tf} X , d_{\tf}(\vn_N(q)\times Y) \ )\\
%&= \frac{1}{h(\tf(q))^2}\langle d_{\tf}X , \vn_M(\tf(q)) \times d_{\tf}Y \rangle\\
%&= \langle d_{\tf}X , \frac{1}{h(\tf(q))^2} \vn_M(\tf(q)) \times d_{\tf}Y \rangle\\
%&= \sigma ( d_{\tf}X, \frac{1}{h(\tf(q))^2} d_{\tf}Y )
%\end{align*}

%thus we see that 
%
%\begin{align*}
% \sgrad H &= \frac{1}{h(\tf(q))^2} d_{\tf} Y\\
%\Leftrightarrow  h(\tf(q))^2 \sgrad H &= d_{\tf}Y\\
%\Leftrightarrow  h(\tf(q))^2 d_f\sgrad H &= d_f d_{\tf}Y\\
%\Leftrightarrow  h(\tf(q))^2 d_f ( \ \vn_M(\tf(q)) \times \grad H \ ) &= Y\\
%\Leftrightarrow  h(\tf(q))^2  \vn_N(q) \times d_f(  \grad H ) &= \sgrad_{q\in N} H\\
%\end{align*}

\end{proof}

\begin{remark}[Symplectic Gradient after Conformal Mapping]
We used the conformal relation $\tg=h^2g$ for the mapping $f:M\longrightarrow N$ for the metrics $\tg,g$ for the manifolds $M,N$ respectively. However, sometimes the reltation $h^2\tg=g$ is used instead. In that case Theorem \ref{thr:Symplectic Gradient after Conformal Mapping} adapts by flipping the conformal factor:

$$ \sgrad_{q\in N}\tH(q) = \frac{1}{h(\tf(q))^2} d_f \sgrad_{\tf(q)\in M} H $$
\end{remark}

Until now we only looked at one single point $p\in M$ and defined the symplectic gradient for that one. How do we use the symplectic gradient for a point in $M^n$ if we want to describe our $n$ points dynamics?

$$ \underbrace{ M\times \ldots \times M }_{ n \text{times} } = M^n $$

The answer is that we have to grab $n$ symplectic gradients and unite them together in one big operator acting on the individual spaces. Imagine having the points $p_1,\ldots,p_n\in M$ in your $n$-point system with Hamiltonian $H$. If we fix $(n-1)$ points at the time $t=0$, then the system turns back into a 1-point system and the movement of the movable point $p_i$ is described by the symplectic gradient on $M$ again as we learned above. Infinitesimally, for the motion of $p_i$ it makes no difference if the other $(n-1)$ points are fixed or not, which is why we can express each point's motion using their own symplectic gradient.

To write down the linear operator unifying all the symplectic gradients of each point in the $n$-point systsem we rewrite the operator $\vn \times ( \ )$ as a $3\times 3$ matrix

$$
[\vn]_{\times}=\left(
\begin{array}{ccc}
0 & -n_3 & n_2 \\
n_3 & 0 & -n_1 \\
-n_2 & n_1 & 0 \\
\end{array}
\right) \ \ \ \Longleftrightarrow \ \ \ [\vn]_{\times}p = \vn\times p
$$

Next we stack up all point-coordinates $p\in M$ into one big vector in $\textbf{p}\in\R^{\dime(M)}$ and define our matrix:

$$
J=\left(
\begin{array}{c|c|c|c}
[\vn(p_1)]_{\times}& & & \\
\hline
 & [\vn(p_2)]_{\times} & & \\
\hline
  & & \ \ \ddots \ \ & \\
\hline
 & & & [\vn(p_n)]_{\times} \\
\end{array}
\right)
\in \R^{\dime(M)n\times \dime(M)n}
$$

Thus the symplectic gradient of $H:M^n\longrightarrow \R$ now becomes

$$ \sgrad_{\textbf{p}}H=J\grado H $$

If this way of thinking seems a bit odd we would like to recall the following: imagine $n\in\N$ body dynamics in $\R^3$ where every point is described by a position $x_i$ and a momentum $p_i$. The Energy is usually a mix of potential and kinetic energy and thus a function of $H:\ M^n \longrightarrow \R$ with $M$ being the position-momentum space $M=\R^3\times\R^3$. Hamiltonian mechanics now claim that the time evolution of this system is uniquely defined by

$$\forall i \ : \ \frac{dp_i}{dt}=-\frac{\partial H}{\partial x_i} \ , \ \frac{dx_i}{dt}=+\frac{\partial H}{\partial p_i}$$

in fact

$$\forall i \ : \ddv{X_i}{Y_i} = \ddv{\frac{dx_i}{dt}}{\frac{dp_i}{dt}} = \ddv{+\frac{\partial H}{\partial p_i}}{-\frac{\partial H}{\partial x_i}} = \left( \begin{smallmatrix} 0&+1\\ -1&0 \end{smallmatrix} \right)\ddv{\frac{\partial H}{\partial x_i}}{\frac{\partial H}{\partial p_i}} = J_i \grad_{M} H = \sgrad_{M} H$$

where $\ddv{x_i}{p_i}$ is an element of the position-momentum space $M=\R^3\times\R^3$. This just illustrates the symplectic gradient in $M$ for one $i$. The whole system would line up all the $i$s under one vector together with the matrix

$$
J=\left(
\begin{array}{c|c|c|c}
J_1& & & \\
\hline
 & J_2& & \\
\hline
  & &\ddots & \\
\hline
 & & &J_n \\
\end{array}
\right)
\in \R^{2n\times 2n}
$$

$J_i$ is then the rotation of $\grad_M H$ inside the position-momentum space $M$. $J$ merely collects all of these rotations from the separated spaces to be applied to elements in $M^n$. Thus our entire Hamiltonian system can be described using $\sgrad_{M^n} H = J\grad_{M^n} H$. Here $\sgrad H = J \grado H$. %$\sigma(X,Y)=\langle X, J Y \rangle$.

%The system is described as a point $p$ in $M^n$ and $\sgrad_{M^n} H$ provides the flow direction of that point to the next system description.

%=============================================================================
%\import{chapters/chapter03/}{chap03.tex}
%\clearemptydoublepage

%+ + + + + + + + + + + + + + + + + + + + + + + + + + + + + + + + + + + + + + + + + + + + + + + + + + + + + + + + + + + + + + + + + + + + + + + + + + + + + + + + + + + + + + + + + + + + + + + + + + + + + + + + + + + + + + + + + + + + + + + + +    NEW CHAPTER     + + + + + + + + + + + + + + + + + + + + + + + + + + + + + + + + + + + + + + + + + + + + + + + + + + + + + + + + + + + + + + + + + + + +

\chapter{Point Vortices}
\label{chap:Point Vortices}

Finally we arrived at the chapter where we dive into the details of the main method of this thesis. The previous chapters served to establish the history, background, motivation and basics to understand what is about to be presented. Now we will show a formal definition of \emph{point vortices} and explain how they serve our cause to make a fast fluid simulation.

While reading this keep the following in mind: point vortices attempt to approximate any given vorticity function $\w$ through a finite number of points.

%=============================================================================
\section{Definition of Point Vortices}
\label{sec:Definition of Point Vortices}

\begin{definition}[Point Vortices]
\label{def:Point Vortices}
Given $n\in\N$ points $p_i\in M$ with $n$ scalars $\w_i\in\R$ respectively, we call the ``$p_i$''s the \emph{point vortices} with vortex strengths $\w_i$ if the vorticity of the fluid they represent is expressed as:
\begin{equation}
\w(x) = \sum_{i=1}^n \w_i\delta(x-p_i)
\end{equation}

where $\delta$ is the Dirac delta function on $M$.
\end{definition}

Is this a bold move? Instead of adjusting to a given initial velocity field $\vu$ we just push the velocity field aside and define the vorticity field $\w$ the way we desire it to have. Further more we just introduced singularities and spikes at the point vortices on the surface.

Fear not, our explanation arrives now. For this we need to recall our description of vorticity as we have made in section \ref{sec:Vorticity of Fluids}. There we justified how the magnitude of the vorticity describes the amount of rotation an object in that flow experiences. At the center of a spiraling whirl in the fluid this amount of rotation is exceptionally high and this is why we aim to approximate our vorticity by focusing on it's strongest bits.

The use of \emph{Dirac delta functions} $\delta$ is always dangerous and makes only sense when integration comes into play. This is precisely what we have been preparing in section \ref{sec:Computing the Stream Function in 2D} \emph{computing the stream function in 2D}. The singularities will become a problem when we try to compute velocities that are very close to the point vortices, but as we will see later, computing the velocity of the point vortex itself is no big deal.
Computing the velocity with a some distance to the points leads to similar results as if the vorticity was not on one point but spread in the neighborhood of the point. The reasons will be similar to how we can often assume planets to be point masses when computing gravitational forces.

All that we do is the approximation that $\curlo\vu(x)=\w(x)\approx \sum_{i=1}^n \w_i\delta(x-p_i)$. As a motivational image we can look at the miso cup and slide a spoon through it. You would notice two spirals turning in opposite directions on the surface. Imagine each center of a spiral to be a point vortex and the the rotation speed being its vortex strength. See figure \ref{fig:coffee}.

\begin{figure}
\centering%\includegraphics[width=6cm]{images/general/soup_vortex}
\def\svgwidth{0.47\textwidth}
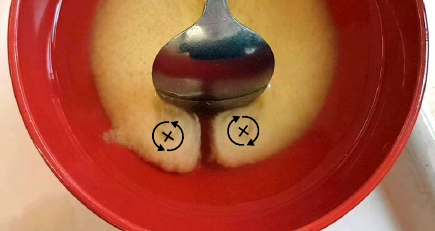
\import{images/vectorgraphics/}{2dmultiplevortices3.tex}
\caption{Soup as seen by mathematicians. The spoon is sliding upwards and now the fluids spirals around two points from the edge of the spoon. We think of point vortices as the center of each spiral.}
\label{fig:coffee}
\end{figure}

%=============================================================================
\section{Derivation of the Point Vortices Stream Function}
\label{sec:Derivation of the Point Vortices Stream Function}

In section \ref{sec:Computing the Stream Function in 2D} we derived how to compute the velocity field from vorticity through arguments using integrals in combinations with Green's functions $G$. Our definition of point vortices (definition \ref{def:Point Vortices}) uses Dirac delta functions and that is precisely what we are about to exploit with the stream function $\psi$.

\begin{theorem}[Point Vortices Stream Function]
\label{thr:Point Vortices Stream Function}
Given a point vortices system with vorticity $\w(x) = \sum_{i=1}^n \w_i\delta(x-p_i)$ on a surface $M$, the corresponding stream function can be written as

$$\psi(x) = \sum_{i=1}^n \w_i G(x,p_i)$$

for $x\notin\{p_1, \ldots, p_n\}$.

\end{theorem}

\begin{proof}

We recall proposition \ref{def:2D Stream function} first, as it links the vorticity $\w$ with the 2d stream function $\psi$.

\begin{equation}
\psi(x) = \int_M G(x,p)\ \w(p)\ dp
\end{equation}

We can now inject the point vortices assumption and specify the vorticity $\w$ from the stream function definition with our finite sum of Dirac delta functions. This causes the following chain of changes:

\begin{align*}
\psi(x) &= \int_M G(x,p)\ \w(p)\ dp \\
&= \int_M G(x,p)\ \sum_{i=1}^n \w_i\delta(p-p_i)\ dp \\
&= \sum_{i=1}^n \w_i\int_M G(x,p) \delta(p-p_i)\ dp \\
&= \sum_{i=1}^n \w_i G(x,p_i)
\end{align*}

which concludes our proof.

\end{proof}

%=============================================================================
\section{Point Vortex Velocity Field}
\label{sec:Point Vortex Velocity Field}

Theorem \ref{thr:Point Vortices Stream Function} gives us an easy way to compute the stream function of our point vortex system if we know the Green's function of the surface $M$. This allows us to straight forwardly compute the velocity induced from the vorticity through the stream function through the relation 

$$\vu = \sgrad \psi$$

\begin{theorem}[Point Vortex Velocity Field]
\label{thr:Point Vortex Velocity Field}
Given $n \in \N$ point vortices $p_i$ on the surface $M$ with vortex strengths $\w_i\in\R$ respectively, then the velocity $\vu$ at a point $x\in M$ with $\forall i: x\neq p_j$ is expressed by
\begin{equation}
\vu(x) = \sum_{ \substack{i=1} }^n \w_i \sgrad_x G(x,p_i)
\end{equation}
Where $J$ denotes the 90 degree anti-clockwise rotation around the surface normal.
\end{theorem}

\begin{proof}
Just insert the stream function constructed in theorem \ref{thr:Point Vortices Stream Function} into the relation $\vu = \vn \times \grado \psi$.
\begin{align*}
\vu(x) &= \sgrad_x \psi(x)\\
&= \sgrad_x \sum_{i=1}^n \w_i G(x,p_i)\\
&= \sum_{i=1}^n \w_i \sgrad_x  G(x,p_i)
\end{align*}
\end{proof}

We now have easy access to the fluid velocity field almost anywhere and any passive point can now be advected. However, we still cannot compute the velocity at each point vortex. The problem arises from the Green's function being singular when $x=p_i$ in the above formula. The next section will explain what to compute $\vu(p_i)$.

%=============================================================================
\section{Point Vortex Dynamics}
\label{sec:Point Vortex Dynamics}

Theorem \ref{thr:Point Vortex Velocity Field} blessed us with an easy access to the fluid velocity field away from the point vortices. However, we still need to find an expression that governs the motion of each point vortex. As mentioned with the introduction of the stream function in section \ref{sec:Stream Functions} \emph{stream functions} we can compute the velocity $\vu$ of any point $q\in M$ on the fluid away from the vortices using:

\begin{equation}
\vu = \vn \times \grado \psi = \sgrad\psi
\end{equation}

The operator $\sgrad$ in the last equation is the rotation operator together with the gradient as mentioned in section \ref{sec:Symplectic Manifolds and Symplectic Gradients}. $\vn \times (\ )$ resembles an anti-clockwise rotation around the normal vector of the surface. Before inserting different Green's functions we want to collect a few statement about this induced velocity field. % This illustrates that $\vu$ is actually the symplectic gradient of $\psi$.

On the point vortices them self we cannot compute the velocity because of the singularities at them. In fact, the velocity induced by one single point vortex drops inversely proportional with the distance to the point, which is why it will blow up the closer we get to the vortex point. The speed induced by a point vortex is proportional to the point vortex strength $\w_i$. See figure \ref{fig:vortexstrength}.

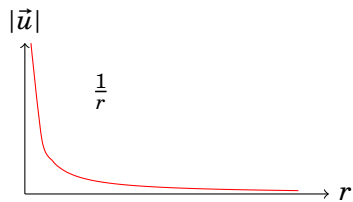
\begin{figure}
\centering
\begin{tikzpicture}
      \draw[->] (0,0) -- (4,0) node[right] {$r$};
      \draw[->] (0,0) -- (0,2) node[above] {$|\vu|$};
      \draw (1,1)  node[above] {$\frac{1}{r}$};
      \draw[scale=0.4,domain=0.2:9,smooth,variable=\x,red] plot ({\x},{1/\x});
      %\draw[scale=0.5,domain=-3:3,smooth,variable=\y,red]  plot ({\y*\y},{\y});
\end{tikzpicture}
\caption{The magnitude of the velocity $|\vu|$ induced by one point vortex plotted against the distance $r$ to that point.}
\label{fig:vortexstrength}
\end{figure}

Figure \ref{fig:vortexstrengthview} shows the velocity field $\vu=J\grado\psi$ induced by the stream function from one single point vortex $\w=\delta$ on the plane $\R^2$. Even when not on a flat plane, we can imagine each point vortex to be inducing a rotation around itself, a velocity that looses its strength inversely proportional by the distance on the surface to the point. It spins counter clockwise with positive vortex strength and clockwise otherwise.

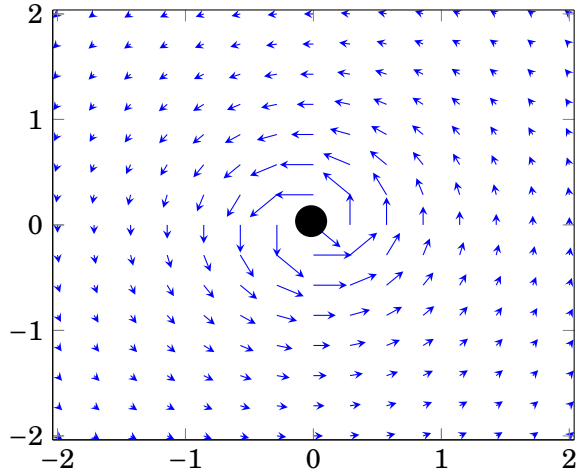
\begin{figure}
\centering
\def \radiusfieldfunc{-0.5*(x^2+y^2)^(-3/2)}
\def \xfieldfunc{-y}
\def \yfieldfunc{x}
%\def \factor{max(0,1)}

% leave the bottom part as untouched as possible. manipulate the function from above
\def\veclength{sqrt((\xfieldfunc)^2+(\yfieldfunc)^2)}

%my
\def \factor{min(1/(\veclength),2)}

\begin{tikzpicture}
\begin{axis}[domain=-2:2, view={0}{90}]
\addplot3[
blue, quiver={
u={\factor*\xfieldfunc/(\veclength)},
v={\factor*\yfieldfunc/(\veclength)},
scale arrows=0.15},
-stealth,samples=15
] {0};
\end{axis}
%\draw (3,3)  node[above] {$\w$};
\node[circle,fill=black,inner sep=0pt,minimum size=12pt] (a) at (3.4,2.9) {};
\end{tikzpicture}

%this example works fine
\caption{The velocities induced by a single point vortex $\w(\vx)=\delta(\vx)$ in the plane. The vortex is visualized by a black spot.}
\label{fig:vortexstrengthview}
\end{figure}

Where do we start to describe the dynamics of the point vortices? How does the velocity field induced by all the point vortices affect the point vortices them self? The best justifications comes from section \ref{sec:The Vorticity Equation} where we established the vorticity equation and derived the statement: \emph{vorticity is simply advected along the fluid flow} \cite{Elcott-2007}. This means that we want to advect $\w$ along the flow of their induced velocity field and thus we have to advect the point vortices themself along the flow. To progress we will need Kirchhoff's assumption \cite{kirchhoff1891vorlesungen} and we will now justify it by using some figures as an example.

\begin{figure}
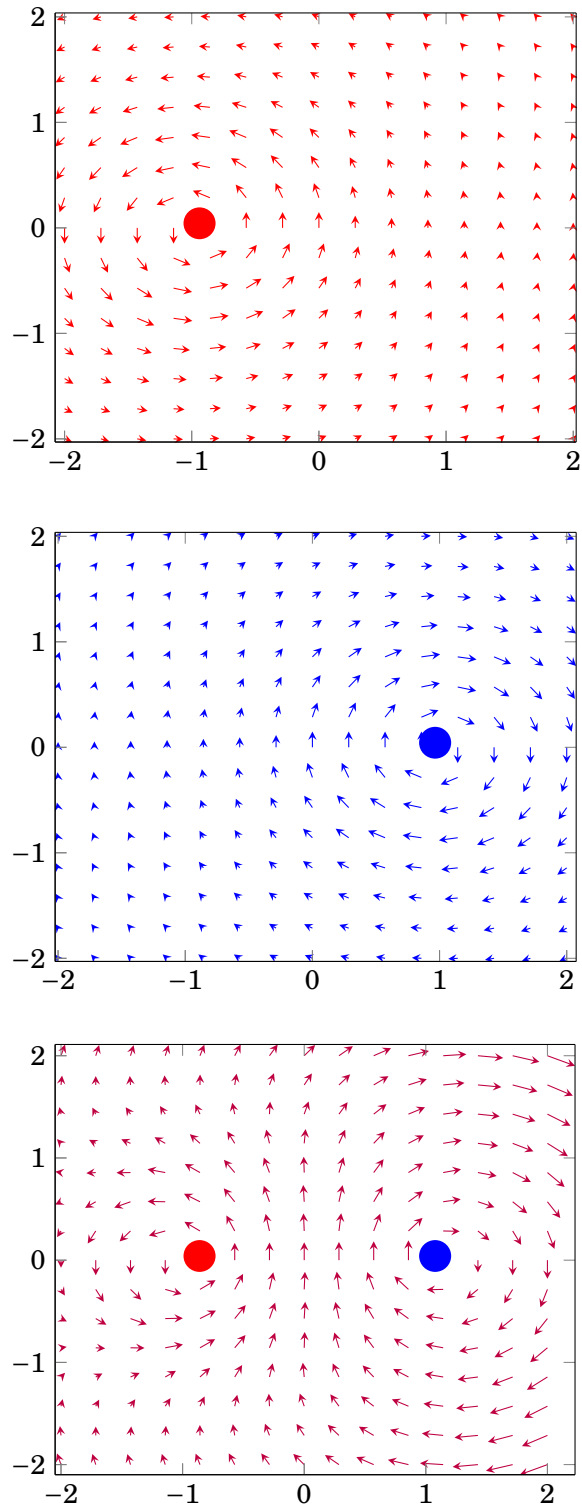

\centering

\begin{minipage}{.4\textwidth}
  \centering
  \import{images/vectorgraphics/}{2dmultiplevortices1.tex}
\end{minipage}%

\begin{minipage}{.4\textwidth}
  \centering
  \import{images/vectorgraphics/}{2dmultiplevortices2.tex}
\end{minipage}

\begin{minipage}{.4\textwidth}
  \centering
  \import{images/vectorgraphics/}{2dmultiplevortices3.tex}
\end{minipage}

\caption{Left point vortex with vortex strength $\w=+1$ (Top). Right point vortex with vortex strength $\w=-1$ (Middle). Both point vortices acting together (Bottom).}
\label{fig:two point vortices apart}
\end{figure}

Figure \ref{fig:two point vortices apart} shows how the velocity field is shaped in three different situations. At first we are given a point vortex that is shifted to the left and has positive vorticity strength. It's vorticity is given by

\begin{equation}
\w_1 \ddv{x_1}{x_2} := +\delta \left( \ddv{x_1}{x_2} + \ddv{1}{0} \right)
\end{equation}

The point vortex of the second image is shifted to the right and has negative vorticity and thus creates a clockwise moving velocity field around itself. It's vorticity is given by

\begin{equation}
\w_2 \ddv{x_1}{x_2} := -\delta \left( \ddv{x_1}{x_2} + \ddv{-1}{0} \right)
\end{equation}

The final image of figure \ref{fig:two point vortices apart} shows the velocity field both point vortices acting together. The vorticity is now given by

\begin{equation}
\w \ddv{x_1}{x_2} := \w_1 \ddv{x_1}{x_2} + \w_2 \ddv{x_1}{x_2} = \delta \left( \ddv{x_1}{x_2} + \ddv{1}{0} \right) \ - \ \delta \left( \ddv{x_1}{x_2} + \ddv{-1}{0} \right)
\end{equation}

The important message to take home from these images is that \emph{velocity fields add up} with added vorticity. One point vortex moves the whole velocity field created by the other one and vise versa.

\begin{corollary}[Vorticity Fields Add Up]
\label{cor:Vorticity Fields add up}
Let $M$ be a surface with $n\in\N$ point vortices $p_i\in M$ with vorticity strengths $\w_i\in\R$ respectively. Each point vortex defines a vorticity field

$$\tilde{\w_i}(x):=\w_i\delta(x-p_i)$$

each vorticity field induces a velocity at $x\neq p_i$

$$\tilde{u}_i(x):=\w_i\sgrad G(x,p_i)$$

The combined vorticity of the point vortex system is given by

$$\w(x):=\sum_{i=1}^n\tilde{\w_i}(x)=\sum_{i=1}^n\w_i\delta(x-p_i)$$

The velocity field created by the combined vorticity is then equal to the sum of the individually induced velocity fields for $x\notin\{p_1, \ldots, p_n\}$.

$$\vu(x)= \sum_{i=1}^n\tilde{u}_i(x)=\sum_{i=1}^n \w_i\sgrad G(x,p_i)$$
\end{corollary}

\begin{proof}
The proof is simply a consequence of the linearity of the operations.
\end{proof}

Corollary \ref{cor:Vorticity Fields add up} gives us a hint on how to compute the velocities of the point vortices themself. The key idea is that \emph{each point vortex's motion is determined only by the velocity fields induced by all other point vortices}. This is the standard assumption made by Kirchhoff \cite{kirchhoff1891vorlesungen}. 

Figure \ref{fig:two point vortices apart} illustrates this with it's two vortices in a plane. See how the blue point vortex alone does not determine any velocity direction in its own location but the red point vortex does generate a velocity at the position of the blue vortex. In fact, we can imagine that the entire spiraling structure of velocities created by the blue point vortex is being advected by the velocity field of the red point vortex and vice versa.

This motivates us to compute the velocities on the point vortices specifically by just ignoring the contribution of that particular point vortex to itself. Corollary \ref{cor:Vorticity Fields add up} about the sum of velocity fields makes this particularly easy.

For the figure \ref{fig:two point vortices apart} specifically this means that

$$\vu\ddv{-1}{0} = \w_2\sgrad G\left(\ddv{-1}{0},\ddv{+1}{0} \right)$$
$$\vu\ddv{+1}{0} = \w_1\sgrad G\left(\ddv{+1}{0},\ddv{-1}{0} \right)$$

Where $G(x,p):=-\frac{1}{2\pi}\ln(|x-p|)$ as mentioned in proposition \ref{prop:Green's function in the 2D plane} for $\R^2$. We present the general version of this statement into a proposition.

\begin{proposition}[Point Vortex Dynamics]
\label{prop:Point Vortex Dynamics}
Given $n \in \N$ point vortices $p_i$ on the surface $M$ with vortex strengths $\w_i\in\R$ respectively, then the velocity $\vu(p_j)$ of the point vortex at $p_j$ is expressed by
\begin{equation}
\vu(p_j) = \sum_{ \substack{i=1 \\ i\neq j} }^n \w_i\sgrad_{p_i} G(p_j,p_i)
\end{equation}
\end{proposition}

Note again that getting the Green's function $G$ is a huge challenge. Note also that if we have it, then we can easily compute most of the fluid properties we could ask for. The great thing about the Green's function is that it takes care of all the shape properties of $M$. We derived everything without asking for curvature.

%=============================================================================
\section{Point Vortex Energy}
\label{sec:Point Vortex Energy}

For the planar and spherical vortex dynamics we have all the tools ready to describe their motion but for general closed surfaces it is necessary to introduce the energy of the system. This is not so simple due to the singularities of the stream function $\psi$ at the point vortices.

First quickly note that:

\begin{lemma}
For any scalar function $\psi$ the following equation is true:
$$\diveo(\psi\grado\psi) = |\grado\psi|^2 + \psi\lap\psi$$
\end{lemma}

Then we want to throw in a statement about the kinetic energy of our system:

\begin{lemma}
\label{lemma:kinetic energy}
In an incompressible flow with uniform mass, given a vorticity $\w$ with stream function $\psi$ on a closed surface $M$ the kinetic energy can be expressed as:
$$E= - \frac{1}{2}\int_M\psi(p)\w(p) dp$$
\end{lemma}

\begin{proof}

To prove this we use the lemma we just mentioned , the relationship $\w=\lap\psi$ established in section \ref{sec:Computing the Stream Function} and the divergence theorem with the fact that our closed surfaces have no boundary $\partial M=\emptyset$.

\begin{align*}
E &= \frac{1}{2}\int_M |\vu(p)|^2dp\\
&= \frac{1}{2}\int_M |\sgrad \psi(p)|^2dp\\
&= \frac{1}{2}\int_M |\vn\times\grado \psi(p)|^2dp\\
&= \frac{1}{2}\int_M |\grado \psi(p)|^2dp\\
&= \frac{1}{2}\int_M \diveo(\psi(p)\grado\psi(p) ) - \psi(p)\underbrace{\lap\psi(p)}_{=\w(p)}dp\\
&= \frac{1}{2}\int_M \diveo(\psi(p)\grado\psi(p))dp - \frac{1}{2}\int_M\psi(p)\w(p) dp\\
&= \frac{1}{2}\int_{\partial M = \emptyset} \vn\cdot(\psi(p)\grado\psi(p)) - \frac{1}{2}\int_M\psi(p)\w(p) dp\\
&= - \frac{1}{2}\int_M\psi(p)\w(p) dp
\end{align*}

\end{proof}

Next we want to combine the result of lemma \ref{lemma:kinetic energy} with all the results of our point vortex assumption.

\begin{theorem}[Kinetic Energy of Point Vortex System]
\label{thr:Kinetic Energy of Point Vortex System}
Given $n$ point vortices on a closed surface $M$ at locations $p_i$ with strengths $\w_i$ we can express the finite \emph{excess energy} of the system  as: %excluding the $p_i$ points

%$$ E = -\frac{1}{2}\sum_{\substack{i=1 \\ j=1}}^n\w_i\w_jG(p_i,p_j)$$
$$E=  -\sum_{\substack{i<j}}^n  \w_i\w_j G(p_i,p_j)$$
\end{theorem}

\begin{proof}

The point vortex relation tell us that $\w(p) = \sum_{i=1}^n \w_i\delta(p-p_i)$ and from theorem \ref{thr:Point Vortices Stream Function} we know that $\psi(p) = \sum_{j=1}^n \w_j G(p,p_j)$. We want to insert both of these relations into the result from lemma \ref{lemma:kinetic energy}.

But wait! The singularities! If we actually plug in both of these equations we end up with $G(p_i,p_i)$ terms emerging in the sum. That is forbidden! Let's see how far we get:

\begin{align*}
E &= - \frac{1}{2}\int_M\psi(p)\w(p) dp\\
&= - \frac{1}{2}\int_M \left( \sum_{j=1}^n \w_j G(p,p_j)\right)\left( \sum_{i=1}^n \w_i\delta(p-p_i)\right) dp\\
&= - \frac{1}{2}  \int_M \sum_{\substack{i=1 \\ j=1}}^n\w_i\w_j G(p,p_j) \delta(p-p_i) dp\\
&= - \frac{1}{2}  \sum_{\substack{i=1 \\ j=1}}^n \w_i\w_j \int_M G(p,p_j) \delta(p-p_i) dp\\
&= - \frac{1}{2}  \sum_{\substack{i=1 \\ j=1}}^n \w_i\w_j G(p_i,p_j)\\
&= - \frac{1}{2}  \sum_{\substack{i,j=1 \\ i\neq j}}^n \w_i\w_j G(p_i,p_j)   - \frac{1}{2}  \sum_{\substack{i=1}}^n \w_i^2G(p_i,p_i) \\
&= - \sum_{\substack{i<j}}^n \w_i\w_j G(p_i,p_j)   - \frac{1}{2}  \sum_{\substack{i=1}}^n \w_i^2G(p_i,p_i)
\end{align*}

These $G(p_i,p_i)$ terms describe the effect of point vortices on them self. They are kind of trying to induce an infinite velocity without a direction. We obviously have to do something to include Kirrchoff's assumption \cite{kirchhoff1891vorlesungen} into this calculation, that \emph{each point vortex's motion is determined only by the velocity fields induced by all other point vortices}. One way to do this is to modify the stream function such that $j\neq i$ in the sum.

\begin{center}
$\w(p)\psi(p)=\sum_{i=1}^n \w_i\delta(p-p_i) \sum_{\substack{j=1}}^n \w_j G(p,p_j)$

becomes 

$\w(p)\psi(p)=\sum_{i=1}^n \w_i\delta(p-p_i) \sum_{\substack{j=1 \\ j\neq i}}^n \w_j G(p,p_j)$
\end{center}

This is justified by proposition \ref{prop:Point Vortex Dynamics} where we excluded the velocity induced by the point vortex on itself, which is equivalent to removing one term from the stream function. Let's see if this works out:

\begin{align*}
E &= - \frac{1}{2}\int_M\w(p)\psi(p) dp\\
&= - \frac{1}{2}\int_M\sum_{i=1}^n \w_i\delta(p-p_i) \sum_{\substack{j=1 \\ j\neq i}}^n \w_j G(p,p_j)dp\\
&= - \frac{1}{2} \sum_{\substack{i,j=1 \\ j\neq i}}^n  \w_i\w_j\int_M G(p,p_j)\delta(p-p_i)dp\\
&= - \frac{1}{2} \sum_{\substack{i,j=1 \\ j\neq i}}^n  \w_i\w_j G(p_i,p_j)\\
&=  -\sum_{\substack{i<j}}^n  \w_i\w_j G(p_i,p_j)\\
\end{align*}

As you can see our singularity disappeared and we can finally rest in peace.

\end{proof}

This gives us what we need to establish the following connection between the symplectic gradient of the energy and the point vortex velocities:

\begin{theorem}[Point Vortex Energy Motion]
\label{thr:Point Vortex Energy Motion}

Given $n$ distinct point vortices $p_i$ with strengths $\w_i$ respectively on a closed surface $M$, then under Kirchhoff's assumption the velocity $\vu(p_k)$ of the $k$th point vortex can be expressed as:

\begin{align*}
\vu(p_k) = -\frac{1}{\w_k} \sgrad_{p_k} E
\end{align*}

\end{theorem}

\begin{proof}

We start with the right hand side and play with it until we can apply proposition \ref{prop:Point Vortex Dynamics} \emph{Point Vortex Dynamics} to the equation, which we is allowed because Kirchhoff's assumption is given. $\grado_{p_k}G(p_i,p_j)$ will disappear whenever $i\neq k$ and $j\neq k$.

\begin{align*}
-\frac{1}{\w_k} \sgrad_{p_k}  E
&=  -\frac{1}{\w_k}\sgrad_{p_k}\left(-\sum_{\substack{i<j}}^n  \w_i\w_j G(p_i,p_j)\right) \\
&=  \sum_{\substack{i<j}}^n  \frac{\w_i\w_j}{\w_k} \sgrad_{p_k}G(p_i,p_j)\\
&=  \sum_{\substack{i=1\\i\neq k}}^n  \frac{\w_i\w_k}{\w_k} \sgrad_{p_k}G(p_i,p_k)\\
&=  \sum_{\substack{i=1\\i\neq k}}^n  \w_i \sgrad_{p_k}G(p_i,p_k)\\
&= \vu(p_k)
\end{align*}

\end{proof}

This energy expression can and will later be used to derive the equations of motion for point vortices on general closed surfaces. How can we understand theorem \ref{thr:Point Vortex Energy Motion}? The math shows us that it returns us the correct velocity but why are we dividing by $\w_k$? The simple answer is that the devision by $\w_i$ and the elimination of non-$p_k$-terms through $\grado_{p_k}$ do the correct adjustments needed by Kirchhoff's assumptions.

Remember this: until now the only assumptions we made in this section are that we have point vortices on a smooth closed surface. All these results will remain applicable for planes, spheres and other closed surfaces.

%=============================================================================
\section{Methodology}
\label{sec:Methodology}

In the following three chapters we will apply the same methodology to the problem at hand. The differences in each chapter arise only from the different geometries used as the basis to run the point vortices on. The chronology is as such that the reader is guided from the simplest case to understand up to the most difficult one.

\begin{equation*}
\text{Planar case} \ \ \ \ \longrightarrow \ \ \ \ \text{Spherical case} \ \ \ \ \longrightarrow \ \ \ \ \text{Closed surface case (genus zero)}
\end{equation*}

\textbf{1:} At first we start with planar vortex dynamics (chapter \ref{chap:\planarvd}) where we have not to worry about the effects of curvature to derive easy to follow equations of motion. This serves as an ideal introduction to get familiar with the effects of vortex dynamics that will then be generalized in the follow up chapters.

\textbf{2:} We then proceed to explain spherical vortex dynamics (chapter \ref{chap:\sphericalvd}) where we are lucky to also be able to fully derive the equations of motions directly. By now the effects with curvature become more evident.

\textbf{3:} Finally we will explain the closed surface vortex dynamics (chapter \ref{chap:\generalvd}) where we will in detail explain the findings of \cite{Boatto2015} and how to use them to establish the point vortices dynamics on any closed surface of genus zero. For each of the three chapters mentioned above we will repeat the following steps:

\vspace{5mm}

\begin{center}
\text{Observe stream function or energy of the system of the given geometry}

$\downarrow$

\text{Derive equations of motion from the stream function or the energy}

$\downarrow$

\text{Simulate the results. Observe test cases.}
\end{center}

\vspace{5mm}

Once we finish all of these smooth theory cases we then move on to chapter \ref{chap:Implementation} to give crucial implementation details for the readers to understand it's execution thoroughly and if necessary to implement it them self.

\section{Results to Observe}
\label{sec:Results to Observe}

Let us mention a number of expected behaviors that we want to see from our vortex dynamics in order to compare the results across the different geometric cases.

The first and most interesting vortex dynamic effect is \emph{Kimuras conjecture} \cite{Kimura245}, which states that a infinitesimally close pair of two point vortices with equal but opposite vorticity will move together along a geodesic line. This will be especially fun to visualize for point vortex pairs on closed surfaces.

The 3D equivalent of a vortex pair would be a vortex ring moving forward inside a fluid. Take a look at figure \ref{fig:vortex ring} to see what we mean. The 2D slice representation of such a vortex ring looks exactly like two point vortices next to each other. This is also a reminder of an example of point vortices in real life as seen on the surface of a fluid.

\begin{figure}
\centering
\def\svgwidth{0.4\textwidth}
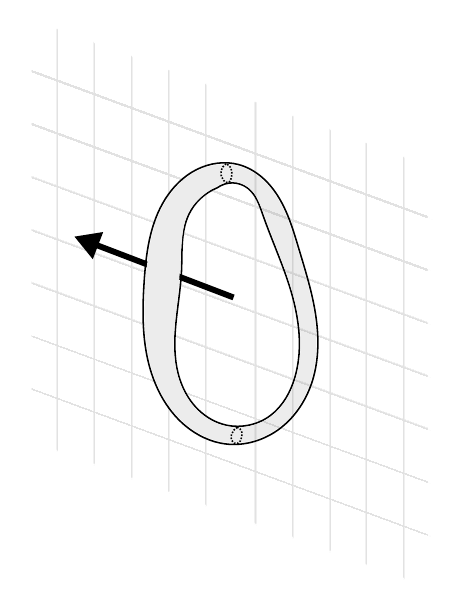
\def\svgwidth{0.4\textwidth}
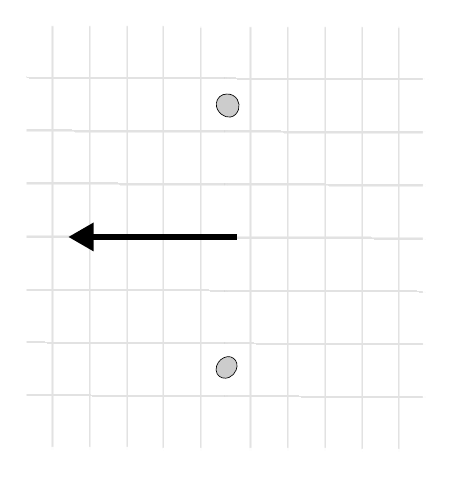
\caption{A vortex ring moving forward (left). A 2D slice representation of the vortex ring as point vortices (right).}
\label{fig:vortex ring}
\end{figure}

Another neat effect is the leap frogging of two vortex rings \cite{doi:10.1063/1.869160} that can also be represented in 2D using for point vortices. Figure \ref{fig:vortex leap frog} shows footage capturing leapfrogging vortex rings in real life. One can easily imagine how it should look like using 4 point vortices by looking at the pictures.

\begin{figure}
\centering
\def\ll{4cm}
\includegraphics[height=\ll]{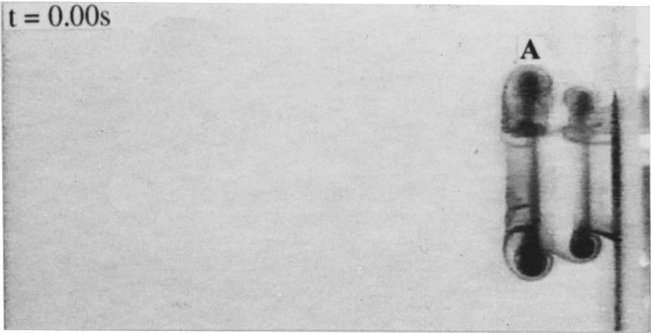}
\includegraphics[height=\ll]{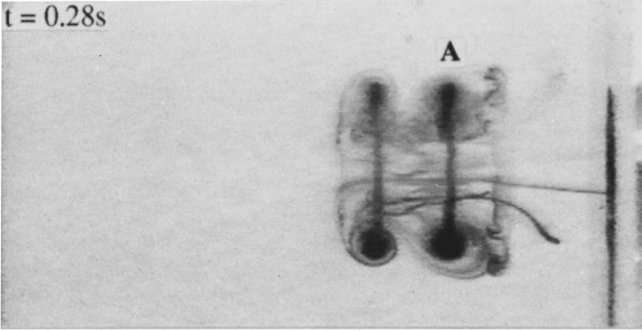}
\caption{Smoke rings frogleaping forward at two frames of time. The unlabled smoke ring slips through the labled one. Image courtesy from \cite{doi:10.1063/1.869160}. If the dark spots are point vortices they should move in similar ways.% A 2D slice representation of the vortex ring a point vortices (right).
}
\label{fig:vortex leap frog}
\end{figure}

Apart from very simple cases it becomes hard to predict the specific motion of many point vortices together. This however should not demotivate us since it is a natural (annoying) habit of fluid motion to be super sensitive and chaotic.

Since point vortices describe a discretization of vorticity $\w$ we can approximate bigger vortex situations and watch how our point vortices behave. This now finally corresponds to more natural fluid simulations rather than just looking at a small number of points vortices. We do this by significantly increasing the amount of point vortices while decreasing the vorticity of each (otherwise they will just hurl each other away).

Our test of choice of this will be the observation of taylor vortices as done extensivly in \cite{Mullen:2009:EIF}. In it separated regions of equal vorticity merge and split together. The same test also used in \cite{azencot2014functional} and is displayed here in figure \ref{fig:taylor_on_hand}.

\begin{figure}
\centering
\def\ll{4cm}
\includegraphics[height=\ll]{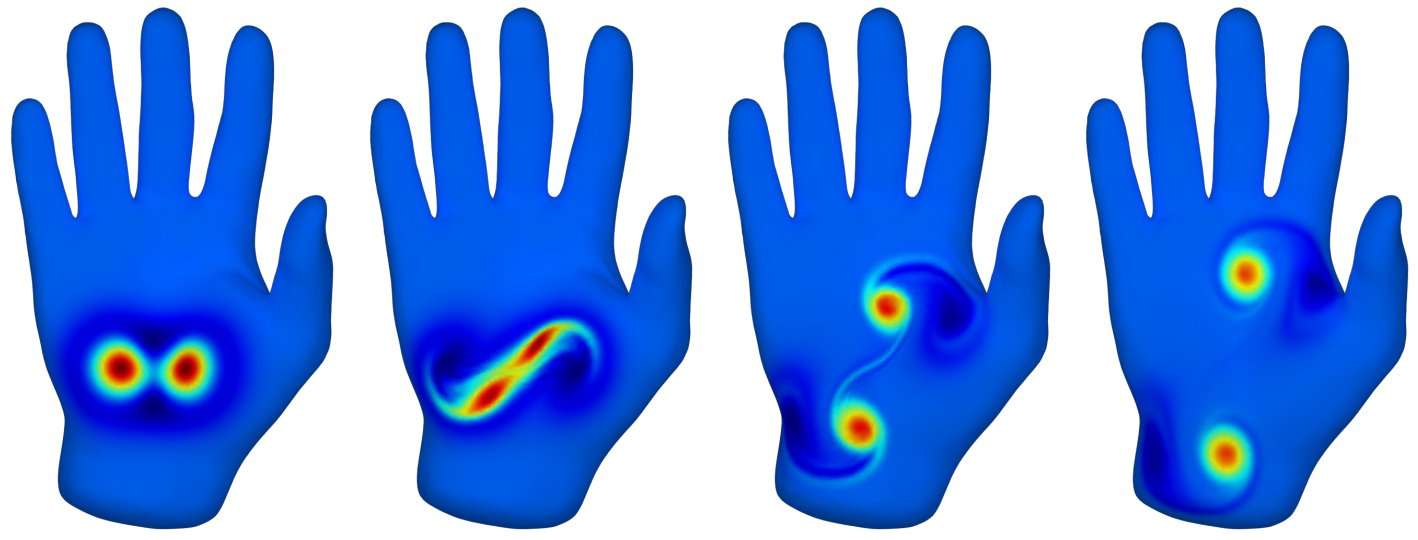}
\caption{Pictures from Ascencot et. al. \cite{azencot2014functional}. Taylor vortices on the surface of the hand first merge and then split. Red indicates areas of high vorticity.}
\label{fig:taylor_on_hand}
\end{figure}

%=============================================================================
%\import{chapters/chapter04/}{chap04.tex}
%\clearemptydoublepage

%+ + + + + + + + + + + + + + + + + + + + + + + + + + + + + + + + + + + + + + + + + + + + + + + + + + + + + + + + + + + + + + + + + + + + + + + + + + + + + + + + + + + + + + + + + + + + + + + + + + + + + + + + + + + + + + + + + + + + + + + + +    NEW CHAPTER     + + + + + + + + + + + + + + + + + + + + + + + + + + + + + + + + + + + + + + + + + + + + + + + + + + + + + + + + + + + + + + + + + + + +

\chapter{\planarvd}
\label{chap:\planarvd}

\begin{figure}[H]
\centering
\begin{tikzpicture}[scale=1.30]
% plane image
\filldraw[fill=gray!10] (0,0) -- ++ (45:2.2) -- ++ (3.3,0) -- ++ (225:2.2) -- cycle;
\draw[very thick,->] (22.5:1.4) -- ++ (0,2) node[left] {$\vn$} ;
\draw[->]   (2,0.5) -- ++ (40:1) node[above right] {$y$};
\draw[->]   (2,0.5) -- ++ (1,0)  node[right] {$x$};
\end{tikzpicture}
\caption{Plane $\R^2$ embedded in $\R^3$. Let's drop some point vortices and watch them flow.}
\label{fig:plane}
\end{figure}
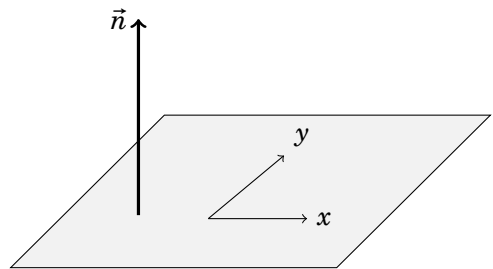

Our first case study will be the plane that has the simplest point vortex dynamics with complete symmetry and zero curvature. To be more consistent with the rest of the theory we describe our plane as a 2D surface through the origin inside 3D space with constant normal vector $\vn=(0,0,1)$. This way we can for example still use cross products as in the streamer function definition \ref{def:2dstream} that made use of $\vn$. However, since such an operation is quite simple:

\begin{equation}
\dddv 0 0 1 \times \dddv x y 0 = \dddv{-y}{x}{0}
\end{equation}

we can just replace $\vn\times(\ )$ with a 90 degree anti-clockwise rotation operator $J = \left( \begin{smallmatrix} \ 0& -1 \\ \ 1 & \ 0 \end{smallmatrix} \right)$ to end up only working in 2 dimensions again.

%=============================================================================
\section{\planarvd Derivation}
\label{sec:\planarvd Derivation}

The great thing now is that we can take our result form section \ref{sec:Point Vortex Dynamics} proposition \ref{prop:Point Vortex Dynamics} \emph{point vortex dynamics} and plug in the Green's function of the plane that we established in section \ref{sec:Greens Functions on Surfaces} proposition \ref{prop:Green's function in the 2D plane}.

Let $M=\R^2$ be the surface with $n\in\N$ point vortices at positions $p_i\in M$ with respective vortex strengths $\w_i$. The dynamics are described by 

$$\vu(p_j) = \sum_{ \substack{i=1 \\ i\neq j} }^n \w_i\sgrad G(p_j,p_i)$$

Remember that the Green's function for $\R^2$ is given by 

$$ G(x,y)=-\frac{1}{2\pi}\ln(|x-y|) $$

Let us explicitly compute the resulting velocity for each point vortex. To do this we need the gradient of the Green's function.

\begin{lemma}
For the Green's function in $\R^2$ where $\R^2$ is embedded as a surface in $\R^3$ we can compute:
$$\sgrad_x G(x,y)=\frac{1}{2\pi}\frac{\vn\times (x-y)}{|x-y|^2}$$
\end{lemma}
\begin{proof}
$\sgrad_x=\vn\times\grado$. Keep in mind that $x,y\in\R^3$ but with the z-component as zero in order to apply our established theory. Also note that 

$$\grado_x\ln(|x-y|)=\frac{x-y}{|x-y|^2}$$

\end{proof}

With this result in mind we can go straight to the induced velocity field of the fluid and the point vortex dynamics:

\begin{theorem}[Planar Point Vortex Velocity Field]
\label{thr:Planar Point Vortex Velocity Field}
Given $n \in \N$ point vortices $p_i$ on the surface $M=\R^2$ with vortex strengths $\w_i\in\R$ respectively, then the velocity $\vu$ at a point $x\in M$ with $\forall i: x\neq p_j$ is expressed by
\begin{equation}
\vu(x) = \frac{1}{2\pi} \sum_{ \substack{i=1 } }^n \w_i \frac{\vn\times (x-p_i)}{|x-p_i|^2}
\end{equation}
Where $\vn$ is the surface normal $(0,0,1)^T$.
\end{theorem}

\begin{theorem}[Planar Point Vortex Dynamics]
\label{thr:\planarvd}
Given $n \in \N$ point vortices $p_i$ on the surface $M=\R^2$ with vortex strengths $\w_i\in\R$ respectively, then the velocity $\vu(p_j)$ of the point vortex at $p_j$ is expressed by
\begin{equation*}
\partio{t} p_j = \vu(p_j) = \frac{1}{2\pi} \sum_{ \substack{i=1 \\ i\neq j} }^n \w_i \frac{\vn\times (p_j-p_i)}{|p_j-p_i|^2}
\end{equation*}
Where $\vn$ is the surface normal $(0,0,1)^T$.
\end{theorem}

These are the results that already Kirchhoff has derived when he initiated the point vortex idea.

%=============================================================================
\section{\planarvd Results}
\label{sec:\planarvd Results}

\def\resultsIntro{provides us with all differential equations we need to simulate the point vortex dynamics on the plane $M=\R^2\subset\R^3$. Implementation details will be saved for later in chapter \ref{chap:Implementation}. The colors chosen represent the level lines of the stream function and are warm and cold colors for positive and negative stream values respectively.}

Theorem \ref{thr:\planarvd} \resultsIntro

Let us take a quick look at the results predicted in section \ref{sec:Results to Observe}. First of all Kimura's conjecture predicts that for $M=\R^2$ that a vortex pair with opposite vorticity should move along a straight line. Figure \ref{fig:planar kimura} shows that this is indeed the case.

\begin{figure}
\centering
\def\locallength{5cm}
\includegraphics[width=\locallength]{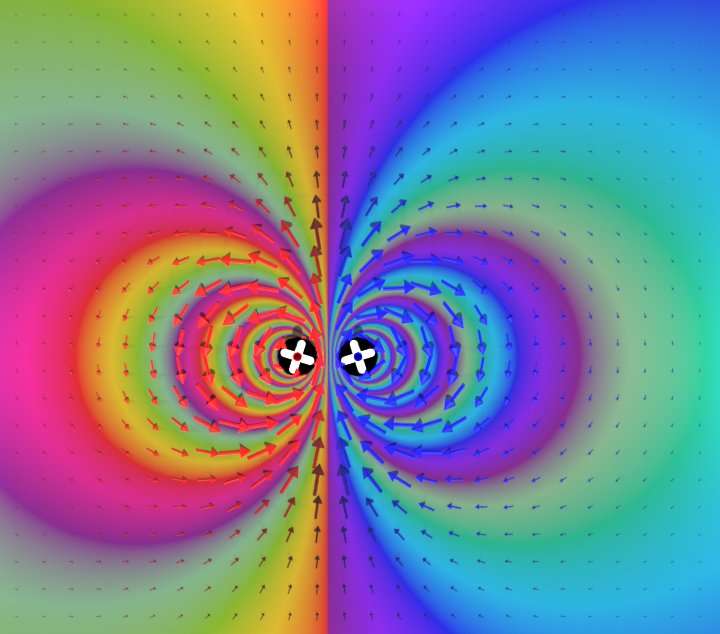}
\includegraphics[width=\locallength]{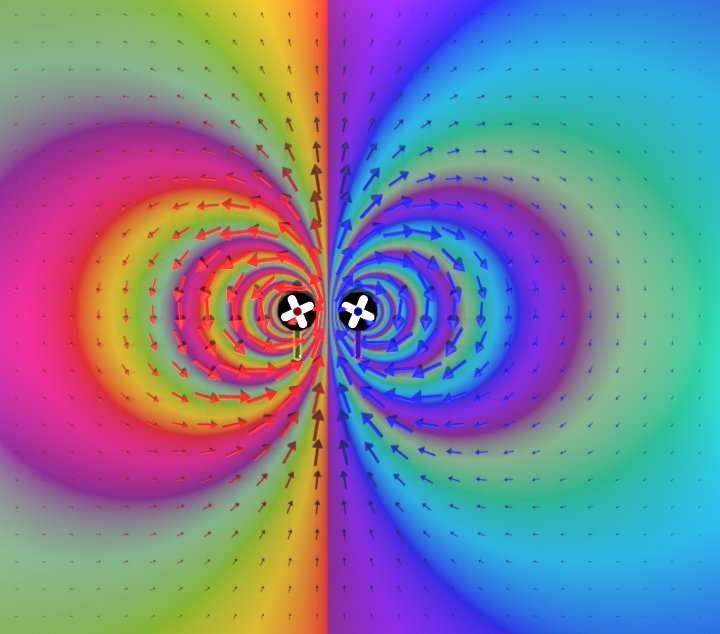}\\
\includegraphics[width=\locallength]{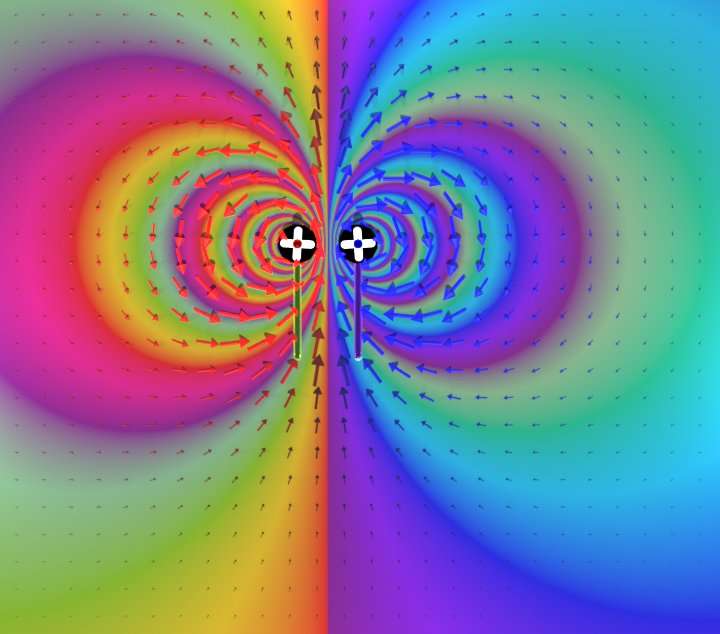}
\includegraphics[width=\locallength]{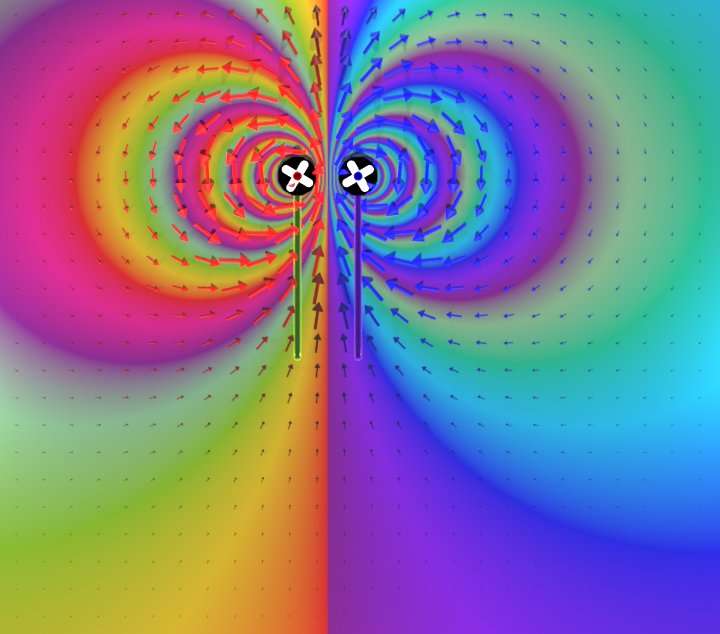}
\caption{Two point vortices with vorticity $+1$ and $-1$ move together on a straight line. (from top left to right)}
\label{fig:planar kimura}
\end{figure}

We can also demonstrate this mathematically to make the reader more familiar with the actual formulas. W.l.o.g. we fix translation and orientation and look at two vortices at $p_1=(1,0,0)^T$ and $p_2=(-1,0,0)^T$ with vortex strength
$\w_1=-1,\w_2=+1$. Then our formulas reveal that the points will move like this:

\begin{align*}
\vu(p_1) &= \frac{1}{2\pi} \sum_{ \substack{i=1 \\ i\neq j} }^n \w_i \frac{\vn\times (p_j-p_i)}{|p_j-p_i|^2}\\
&= \frac{1}{2\pi} \w_2 \frac{\vn\times (p_1-p_2)}{|p_1-p_2|^2}\\
&= \frac{1}{2\pi} (+1) \frac{(0,0,1)^T \times ( (1,0,0)^T-(-1,0,0)^T)}{|2|^2}\\
&= \frac{1}{4\pi}  (0,0,1)^T \times (1,0,0)^T\\
&= \frac{1}{4\pi} (0,1,0)^T
\end{align*}
%\dddv{0,1,0}

One can verify that also $\vu(p_2) = \frac{1}{4\pi} (0,1,0)$ thereby showing that both point vortices move along a straight line (a geodesic in $\R^2$). The closer they become the faster they will move along the geodesic.

We can also observe the leap froggin vortices by setting up four point vortices as seen in figure \ref{fig:planar leapfrog}. This also seems to work well.

\begin{figure}
\centering
\def\locallength{5cm}
\includegraphics[width=\locallength]{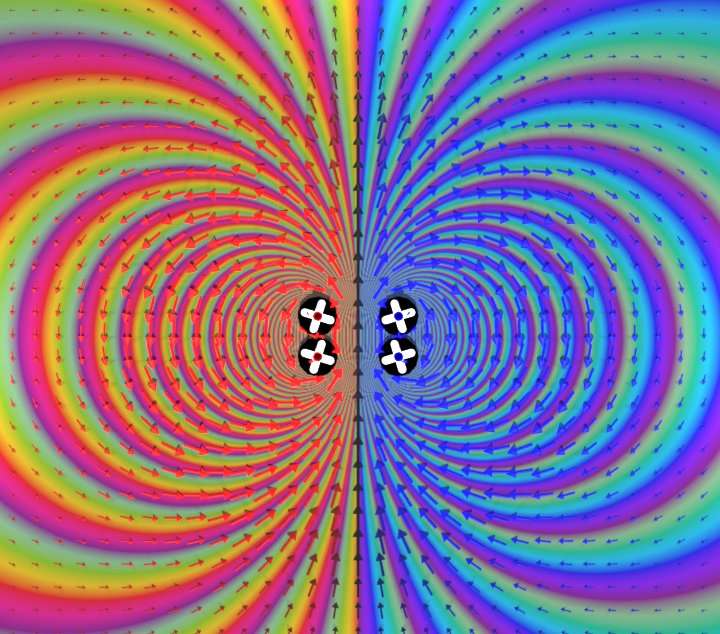}
\includegraphics[width=\locallength]{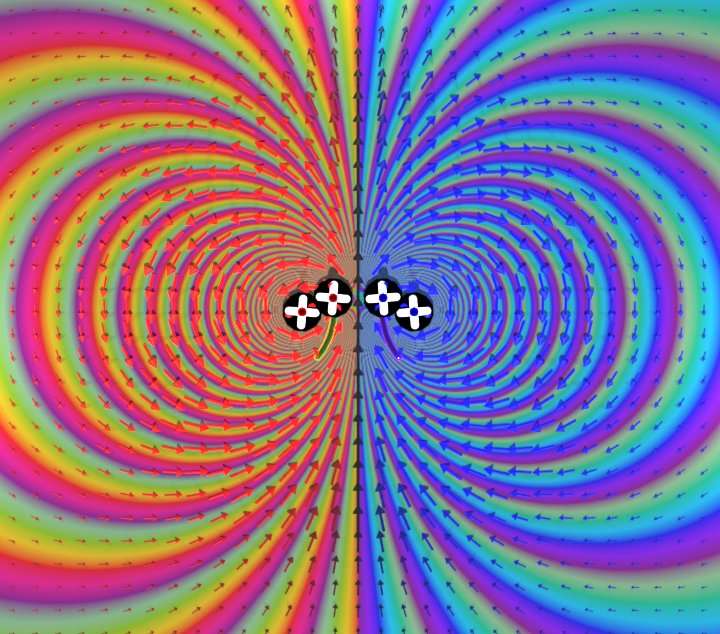}\\
\includegraphics[width=\locallength]{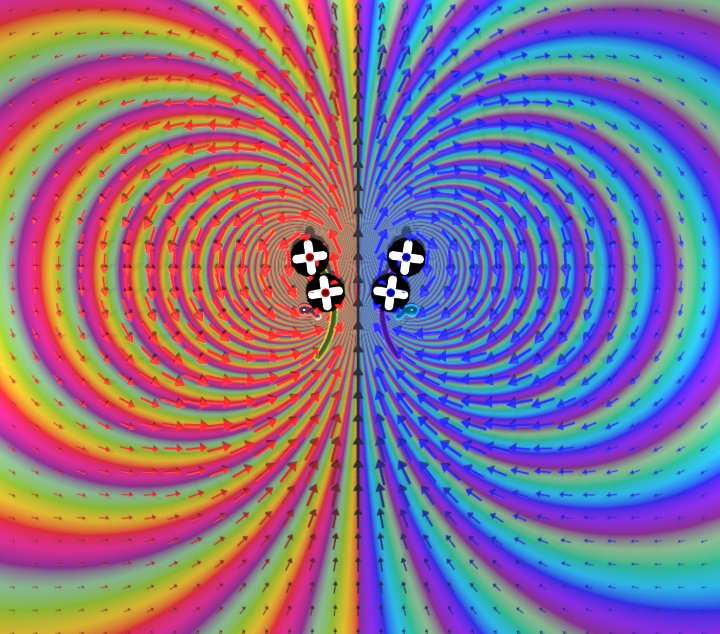}
\includegraphics[width=\locallength]{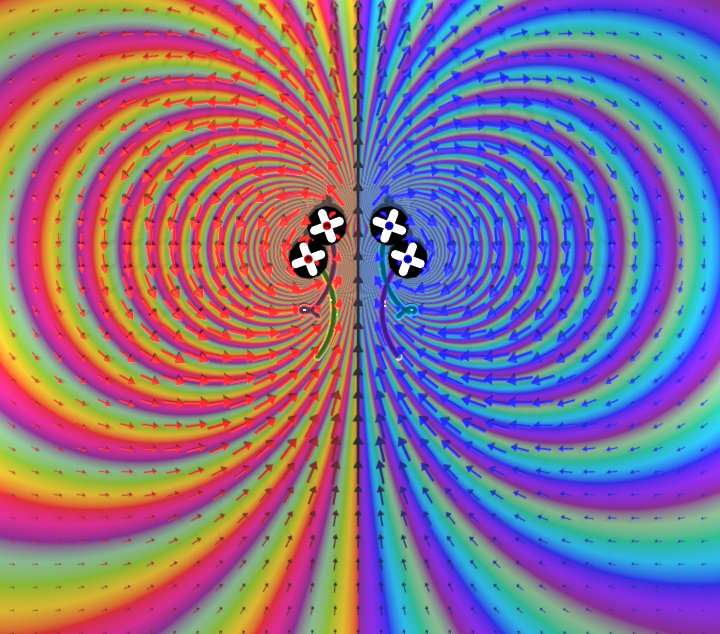}
\caption{4 point vortices in leapfrogging set up with vorticity $+1,+1,-1$ and $-1$. (from top left to right)}
\label{fig:planar leapfrog}
\end{figure}

Once we have more point vortices many effects mix together and it becomes hard to predict anything. Nevertheless we want to show what this looks like when we only add a few more vortices in figure \ref{fig:planar many}.

\begin{figure}
\centering
\def\locallength{5cm}
\includegraphics[width=\locallength]{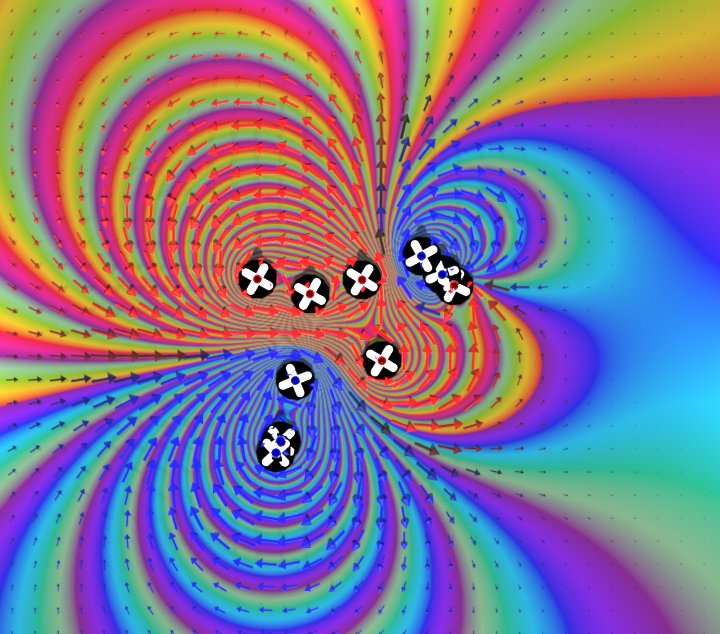}
\includegraphics[width=\locallength]{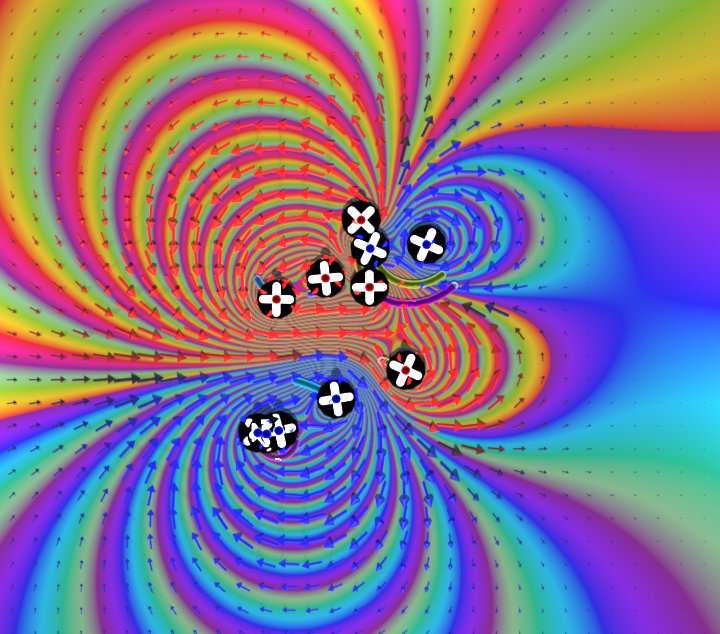}\\
\includegraphics[width=\locallength]{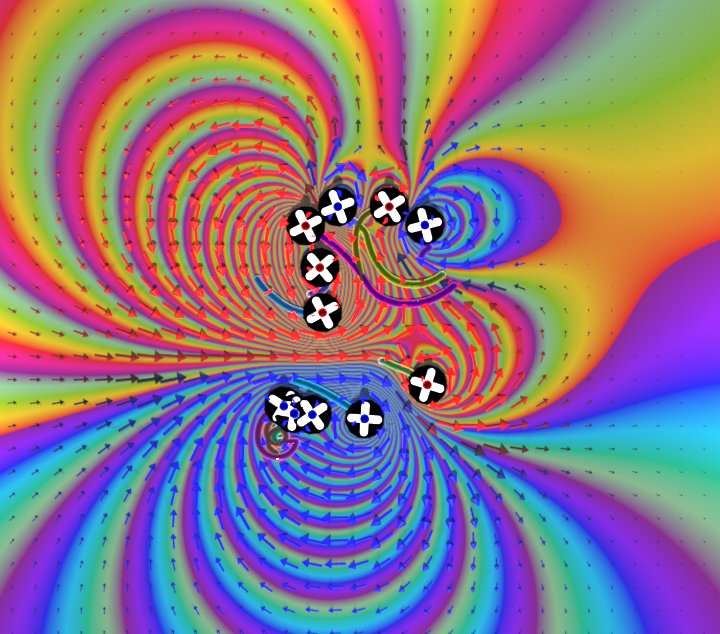}
\includegraphics[width=\locallength]{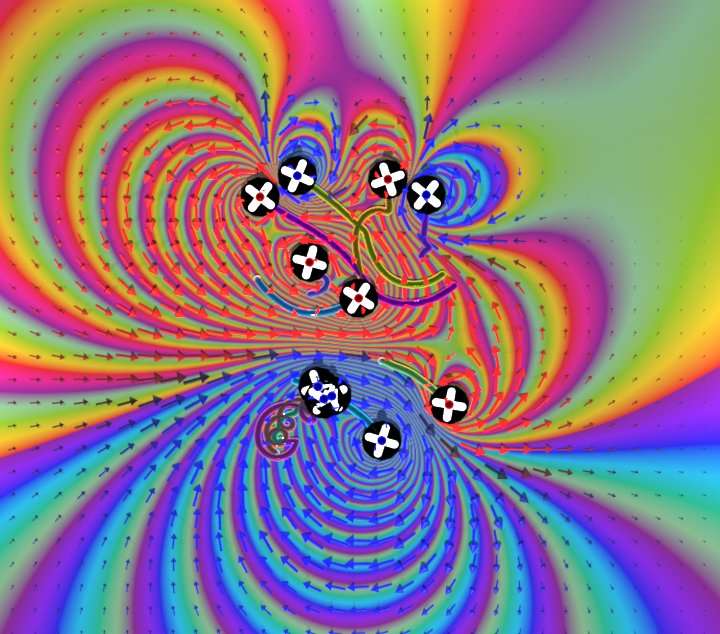}
\caption{Many point vortices with uniformly random vorticties in [-1,+1]. (from top left to right)}
\label{fig:planar many}
\end{figure}

And at last we want to show how many many point vortices look together by constructing the Taylor vortices experiment. We observe that the method of using point vortices like this won't reveal the typical splitting of the newly formed vorticity blob. As seen in figure \ref{fig:taylor plane} shows this. The phenomena can not be reproduced because no new point vortices are created that really flatten out the vorticity. This is however nothing to be sad about. Many fluid solvers fail this test yet we still get visually appealing motion at a cheap price.

\begin{figure}
\centering
\def\locallength{0.24\textwidth}
\includegraphics[width=\locallength]{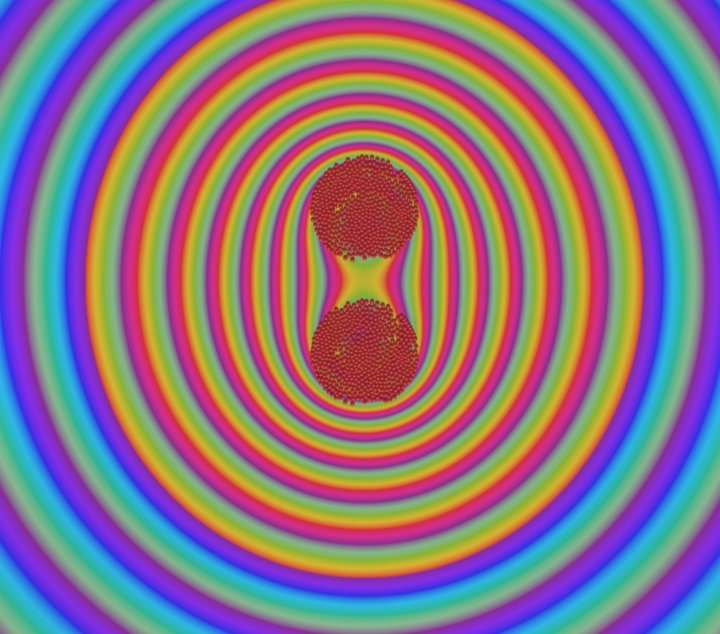}
\includegraphics[width=\locallength]{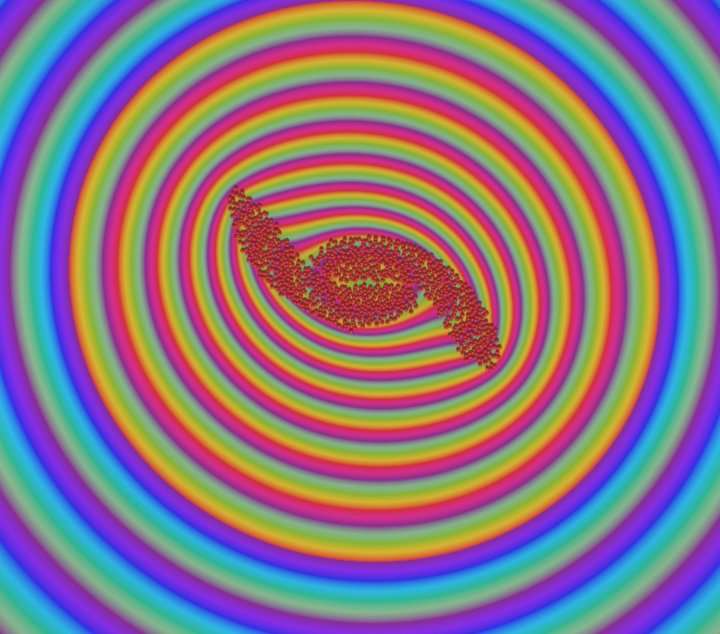}
\includegraphics[width=\locallength]{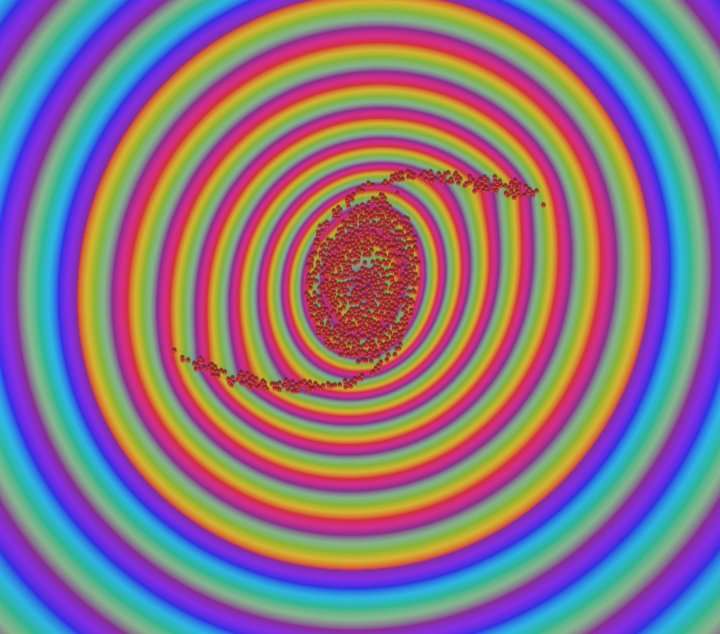}
\includegraphics[width=\locallength]{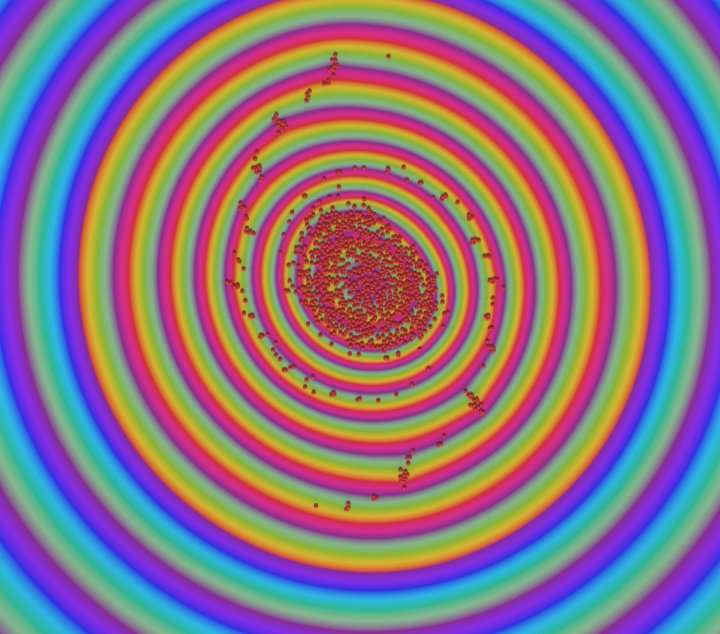}
\caption{Many point vortices with equally small positive vorticty forming the taylor vortex experiment. (from left to right)}
\label{fig:taylor plane}
\end{figure}

%=============================================================================
%\section{\planarvd Discussion}
%\label{sec:\planarvd Discussion}

%=============================================================================
%\import{chapters/chapter05/}{chap05.tex}
%\clearemptydoublepage

%+ + + + + + + + + + + + + + + + + + + + + + + + + + + + + + + + + + + + + + + + + + + + + + + + + + + + + + + + + + + + + + + + + + + + + + + + + + + + + + + + + + + + + + + + + + + + + + + + + + + + + + + + + + + + + + + + + + + + + + + + +    NEW CHAPTER     + + + + + + + + + + + + + + + + + + + + + + + + + + + + + + + + + + + + + + + + + + + + + + + + + + + + + + + + + + + + + + + + + + + +

\chapter{\sphericalvd}
\label{chap:\sphericalvd}

\begin{figure}[H]
\centering
\def\svgwidth{0.8\textwidth}
\begin{tikzpicture}[decoration={markings,mark=at position .5 with {\arrow{latex'}}}]
%\filldraw[color=white] (0,0) circle (1.8cm);
\filldraw[ball color=white] (0,0) circle (1.8cm);
%\foreach \rx in {-1,-.6,...,1}{
%\draw[densely dashed,very thin,postaction={decorate}] (0,1.2) arc (90:-90:{\rx} and 1.2);
%draw[densely dashed,very thin,postaction={decorate}] (-1.2,0) arc (180:0:{1.2} and {\rx});}
\end{tikzpicture}
\caption{The $\Sp$ sphere embedded in $\R^3$.}
\label{fig:sphere}
\end{figure}

We move on to the second case study where we sacrifice the property of zero curvature as seen in the plane but we still maintain full symmetry. This time we have constant curvature and the property that the normal at a point $x\in\Sp$ is the point itself ( $\vn=x$ ). 

%=============================================================================
\section{\sphericalvd Derivation}
\label{sec:\sphericalvd Derivation}

The way we prepared for this moment allows us to reach the dynamics completely analogously to the planar case. We can take our result form section \ref{sec:Point Vortex Dynamics} proposition \ref{prop:Point Vortex Dynamics} \emph{point vortex dynamics} and plug in the Green's function of $\Sp$ that we established in section \ref{sec:Greens Functions on Surfaces} proposition \ref{prop:Green's function on the sphere}.

Let $M=\Sp$ be the surface with $n\in\N$ point vortices at positions $p_i\in M$ with respective vortex strengths $\w_i$. Then the dynamics are described by 

$$\vu(p_j) = \sum_{ \substack{i=1 \\ i\neq j} }^n \w_i\sgrad G(p_j,p_i)$$

Remember that the Green's function for $\Sp$ is given by 

$$ G(x,y)=-\frac{1}{2\pi}\ln \left(\ \sin\left(\frac{1}{2}d(x,y) \right)\ \right ) $$

Where $d(x,y)=\acos(x\cdot y)$. Taking the gradient of this Green's function is less comfortable but we have no choice:

\begin{lemma}
For the Green's function on $\Sp$ we can compute:
$$\sgrad_x G(x,y) =  \frac{1}{4\pi}\frac{x \times y}{ 1 -  x\cdot y} $$
\end{lemma}
\begin{proof}
Let us show this explicitly for coordinate $i\in \{1,2,3\}$ representing the axes:

\begin{align*}
\partial_i \ln \left(\ \sin\left(\tfrac{1}{2} \acos(x\cdot y) \right)\ \right ) &= \frac{1}{\ \sin\left(\tfrac{1}{2} \acos(x\cdot y) \right)\ } \partial_i \left(\ \sin\left(\tfrac{1}{2} \acos(x\cdot y) \right)\ \right )\\
&= \frac{1}{\ \sin\left(\tfrac{1}{2} \acos(x\cdot y) \right)\ }  \cos\left(\tfrac{1}{2} \acos(x\cdot y) \right ) \  \partial_i ( \tfrac{1}{2} \acos(x\cdot y) )\\
&= \frac{\cos\left(\tfrac{1}{2} \acos(x\cdot y) \right )}{\ \sin\left(\tfrac{1}{2} \acos(x\cdot y) \right)\ } \tfrac{1}{2} \  \partial_i   \acos(x\cdot y) \\
&= \frac{\cos\left(\tfrac{1}{2} \acos(x\cdot y) \right )}{\ \sin\left(\tfrac{1}{2} \acos(x\cdot y) \right)\ } \tfrac{1}{2}   \frac{-1}{ \sqrt{1-(x\cdot y)^2} } \partial_i(x\cdot y) \\
&= \frac{\cos\left(\tfrac{1}{2} \acos(x\cdot y) \right )}{\ \sin\left(\tfrac{1}{2} \acos(x\cdot y) \right)\ } \frac{-\tfrac{1}{2}}{ \sqrt{1-(x\cdot y)^2} } y_i \\
&= \frac{\cos\left(\tfrac{1}{2} d(x,y) \right )}{\ \sin\left(\tfrac{1}{2} d(x,y) \right)\ } \frac{-\tfrac{1}{2}}{ \sqrt{1-\cos(d(x,y))^2} } y_i \\
&= \frac{\cos\left(\tfrac{1}{2} d(x,y) \right )}{\ \sin\left(\tfrac{1}{2} d(x,y) \right)\ } \frac{-\tfrac{1}{2}}{ \sin(d(x,y)) } y_i \\
&= -\tfrac{1}{2}y_i\frac{\cos\left(\tfrac{1}{2} d(x,y) \right )}{\ \sin\left(\tfrac{1}{2} d(x,y) \right)\ \sin(d(x,y))}  \\
&= -\tfrac{1}{2}y_i\frac{\cos\left(\tfrac{1}{2} d(x,y) \right )}{ \cos(\tfrac{1}{2}d(x,y)) -  \cos(d(x,y))\cos(\tfrac{1}{2}d(x,y))}  \\
&= -\tfrac{1}{2}y_i\frac{1}{ 1 -  \cos(d(x,y))}  \\
&= -\frac{y_i}{2}\frac{1}{ 1 -  x\cdot y}  \\
\end{align*}

where we used the cosine sum formula $\cos(\tfrac{1}{2}d) = \cos(d - \tfrac{1}{2}d)=\cos d \cos \tfrac{1}{2}d + \sin d\sin \tfrac{1}{2}d$. Thus

$$\grado_x G(x,y) =  \frac{y}{ 4\pi(1 -  x\cdot y)}  $$

At last we use the relation $\sgrad = \vn \times \grado$ with $\vn(x)=x$ on $\Sp$.

$$\sgrad_x G(x,y) =  \frac{1}{4\pi}\frac{x \times y}{ 1 -  x\cdot y} $$

\end{proof}

We then only have to throw in this last lemma into the point vortices theorems:

\begin{theorem}[Spherical Point Vortex Velocity Field]
\label{thr:Spherical Point Vortex Velocity Field}
Given $n \in \N$ point vortices $p_i$ on the surface $M=\Sp$ with vortex strengths $\w_i\in\R$ respectively, then the velocity $\vu$ at a point $x\in M$ with $\forall i: x\neq p_j$ is expressed by
\begin{equation}
\vu(x) = \frac{1}{4\pi} \sum_{ \substack{i=1 } }^n \w_i \frac{x \times p_i}{ 1 -  x\cdot p_i}
\end{equation}
\end{theorem}

\begin{theorem}[Spherical Point Vortex Dynamics]
\label{thr:\sphericalvd}
Given $n \in \N$ point vortices $p_i$ on the surface $M=\Sp$ with vortex strengths $\w_i\in\R$ respectively, then the velocity $\vu(p_j)$ of the point vortex at $p_j$ is expressed by
\begin{equation*}
\vu(p_j) = \frac{1}{4\pi} \sum_{  \substack{i=1 \\ i\neq j}}^n \w_i \frac{p_j \times p_i}{ 1 -  p_j\cdot p_i}
\end{equation*}
\end{theorem}

These formulas are the same as established in \citep{Dritschel-2015}.

%=============================================================================
\section{\sphericalvd Results}
\label{sec:\sphericalvd Results}

Theorem \ref{thr:\sphericalvd} provides the key differential equations we need describe the point vortex dynamics on the sphere $M=\Sp\subset\R^3$. Its Implementation will be described in chapter \ref{chap:Implementation}. The colors chosen represent the level lines of the stream function and are warm and cold colors for positive and negative stream values respectively.

Let us take a quick look at the results predicted in section \ref{sec:Results to Observe}. First of all Kimura's conjecture predicts that for $M=\Sp$ an infinitesimally close vortex pair with opposite vorticity should move along great circles on the sphere. Figure \ref{fig:sphere kimura} shows that this is indeed the case.

\begin{figure}
\centering
\def\locallength{0.24\textwidth}
\includegraphics[width=\locallength]{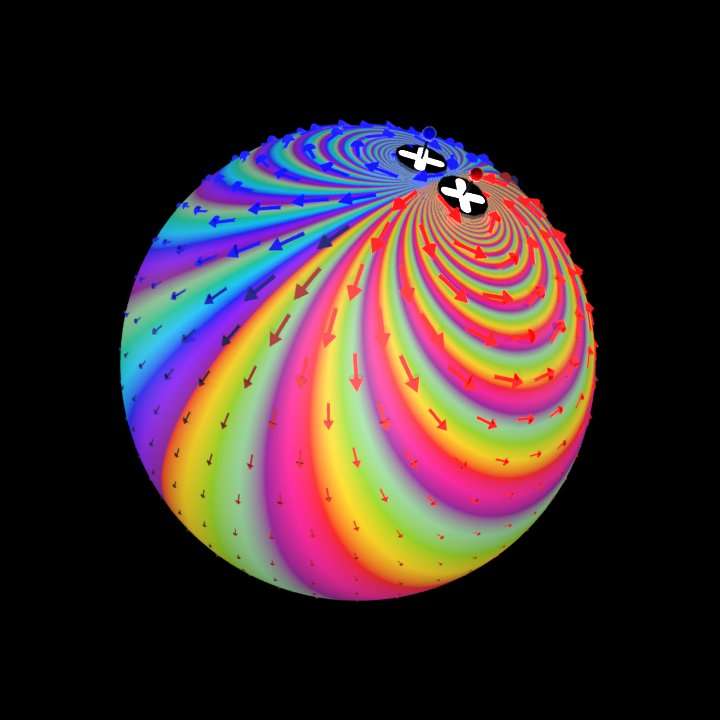}
\includegraphics[width=\locallength]{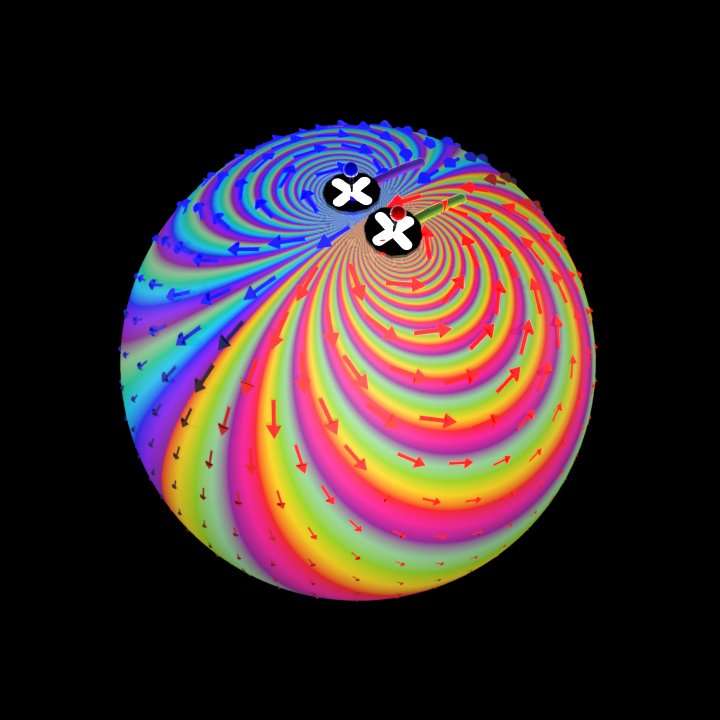}
\includegraphics[width=\locallength]{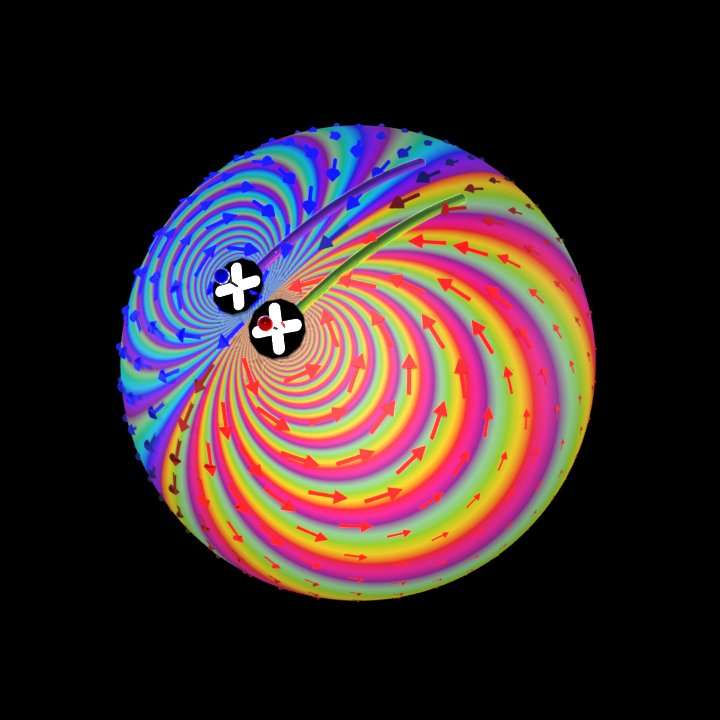}
\includegraphics[width=\locallength]{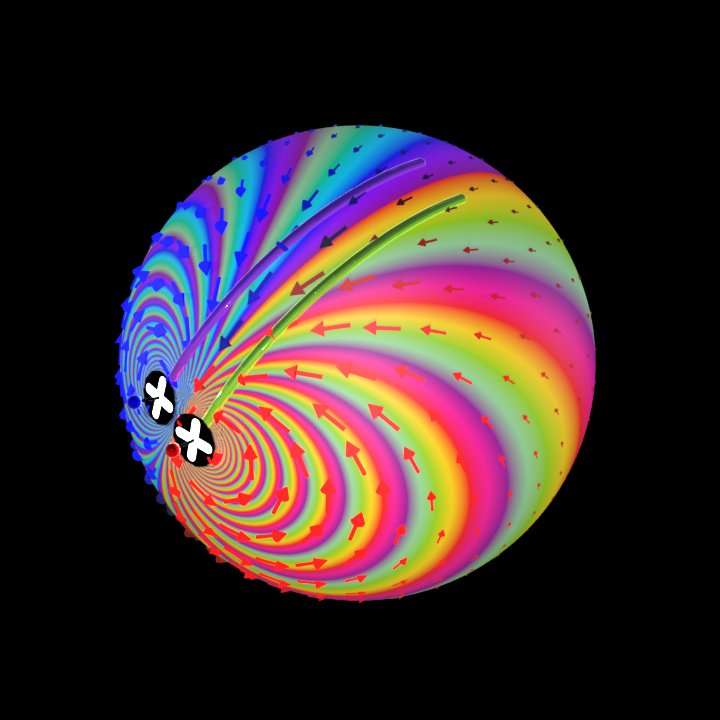}
\caption{Two point vortices with opposite vorticity move together on a great arc. (from left to right)}
\label{fig:sphere kimura}
\end{figure}

Mathematically this geodesic behavior is easy to verify. Since we only have $p_1, p_2\in \Sp$ our velocities of our point vortices will point in the direction of $p_2\times p_1$.

\begin{align*}
\vu(p_1) &= \frac{1}{4\pi} \sum_{  \substack{i=1 \\ i\neq j}}^n \w_i \frac{p_j \times p_i}{ 1 -  p_j\cdot p_i}\\
&= \frac{\w_2}{4\pi( 1 -  p_2\cdot p_1)}  p_2 \times p_1\\
&= \alpha p_2 \times p_1\\
\vu(p_2) &=\frac{\w_1}{4\pi( 1 -  p_1\cdot p_2)}  p_1 \times p_2\\
&= \frac{-\w_2}{4\pi( 1 -  p_1\cdot p_2)} (-p_2 \times p_1) = \alpha p_2 \times p_1
\end{align*}

Note that since $p_1\neq p_2$ but they still are infinitesimally close we have a well defined plane containing $p_1,p_2$. $p_2 \times p_1$ being perpendicular to the this plane is the direction that the points will move along just like on a great arc.

We can also observe the leap froggin vortices by setting up four point vortices as seen in figure \ref{fig:planar leapfrog}. This also seems to work well.

\begin{figure}
\centering
\def\locallength{0.24\textwidth}
\includegraphics[width=\locallength]{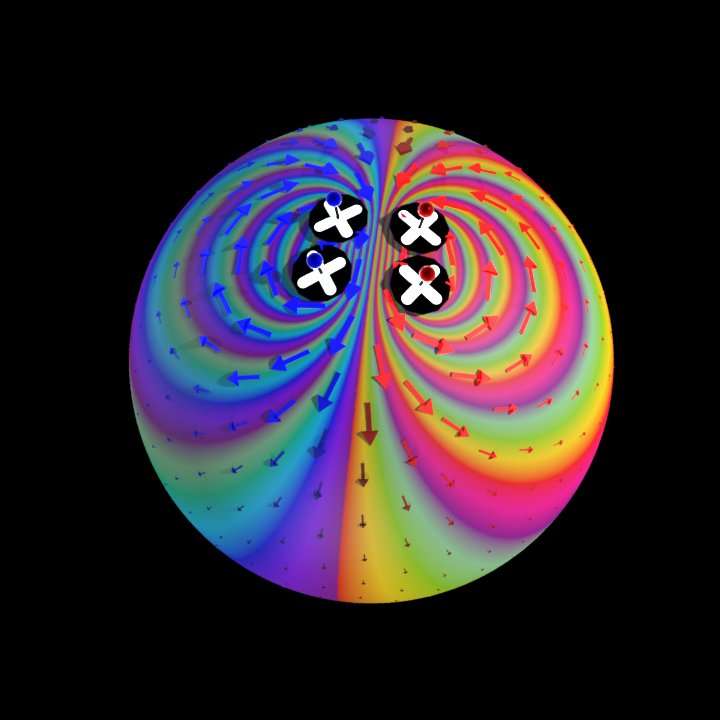}
\includegraphics[width=\locallength]{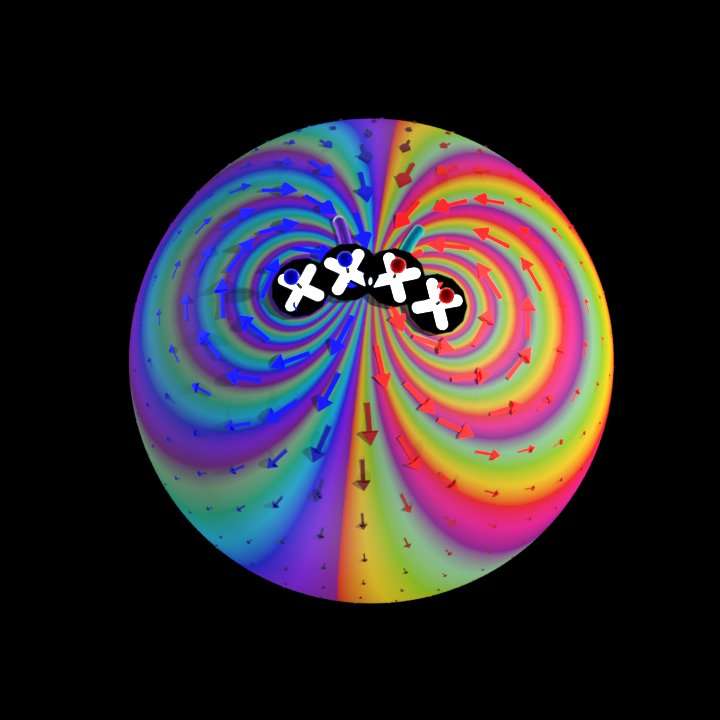}
\includegraphics[width=\locallength]{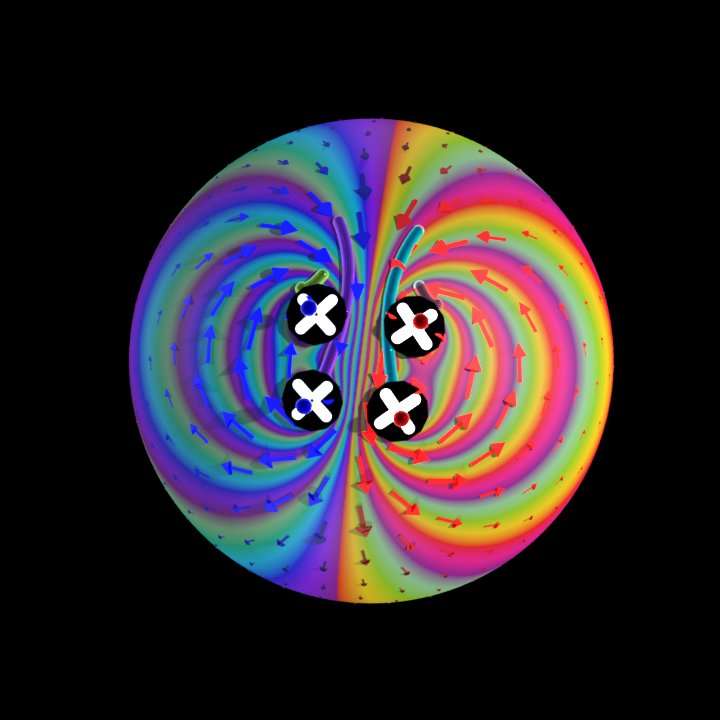}
\includegraphics[width=\locallength]{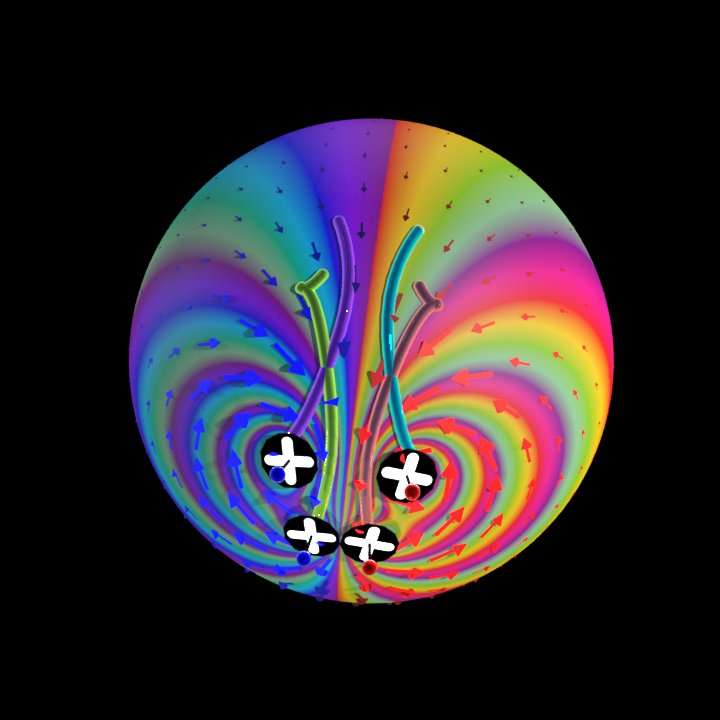}
\caption{4 point vortices in leapfrogging on $\Sp$. (from left to right)}
\label{fig:sphere leapfrog}
\end{figure}

Also in $\Sp$ once we have more point vortices many effects mix together and it becomes hard to predict anything. Nevertheless we want to show what this looks like with just a few more points in figure \ref{fig:sphere many}.

\begin{figure}
\centering
\def\locallength{0.24\textwidth}
\includegraphics[width=\locallength]{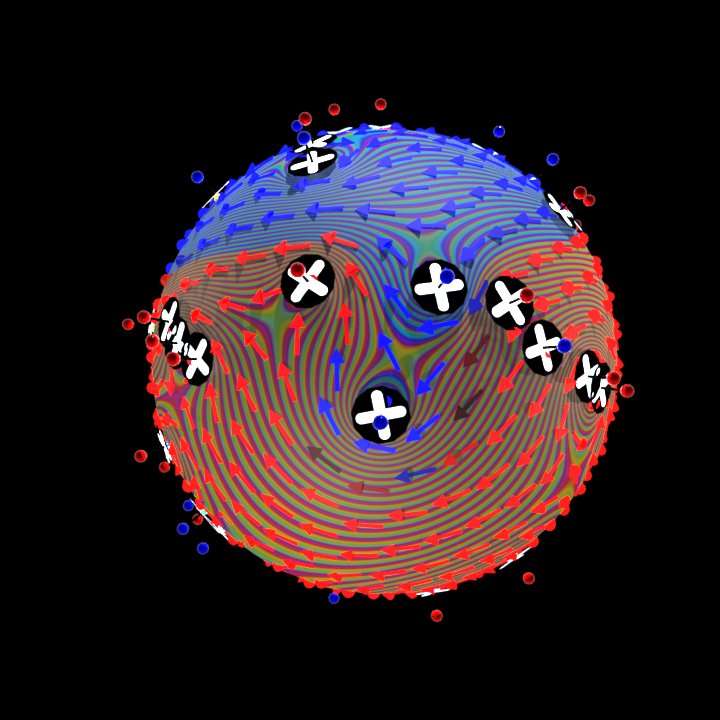}
\includegraphics[width=\locallength]{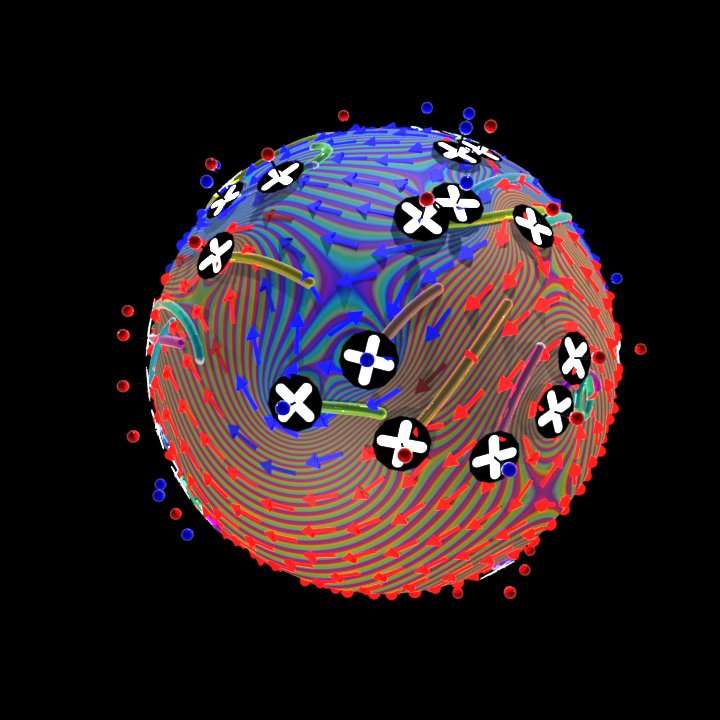}
\includegraphics[width=\locallength]{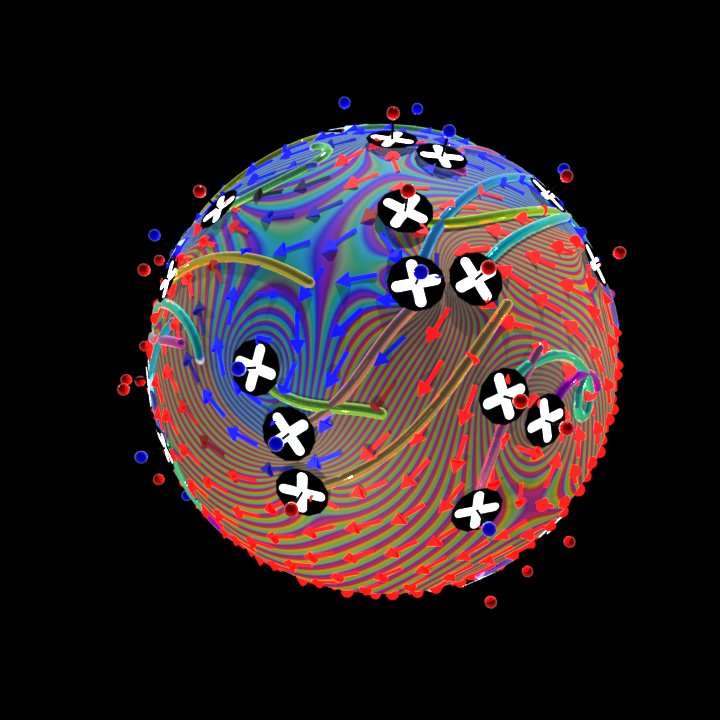}
\includegraphics[width=\locallength]{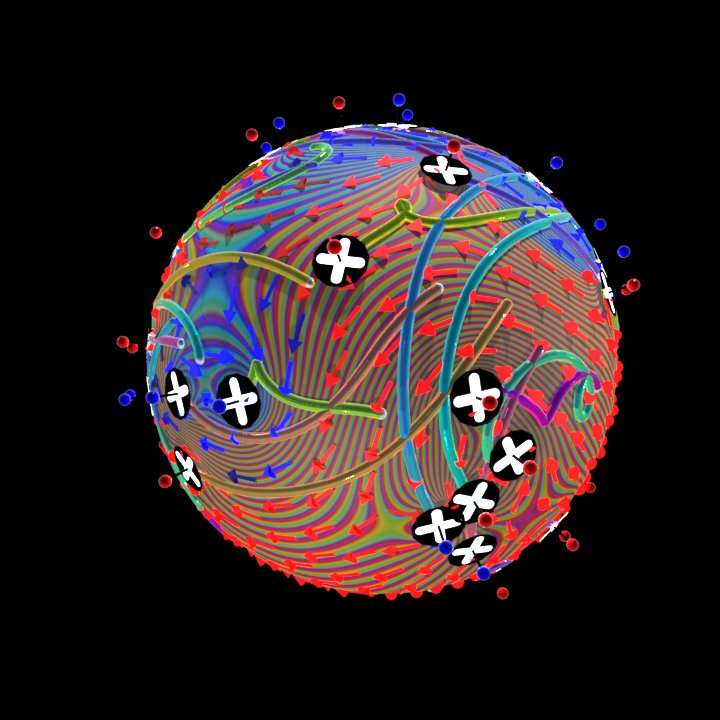}
\caption{Many point vortices on $\Sp$ with uniformly random vorticties in [-1,+1]. (from left to right)}
\label{fig:sphere many}
\end{figure}

At last we take a close look at the taylor vortices on the sphere (figure \ref{fig:taylor plane}). As in the planar case we observe that the method of using point vortices like this won't reveal the typical splitting. The reasons are the same. However, this shows that the spherical dynamics are still holding up.

\begin{figure}
\centering
\def\locallength{0.24\textwidth}
\includegraphics[width=\locallength]{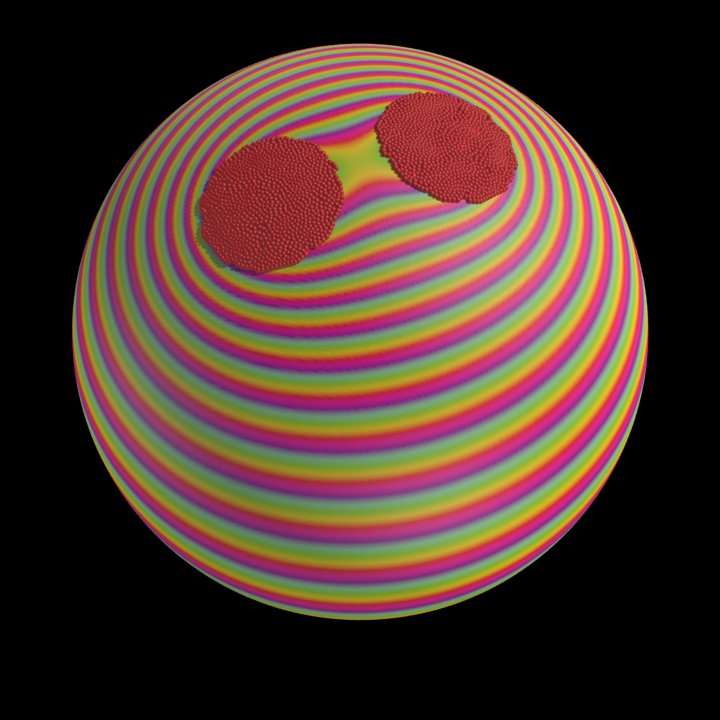}
\includegraphics[width=\locallength]{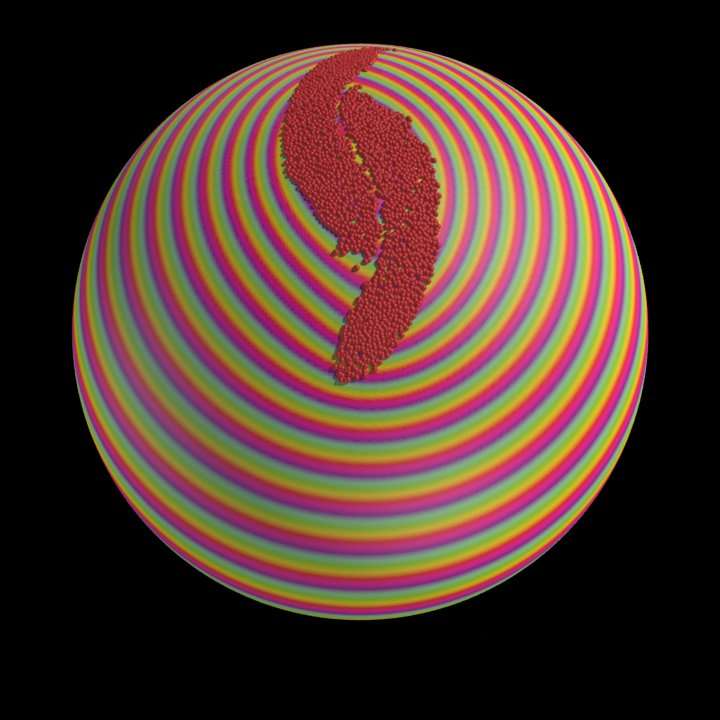}
\includegraphics[width=\locallength]{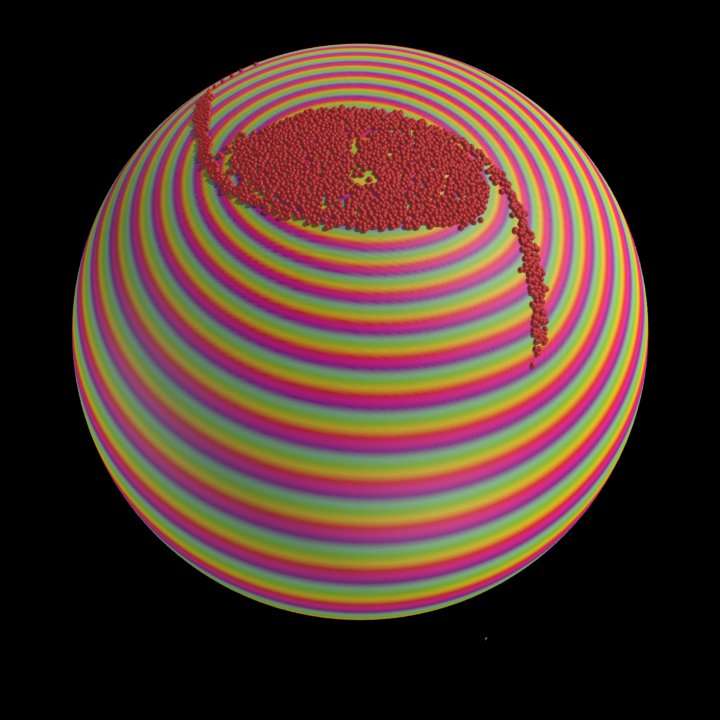}
\includegraphics[width=\locallength]{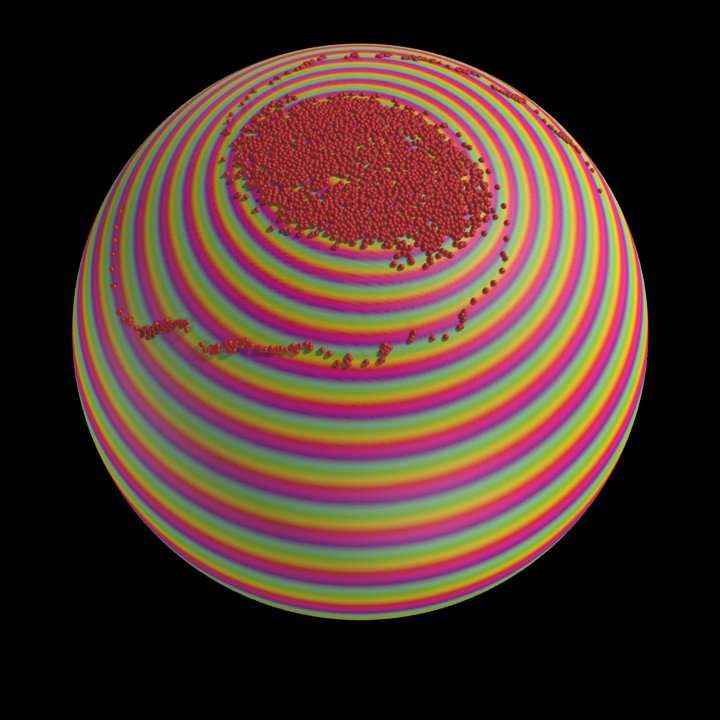}
\caption{Many point vortices with equal small positive vorticty forming the taylor vortex experiment. (from left to right)}
\label{fig:taylor sphere}
\end{figure}

%=============================================================================
%\section{\sphericalvd Discussion}
%\label{sec:\sphericalvd Discussion}
%=============================================================================
%\import{chapters/chapter06/}{chap06.tex}
%\clearemptydoublepage

%+ + + + + + + + + + + + + + + + + + + + + + + + + + + + + + + + + + + + + + + + + + + + + + + + + + + + + + + + + + + + + + + + + + + + + + + + + + + + + + + + + + + + + + + + + + + + + + + + + + + + + + + + + + + + + + + + + + + + + + + + +    NEW CHAPTER     + + + + + + + + + + + + + + + + + + + + + + + + + + + + + + + + + + + + + + + + + + + + + + + + + + + + + + + + + + + + + + + + + + + +

\chapter{\generalvd}
\label{chap:\generalvd}

\begin{figure}[H]
\centering
\def\svgwidth{0.8\textwidth}
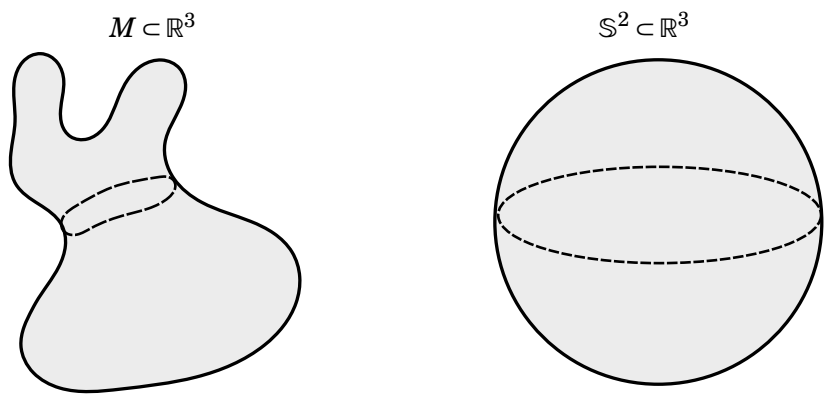
\caption{Two closed surfaces with genus 0.}
\label{fig:some closed surface}
\end{figure}

Here we arrive at the final stop of our ride. This is by far the hardest case to handle and will require the most tricks to overcome. All of the background knowledge and planar and spherical cases where handled in this \thesis in order to make these final steps much easier.

Take a look at figure \ref{fig:some closed surface}. How can we possibly pay respect to all these arbitrary deformations that were absent in the plane and on the sphere? The key point to solve this problem is to use a conformal map $f$ from the arbitrary closed surface $M$ of genus 0 to $\Sp$ to express the point vortex dynamics in a modified form on the sphere. This technique was formally published by Boatto et al. \cite{Boatto2015}. Section \ref{sec:Conformal Mappings} \emph{conformal mappings} explains all we need to know about conformal mappings right now and shows that we can rely on their existence.

%=============================================================================

\section{\generalvd Derivation}
\label{sec:\generalvd Derivation}

Let $M$ be the closed surface we want to run our point vortices on. If you are given a smooth parametrization of $M$ and can derive the Green's function on it then congratulations, you can skip most of this section and just repeat the steps analogous to the planar and spherical case using proposition \ref{prop:Point Vortex Dynamics} to derive the equation of motion for the point vortices. Sadly, this won't be the case for most surfaces which is why the following method will be useful.

Do not despair! Conformal mappings will provide us with a method that allows us to compute the point vortex system energy directly without explicitly computing the Green's functions. The following is only made possible by Stefanella Boatto 
and Jair Koiller's result published in 2015 \cite{Boatto2015}. We will recall it in detail to make use of it.

%=========================QUOTE
\begin{theorem}[``Conformal Metrics'', \cite{Boatto2015} ]
\label{thr:Conformal Metrics}
Within the context of $n\in\N$ point vortices $p_i$ on surfaces with strengths $\w_i$, consider two metrics in the conformal class of $\Sp$, related by a conformal factor $h$, i.e., $\tg=h^2g$ and 2-form $\Omega_{\tg}$. The Hamiltonian $\tH$ for the vortex system in the metric $\tg$ can be obtained from the Hamiltonian $H$ in the metric $g$ by adding two terms:

\begin{equation}
\tH=H - \frac{1}{4\pi}\sum_{i=1}^n\w_i^2\log(h(p_i)) - \left( \sum_{i=1}^n\w_i \right)\frac{1}{\tilde{A}(\Sp)}\sum_{i=1}^n\w_i\Delta_g^{-1}h^2(p_i)
\end{equation}

$\tilde{A}(\Sp)$ describes the area in the $\tg$ norm. For now we will call $\tH$ the \emph{metric Hamiltonian}.
\end{theorem}
%=======================ENDQUOTE
\begin{proof}
Given in \cite{Boatto2015} pages 13-15.
\end{proof}

This is where we are forced to insert a very important request for our point vortices, that \emph{the total vorticity $\int_M \w(x)dx = 0$ on our surface vanishes}.

\begin{remark}[Zero Vortex Sum Demand]
In our point vortex system with $n\in\N$ point vortices with strengths $\w_i$ respectively we require the total vorticity to vanish, which in our case means that the vorticity sum must vanish.

\begin{equation}
0 = \int_M \w(x) dx  = \int_M \sum_{i=1}^n \w_i\delta(x-p_i)dx = \sum_{i=1}^n \w_i \int_M \delta(x-p_i)dx = \sum_{i=1}^n\w_i
\end{equation}

We perform this in order to simplify the usage of the theorem \ref{thr:Conformal Metrics} to:

\begin{equation}
\label{eq:metric Hamiltonian}
\tH=H - \frac{1}{4\pi}\sum_{i=1}^n\w_i^2\log(h(p_i))
\end{equation}

\end{remark}

Let's combine our knowledge in the following proposition

\begin{theorem}[Closed Surface Point Vortex Dynamics]
\label{prop:\generalvd}
Let $f$ be a conformal map from $M$ to $\Sp$ and $h$ the conformal factor of this map. We are given $n \in \N$ point vortices $q_i$ on the closed euclidean surface $M\subset\R^3$ of genus zero with vortex strengths $\w_i\in\R$ respectively. Let $p_i:=f(q_i)\in\Sp$ be the images of these points and let $\sum_{i=1}^n \w_i=0$.  Then the velocity $\vu(p_j)$ of the point vortex at $p_j$ in $Sp$ is expressed by
\begin{equation}
\vu(p_j) = \frac{1}{4\pi h(p_j)^2} \left( \sum_{  \substack{i=1 \\ i\neq j}}^n \left( \w_i \frac{p_j \times p_i}{ 1 -  p_j\cdot p_i} \right)  +   \frac{1}{h(p_j)} \w_jp_j\times\grado_{p_j}h(p_j) \right)
\end{equation}
\end{theorem}

\begin{proof}

\begin{figure}
\centering
\def\svgwidth{1.0\textwidth}
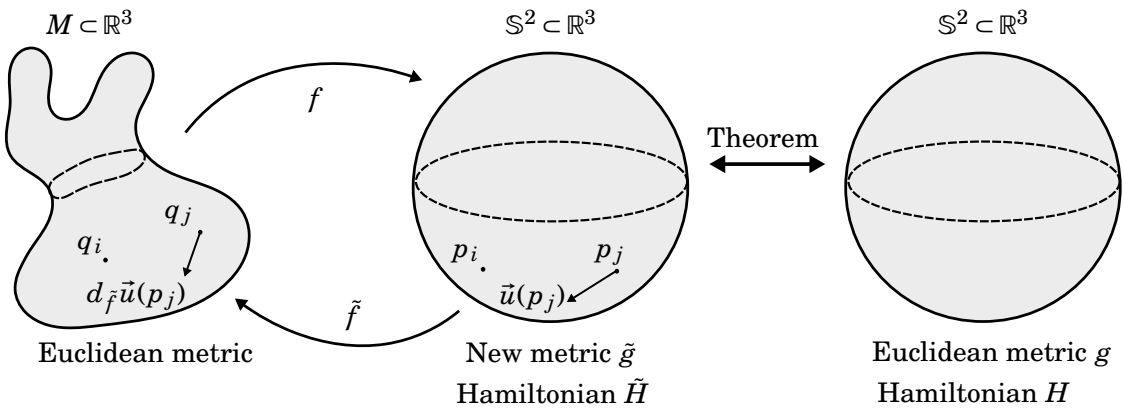
\caption{An illustration of the transports carried out to reach our results. $f$ is the conformal map from $M$ to $\Sp$. The vortex dynamics all happen on $\Sp$ with the influence of the conformal factor $h$ and the result is carried back to $M$.}
\label{fig:mappings with points}
\end{figure}

Take a look at figure \ref{fig:mappings with points} to understand the set up. With the Hamiltonian from equation \ref{eq:metric Hamiltonian} in place it becomes clear now why we need conformal maps. We want to use this Hamiltonian to derive the equations of motion for the point vortices using the previously derived method from Theorem \ref{thr:Point Vortex Energy Motion} \emph{point vortex energy motion} from section \ref{sec:Point Vortex Energy} for the point velocities.

$$\vu(p_j) = -\frac{1}{\w_j} \sgrad_{p_j} E$$ 

But be aware! $\sgrad_{p_j}$ is not the same given any metric. We first map our conformal surface $M$ to the sphere $\Sp$ using $f$ and can rely on the existence of $f$ through the uniformization theorem. And in order to combine both results we need the conformal factor $h$ of the map $f$. In fact, in our case we must adjust $\sgrad_{p_j}$ to be used on $\Sp$ like this:

$$ \sgrad \longrightarrow \tsgrad := \frac{1}{h^2}\sgrad  $$

The arguments for this can be followed in section \ref{sec:Symplectic Manifolds and Symplectic Gradients} and remark \ref{thr:Symplectic Gradient after Conformal Mapping}. Consequently, the metric Hamiltonian from equation \eqref{eq:metric Hamiltonian} shows us that we can directly recycle the result from the spherical vortex dynamics.

$$\underbrace{\tH}_{\text{metric Hamiltonian}}= \underbrace{H}_{\text{sphere Hamiltonian}} - \underbrace{\frac{1}{4\pi}\sum_{i=1}^n\w_i^2\log(h(p_i))}_{\text{new contribution}}$$

Here the spherical vortex dynamics motion will directly by inserted through the sphere Hamiltonian $H$.

$$ -\frac{1}{\w_j} \sgrad_{p_j} H = \frac{1}{4\pi} \sum_{  \substack{i=1 \\ i\neq j}}^n \w_i \frac{p_j \times p_i}{ 1 -  p_j\cdot p_i} $$

Meanwhile the new contribution gives us something to calculate with. We recall that on $\Sp$ we have $\vn(p)=p$.

\begin{align*}
-\frac{1}{\w_j} \sgrad_{p_j} \left( -\frac{1}{4\pi}\sum_{i=1}^n\w_i^2\log(h(p_i)) \right) &=  \frac{1}{4\pi}\sum_{i=1}^n\frac{\w_i^2}{\w_j}\sgrad_{p_j}\log(h(p_i))\\
&=\frac{1}{4\pi}\frac{\w_j^2}{\w_j}\vn(p_j)\times \grado_{p_j}\log(h(p_j))\\
&=\frac{1}{4\pi h(p_j)}\w_j p_j\times\frac{1}{h(p_j)}\grado_{p_j}h(p_j) \\
&=\frac{\w_j}{4\pi h(p_j)} p_j\times\grado_{p_j}h(p_j)
\end{align*}

At last the linearity of the $\sgrad_{p_j}$ operator allows us to just sum up these two terms to get the final result:

\begin{align*}
\vu(p_j) &= -\frac{1}{\w_j} \tsgrad_{p_j} \tH \\
&= -\frac{1}{h(p_j)^2\w_j} \sgrad_{p_j} \tH \\
&=  -\frac{1}{h(p_j)^2\w_j} \sgrad_{p_j} \left( H - \frac{1}{4\pi}\sum_{i=1}^n\w_i^2\log(h(p_i)) \right)\\
&= \frac{1}{h(p_j)^2} \left ( - \frac{1}{\w_j}\sgrad_{p_j} H   -\frac{1}{\w_j} \sgrad_{p_j} \left( -\frac{1}{4\pi}\sum_{i=1}^n\w_i^2\log(h(p_i)) \right ) \right )\\
&= \frac{1}{4\pi h(p_j)^2} \left( \sum_{  \substack{i=1 \\ i\neq j}}^n  \left( \w_i \frac{p_j \times p_i}{ 1 -  p_j\cdot p_i} \right )  +   \frac{1}{h(p_j)} \w_jp_j\times\grado_{p_j}h(p_j) \right )
\end{align*}

%{\color{red} ACHTUNG FEHLER! Ich habe einen $\frac{1}{h}$ beim rechten term. In Boattos Ergebnis kommt dies aber nicht vor. Irgendwie bin ich ja ganz nahe am Ergebnis doch ich weiss nicht warum es nicht klappt.}

%Notice also that since $\sum_{i=1}^n \w_i=0$ we have $\w_j=-\sum_{\substack{i=1 \\ i\neq j}}^n \w_i$. Inserting this into the formula above results in
%
%$$
%\vu(p_j) = \frac{1}{4\pi} \left( \sum_{  \substack{i=1 \\ i\neq j}}^n \w_i \frac{p_j \times p_i}{ 1 -  p_j\cdot p_i}   -   \frac{1}{h(p_j)} \w_ip_j\times\grado_{p_j}h(p_j) \right)
%$$

\end{proof}

The theorem suggests that the point vortices move as if on a sphere but with an additional \emph{self term} $ \frac{1}{h(p_j)} \w_jp_j\times\grado_{p_j}h(p_j)$ given by the gradient of the conformal factor while the overall speed is rescaled $\frac{1}{4\pi h^2(p_j)}$. 

$$\vu(p_j) = \underbrace{ \frac{1}{4\pi} }_{\text{factor}}    \underbrace{ \frac{1}{h(p_j)^2} }_{\text{conformal influence}}     \left( \underbrace{  \sum_{  \substack{i=1 \\ i\neq j}}^n \left( \w_i \frac{p_j \times p_i}{ 1 -  p_j\cdot p_i} \right) }_{\text{spherical dynamics}}  +   \underbrace{ \frac{1}{h(p_j)} \w_jp_j\times\grado_{p_j}h(p_j) }_{\text{self term}} \right)$$

The theorem describes the tangent vector on $\Sp$ and not on $M$. This is just a matter of one transport $d_{\tf}\vu(p_j)$ to bring the result back to $M$. We recommend staying on $\Sp$ also for the implementation and not transporting every tangent vector back to $M$. Solving the differential equations on $\Sp$ to then transport the resulting points back $p_j\mapsto \tf(p_j)$ is slightly more accurate when dealing with discretizations.

Using a neutral vortex particle we can easily derive the velocity fields induced at any given point:

\begin{theorem}[Closed Surface Point Vortex Velocity Field]
\label{prop:\generalvd Velocity Field}
Given $n \in \N$ point vortices $p_i$ on the closed surface $M$ with vortex strengths $\w_i\in\R$ respectively s.t. $\sum_{i=1}^n \w_i=0$. Let $f$ be a conformal map from $M$ to $\Sp$ and $h$ the conformal factor of this map. Then the velocity $\vu(p)$ at $p\in M$ away from any point vortex is expressed by
\begin{equation}
\vu(p) = \frac{1}{4\pi h(p)^2} \sum_{  \substack{i=1} }^n \w_i \frac{p \times p_i}{ 1 -  p\cdot p_i}
\end{equation}
\end{theorem}

\begin{proof}
We can argue by simply placing an additional point vortex $p_{n+1}$ with vorticity $\w_{n+1}=0$. Then it acts as a passive point vortex as it is still moved by all the others while not moving itself. The formulas then read

\begin{align*}
\vu(p_{n+1}) &= \frac{1}{4\pi h(p_{n+1})^2} \left( \sum_{  \substack{i=1 \\ i\neq {n+1}}}^{n+1} \left( \w_i \frac{p_{n+1} \times p_i}{ 1 -  p_{n+1}\cdot p_i} \right)  -   \frac{1}{h(p_{n+1})} \w_{n+1}p_{n+1}\times\grado_{p_{n+1}}h(p_{n+1}) \right)\\
&= \frac{1}{4\pi h(p_{n+1})^2} \sum_{  \substack{i=1} }^n \w_i \frac{p_{n+1} \times p_i}{ 1 -  p_{n+1}\cdot p_i}\\
\end{align*}

\end{proof}

So in other words it moves like on a sphere but with velocities adjusted by the conformal factor $h^2$. What a trivial looking result! Impressive given that the complexity of the problem seems to be only encoded into the conformal factor $h$ alone.

%=============================================================================
\section{\generalvd Results}
\label{sec:\generalvd Results}

Let us take a look at the implementation. Unlike with the plane $\R^2$ and the sphere $\Sp$ we so not have the privilege to compute the stream function $\psi$ this time, which is why we use coloring through the surface normals of the sphere instead.

Just like in the planar and spherical case we will observe the geodesic paths traveled by a vortex pair through Kimuara's conjecture. You will notice that the geodesics seem correct but since the vortex pair has a non-zero fixed distance in between each other in $\Sp$, not $M$. The inverse conformal map $f^{-1}$ seems to spread or contract the vortices after mapping them to $M$ depending on the conformal factor.

\begin{figure}[H]
\centering
\def\locallength{0.24\textwidth}
\includegraphics[width=\locallength]{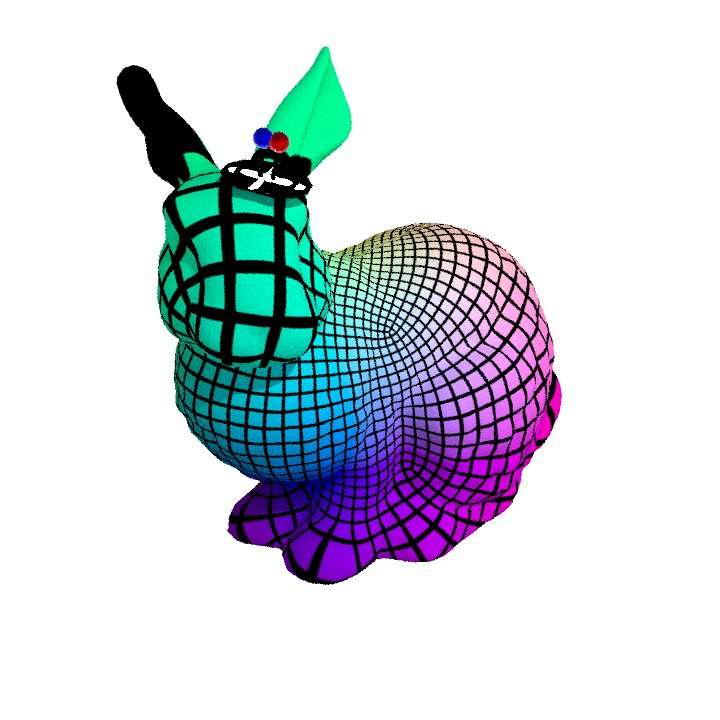}
\includegraphics[width=\locallength]{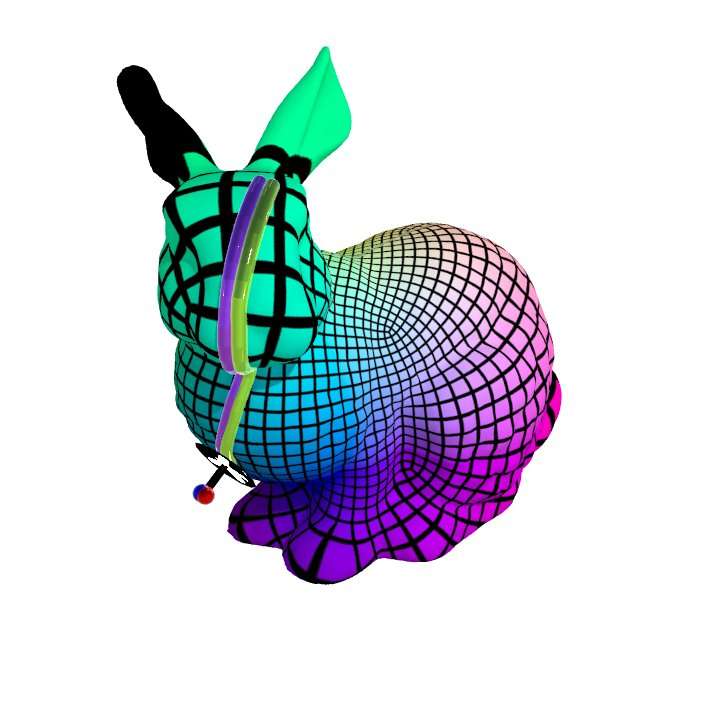}
\includegraphics[width=\locallength]{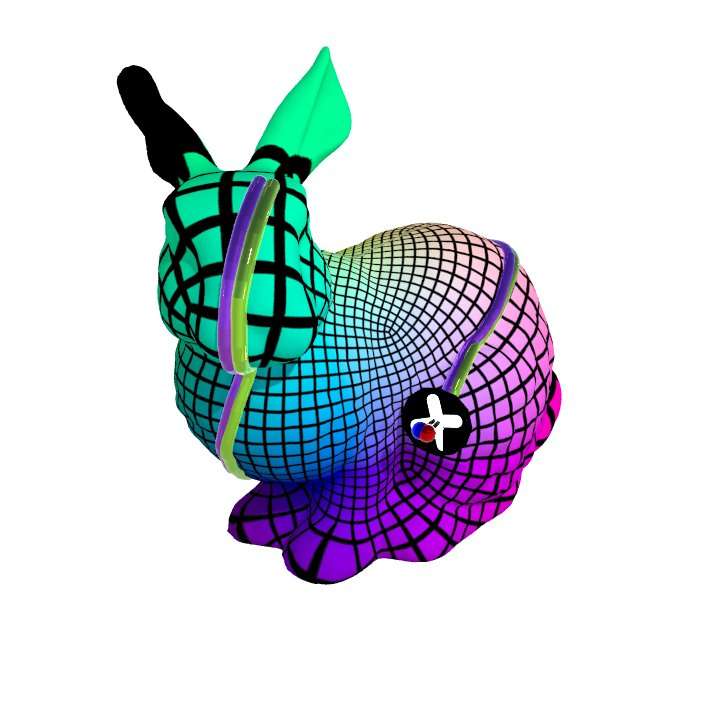}
\includegraphics[width=\locallength]{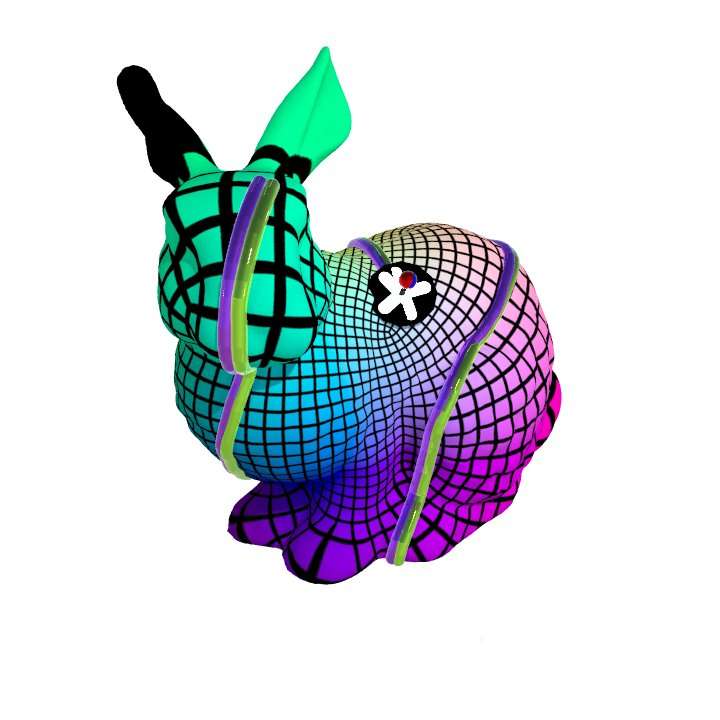}\\
\includegraphics[width=\locallength]{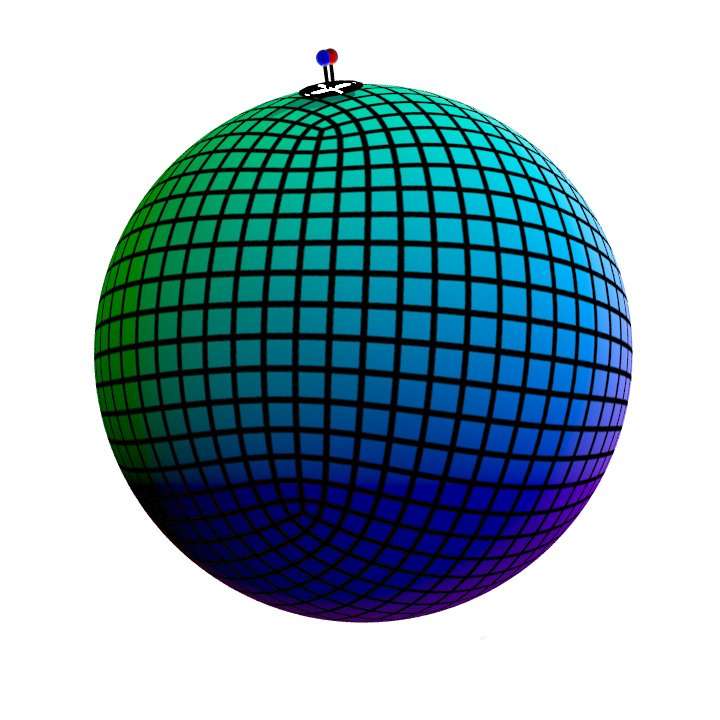}
\includegraphics[width=\locallength]{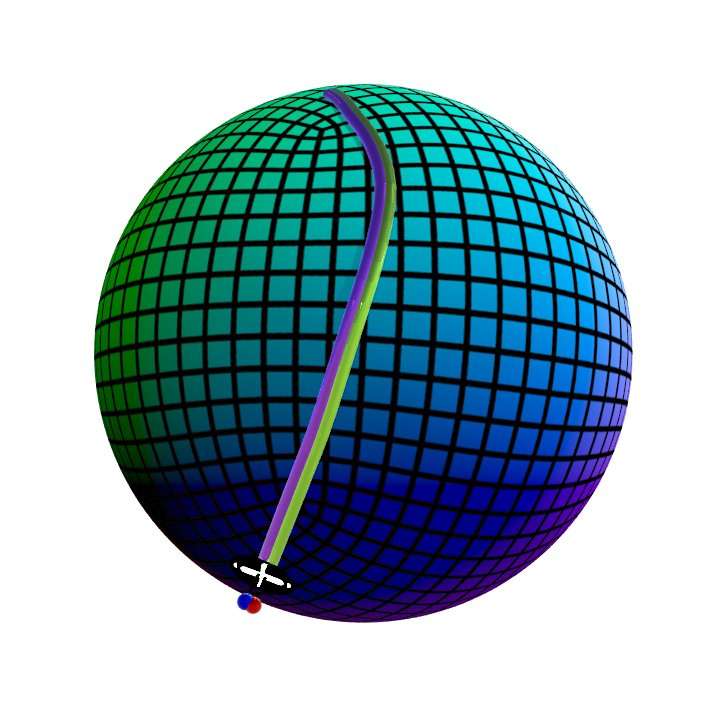}
\includegraphics[width=\locallength]{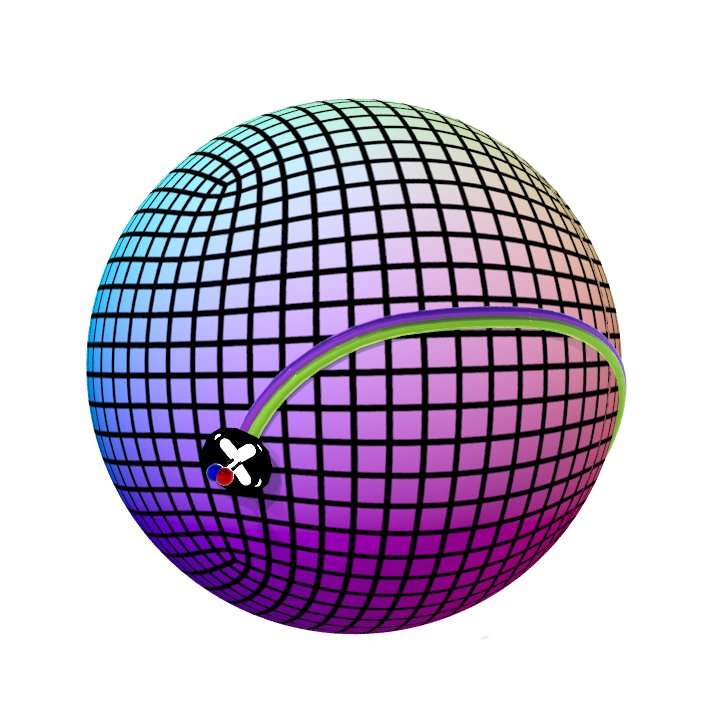}
\includegraphics[width=\locallength]{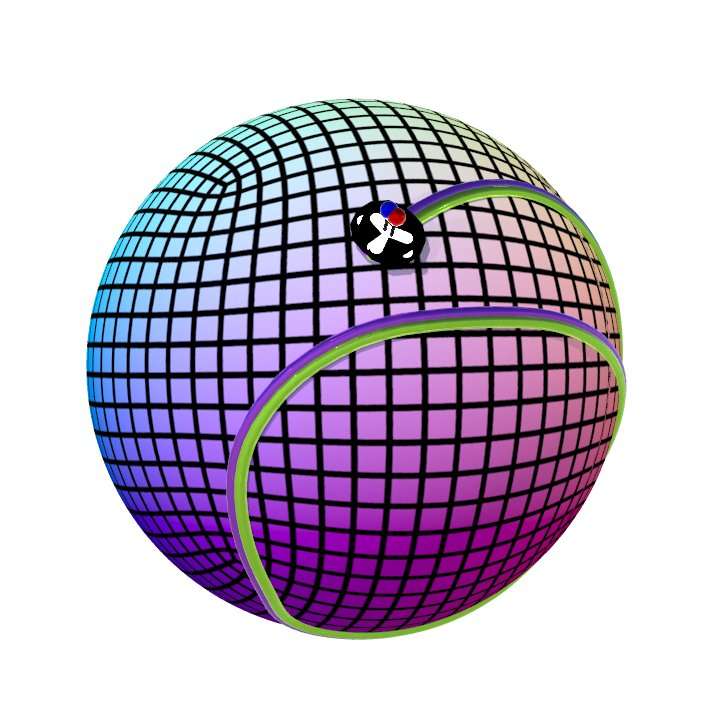}
\caption{A point vortex pair moving along a geodesic on $M$ (top row left to right) and on $\Sp$ after the conformal mapping $f$ (bottom row). $\Sp$ was rotated here to follow the path of the vortex pair.}
\label{fig:general kimura}
\end{figure}

We also observe the leap frogging in figure \ref{fig:general leapfrog}. Notice that it is falling apart after a few iterations. This might be due to the different paths each point vortex takes on the curve surface.

\begin{figure}[H]
\centering
\def\locallength{0.24\textwidth}
\includegraphics[width=\locallength]{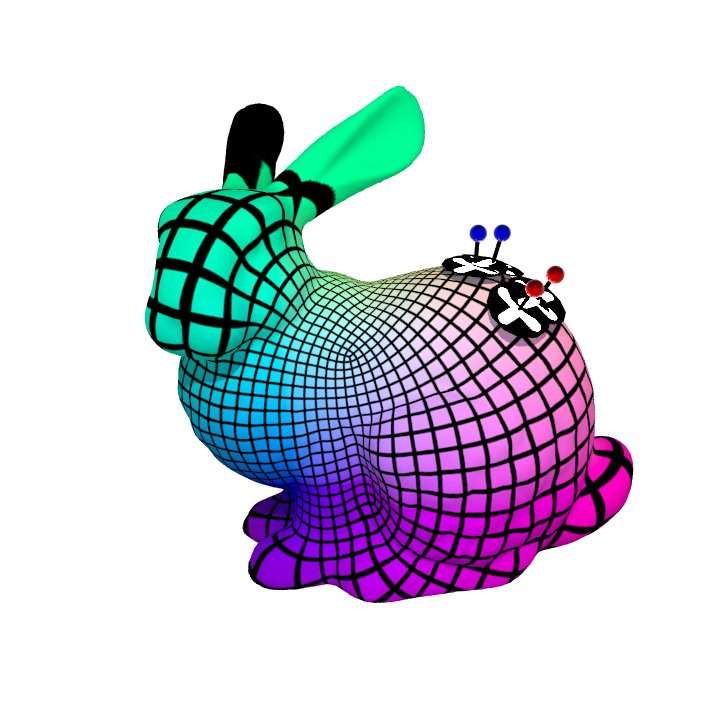}
\includegraphics[width=\locallength]{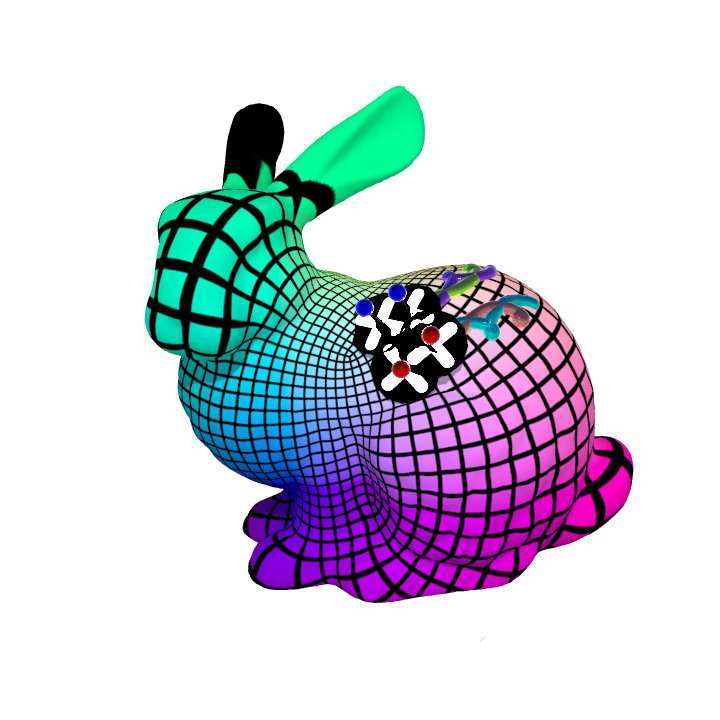}
\includegraphics[width=\locallength]{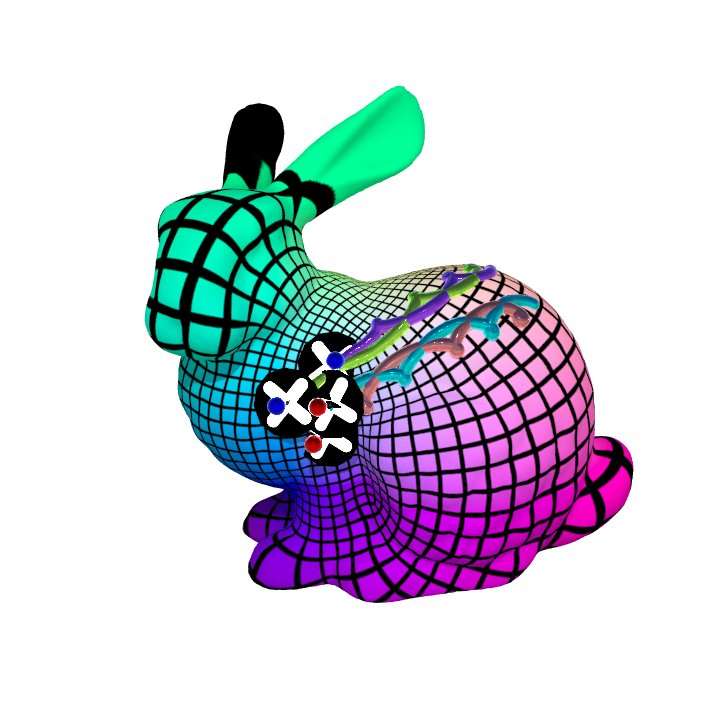}
\includegraphics[width=\locallength]{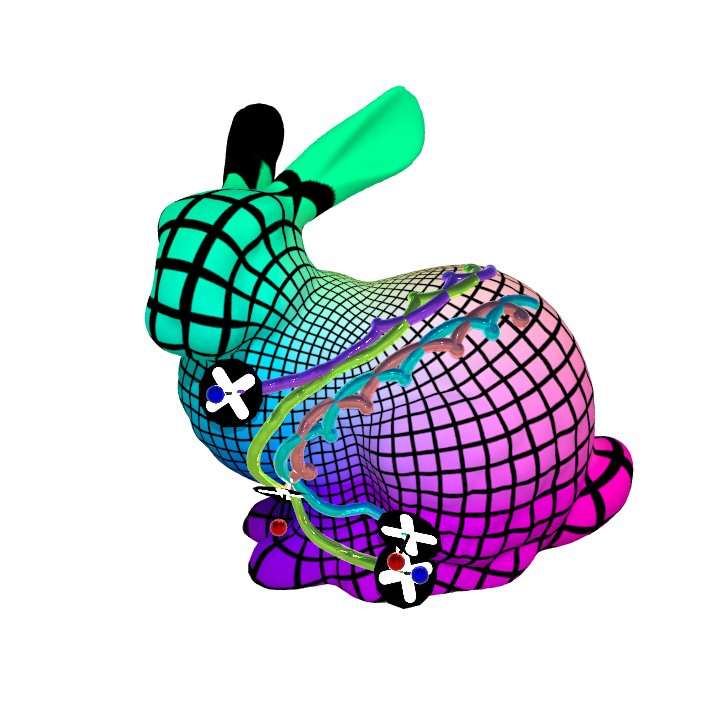}\\
\includegraphics[width=\locallength]{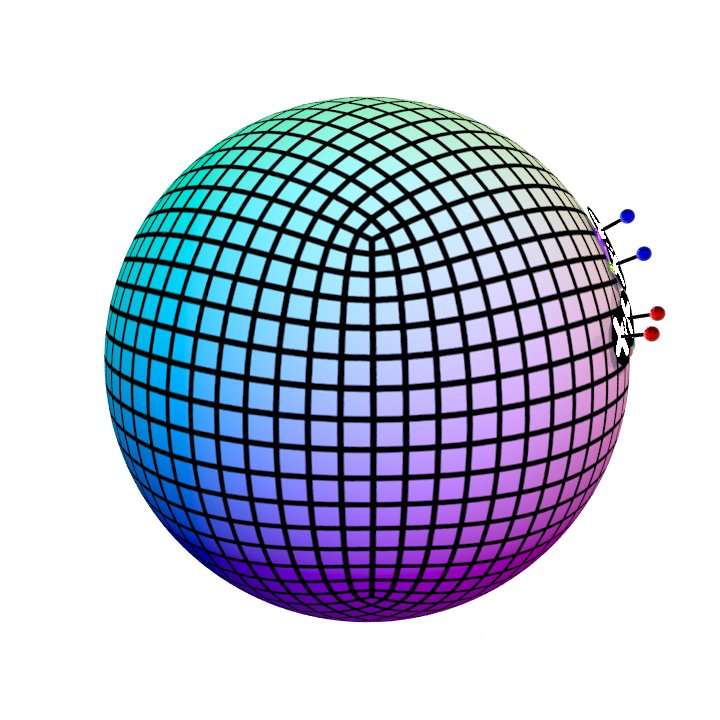}
\includegraphics[width=\locallength]{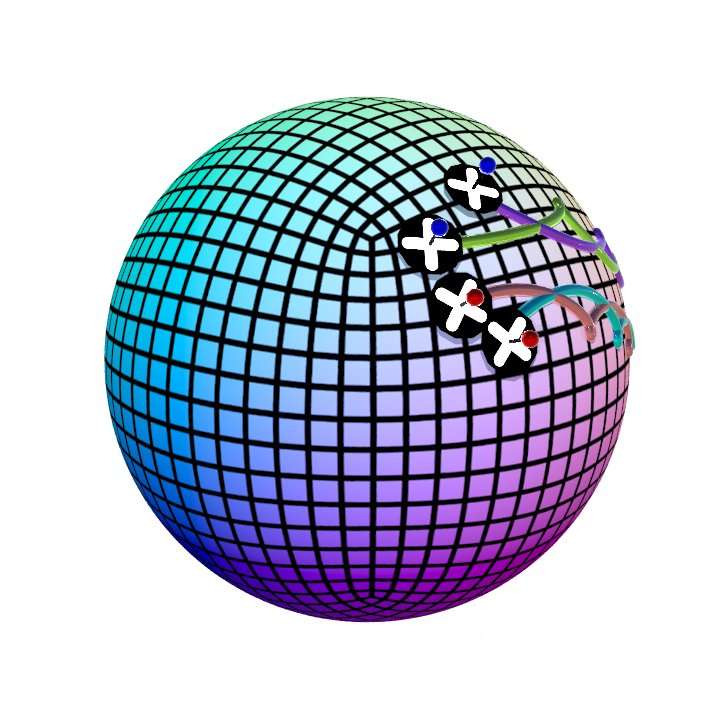}
\includegraphics[width=\locallength]{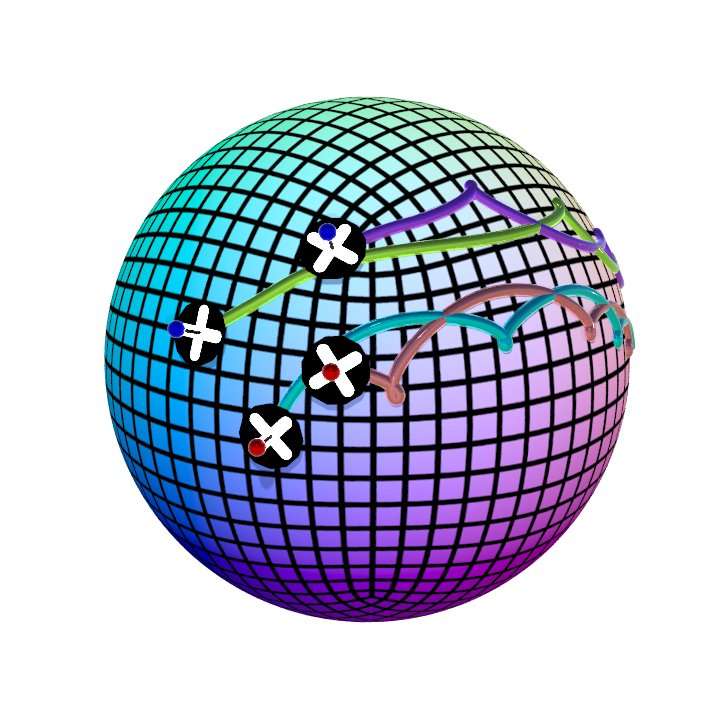}
\includegraphics[width=\locallength]{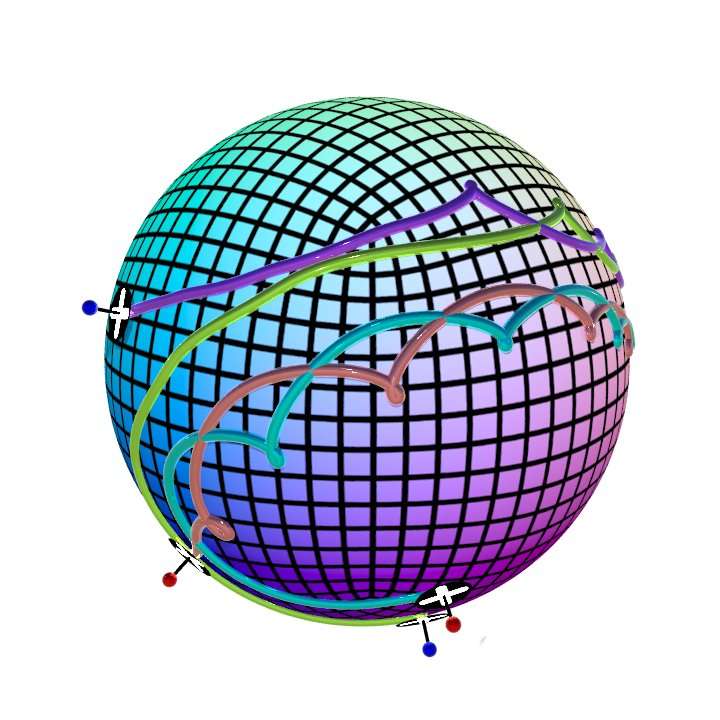}
\caption{4 point vortices in leapfrogging on a surface $M$ (top row) and on $\Sp$ (bottom row). Notice how the constellation falls apart later. (from top left to right)}
\label{fig:general leapfrog}
\end{figure}

In this setting things get especially messy when you look at higher number of point vortices. Figure \ref{fig:general many} shows this. Note again how on areas of extreme conformal factors the computations become less accurate and require finer time steps.

\begin{figure}[H]
\centering
\def\locallength{0.24\textwidth}
\includegraphics[width=\locallength]{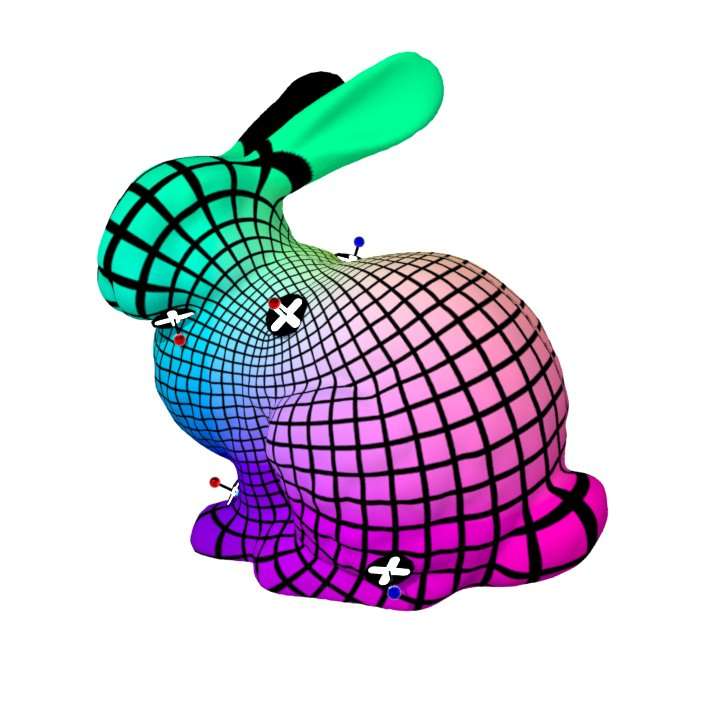}
\includegraphics[width=\locallength]{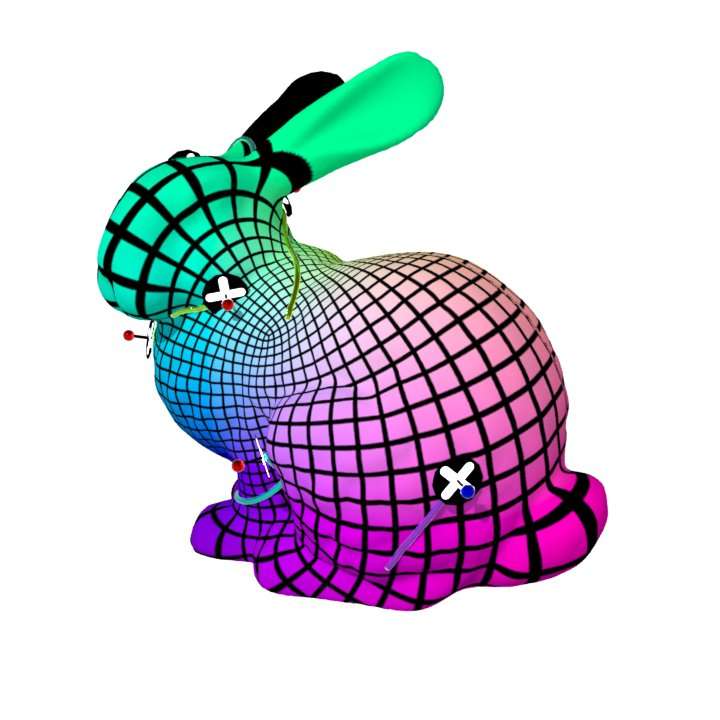}
\includegraphics[width=\locallength]{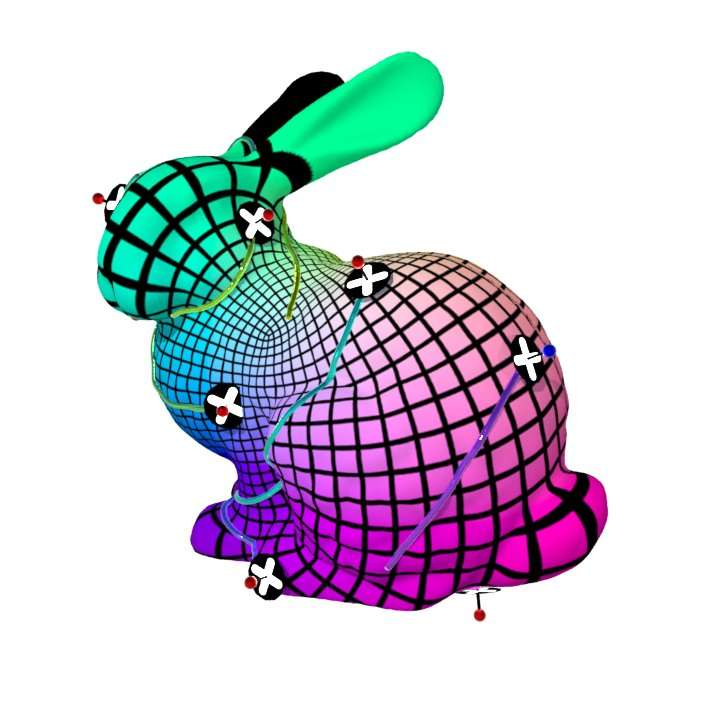}
\includegraphics[width=\locallength]{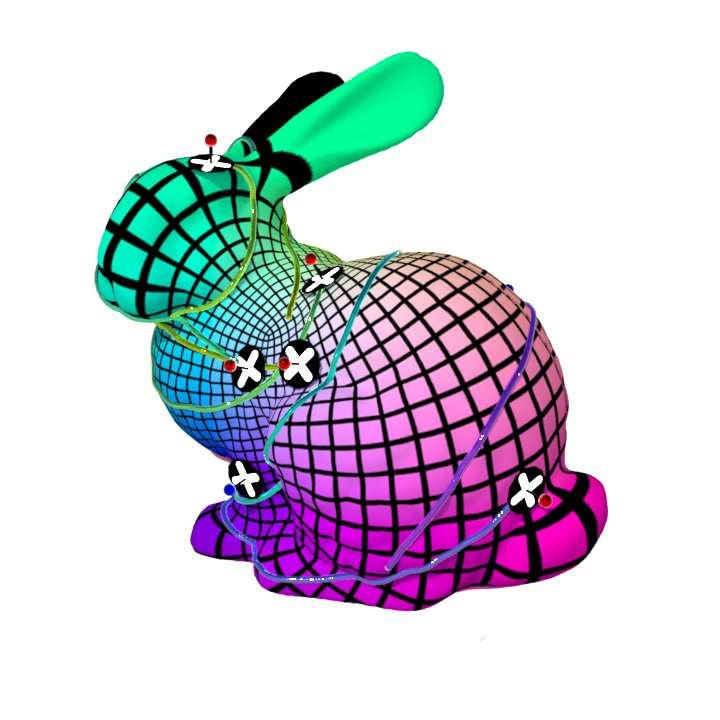}\\
\includegraphics[width=\locallength]{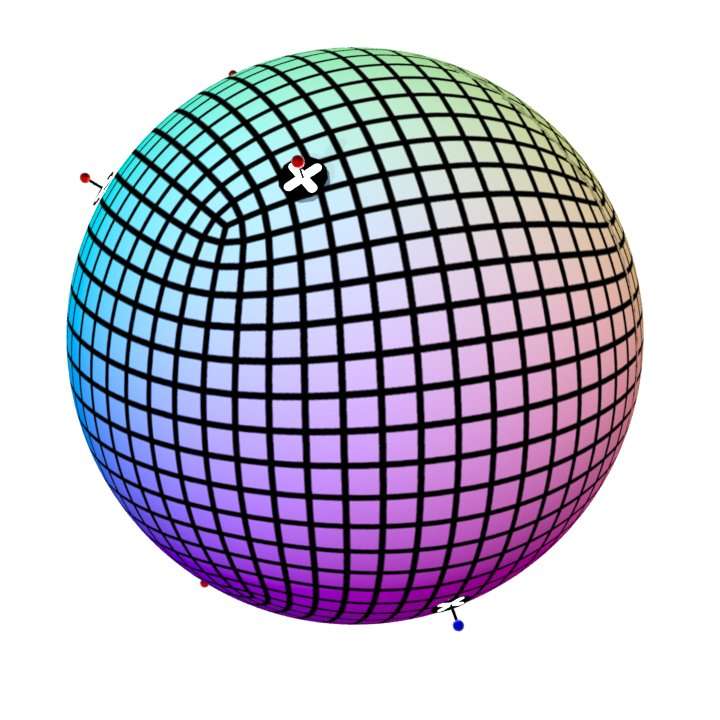}
\includegraphics[width=\locallength]{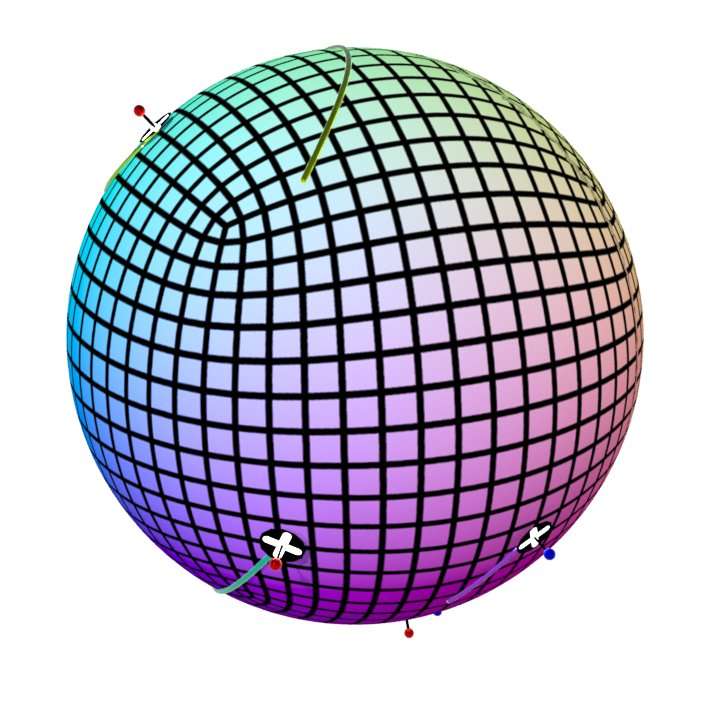}
\includegraphics[width=\locallength]{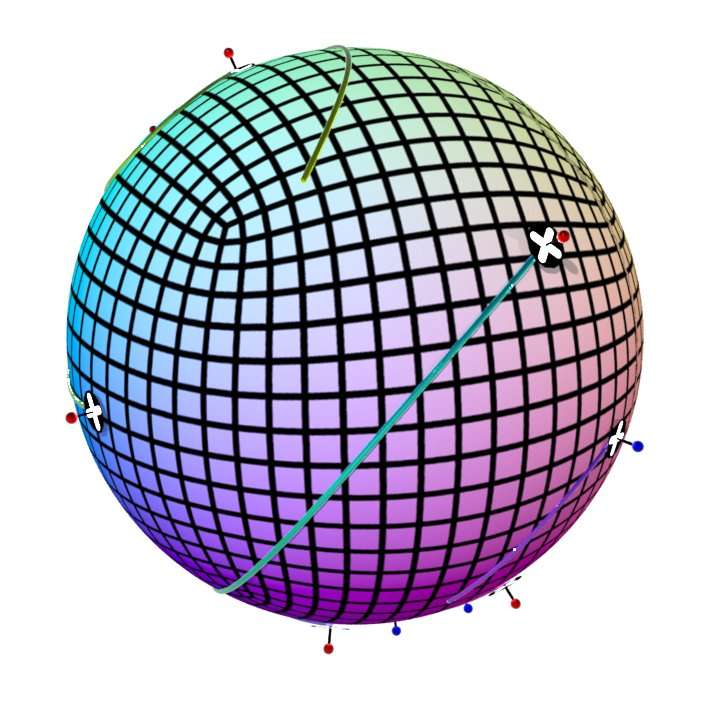}
\includegraphics[width=\locallength]{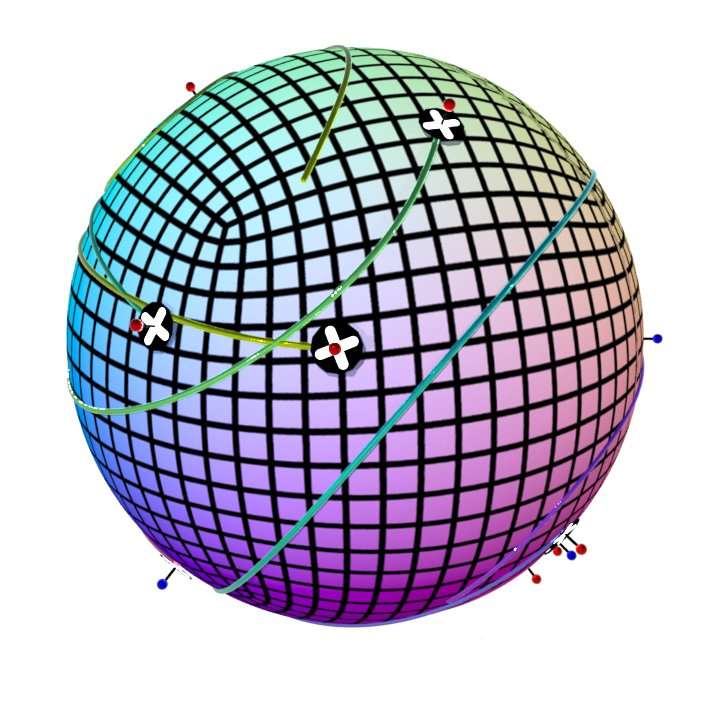}
\caption{Many vortices on $M$. (from left to right)}
\label{fig:general many}
\end{figure}

At last let us take a  look at the taylor vortices now where we simplify our point vortices visualization to single dots again. Remember that we wanted vanishing total vorticity $\sum_{i=1}^n\w_i=0$? Well the taylor vortices experiment require an extensive amount of equal signed point vortices $\sum_{i=1}^n\w_i>0$, which is why we add one single opposite point vortex on the opposite side to counter this unbalance $\w_{n+1}=-\sum_{i=1}^n\w_i$. The results does not seem visually skewed by this addition as seen in figures \ref{fig:general taylor} and \ref{fig:general taylor bear}.

\begin{figure}[H]
\centering
\def\locallength{0.24\textwidth}
\includegraphics[width=\locallength]{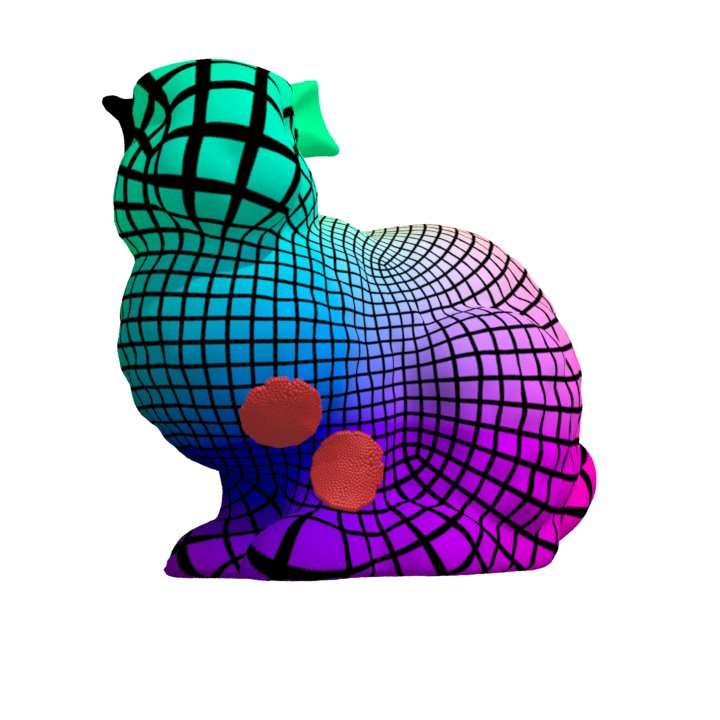}
\includegraphics[width=\locallength]{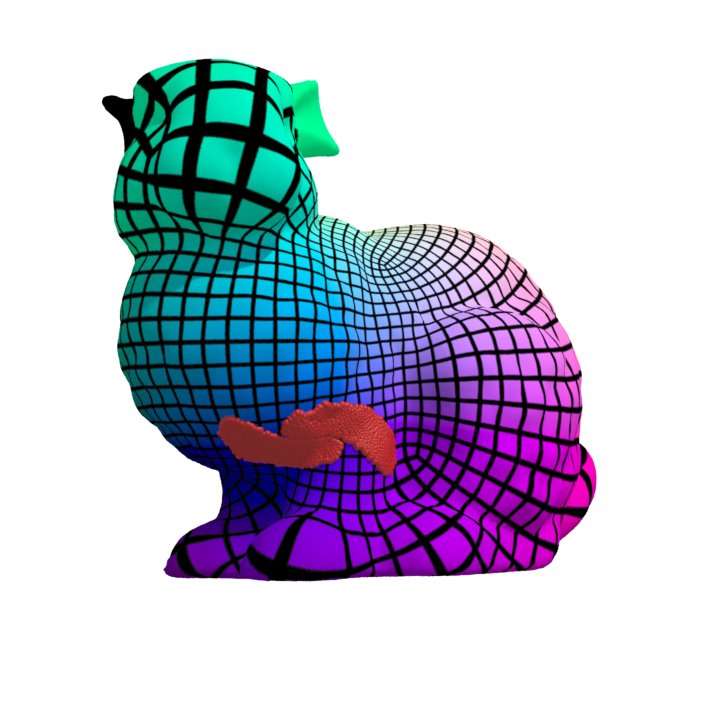}
\includegraphics[width=\locallength]{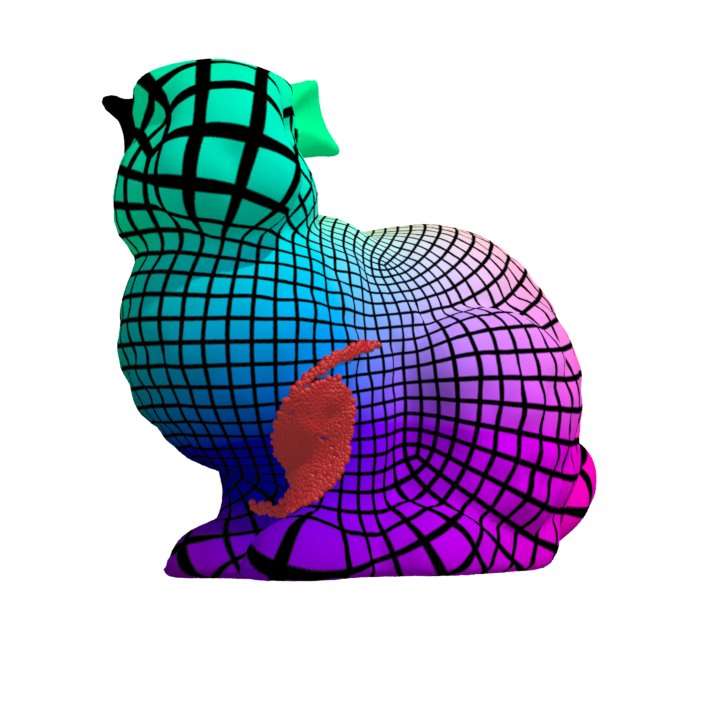}
\includegraphics[width=\locallength]{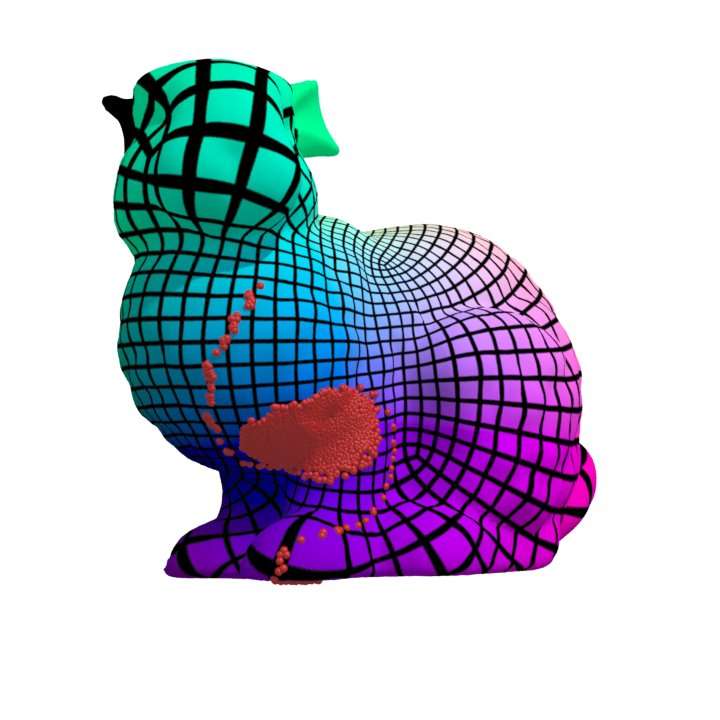}\\
\includegraphics[width=\locallength]{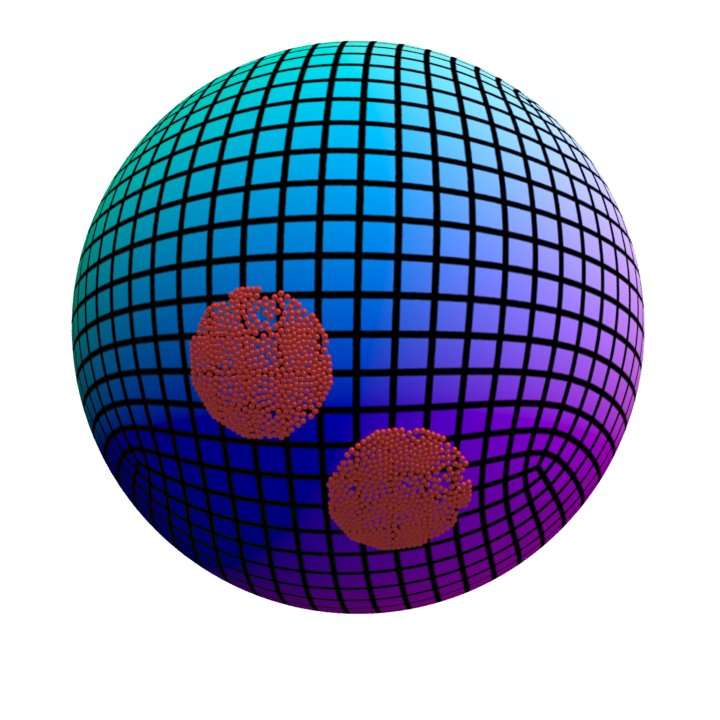}
\includegraphics[width=\locallength]{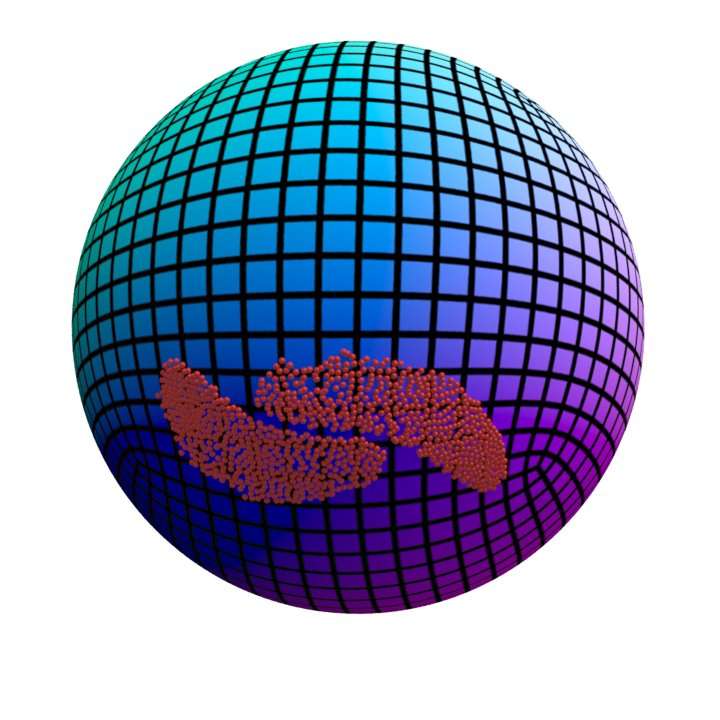}
\includegraphics[width=\locallength]{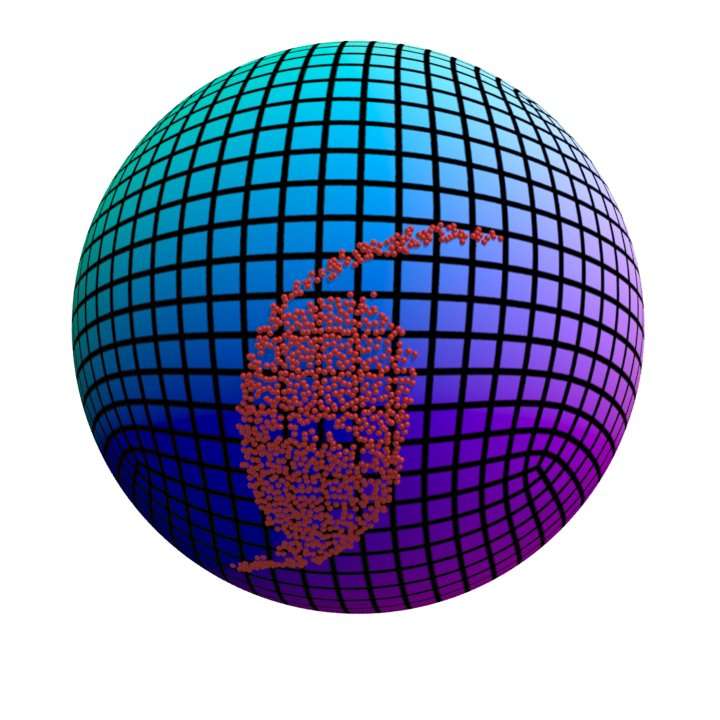}
\includegraphics[width=\locallength]{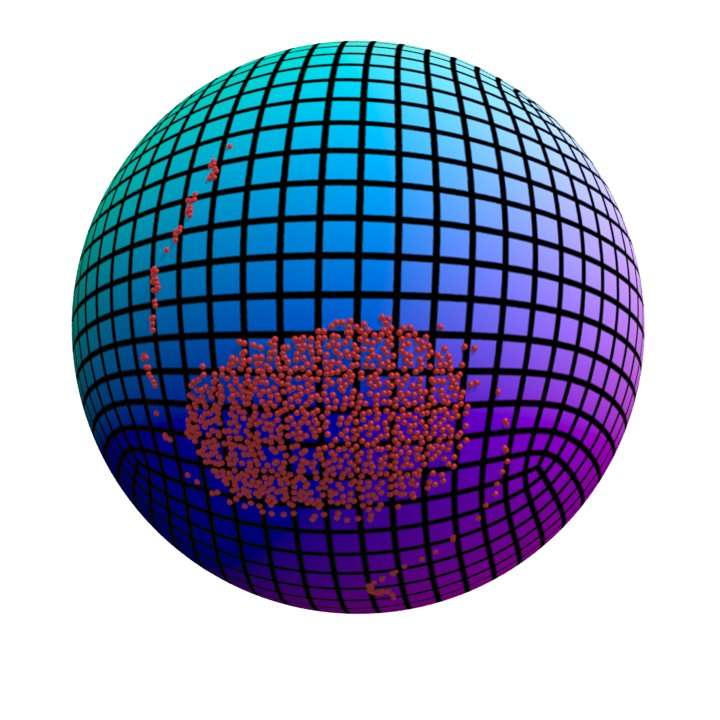}
\caption{Taylor vortices seen on the bunny $M$ and on the sphere again. (from left to right)}
\label{fig:general taylor}
\end{figure}

\begin{figure}[H]
\centering
\def\locallength{0.24\textwidth}
\includegraphics[width=\locallength]{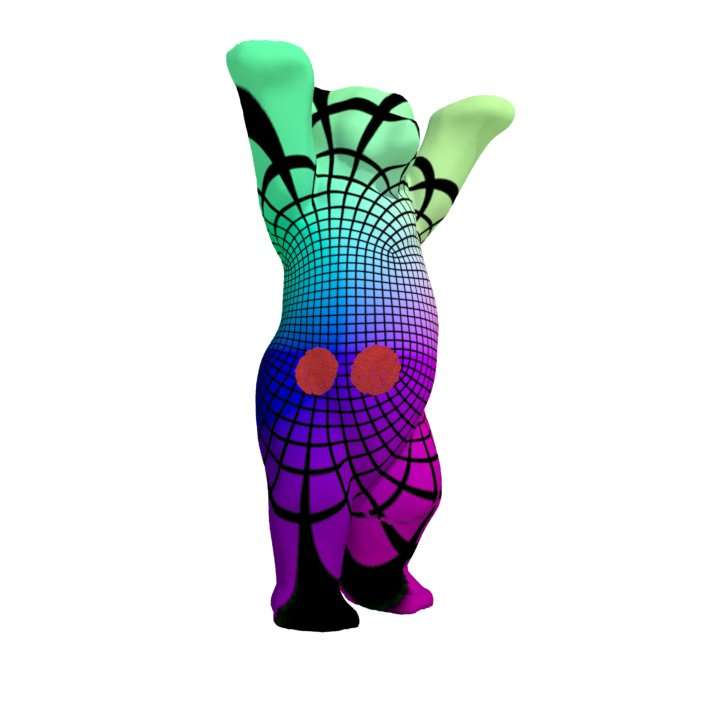}
\includegraphics[width=\locallength]{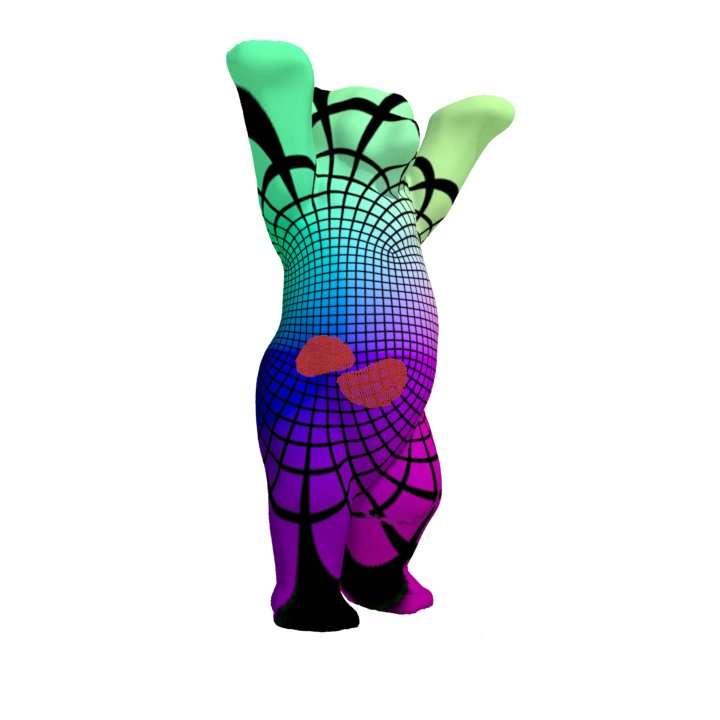}
\includegraphics[width=\locallength]{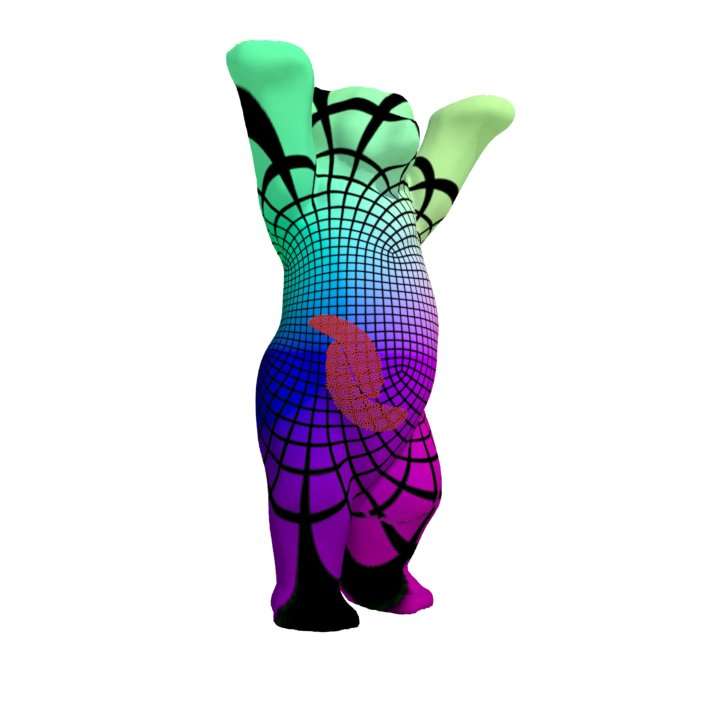}
\includegraphics[width=\locallength]{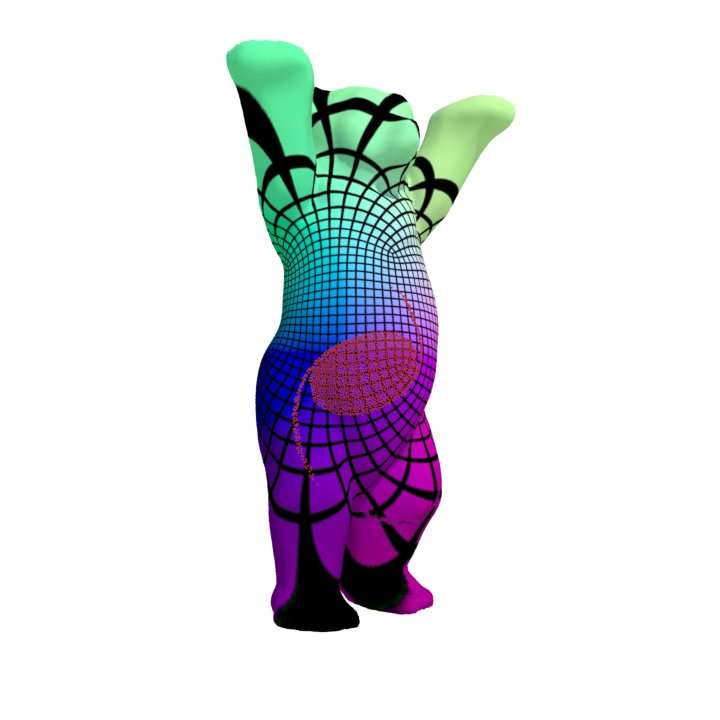}\\
\includegraphics[width=\locallength]{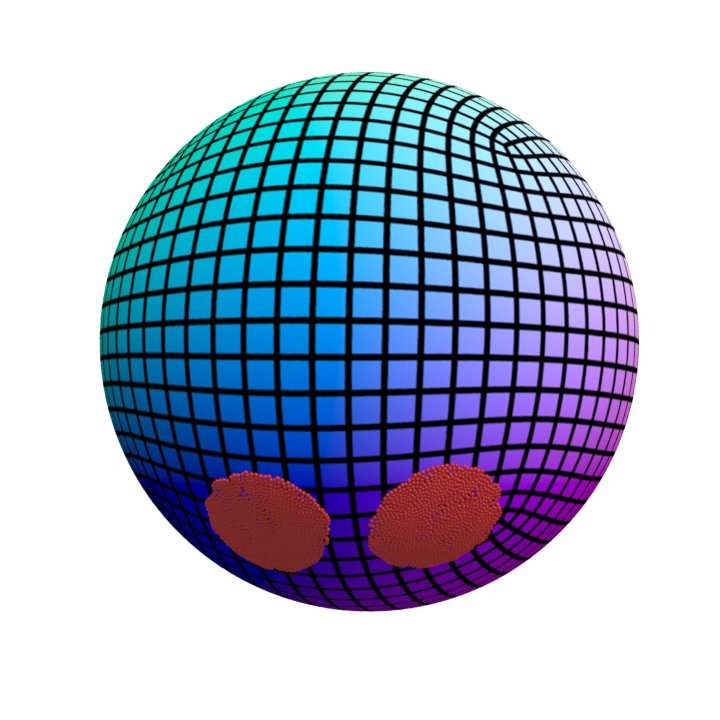}
\includegraphics[width=\locallength]{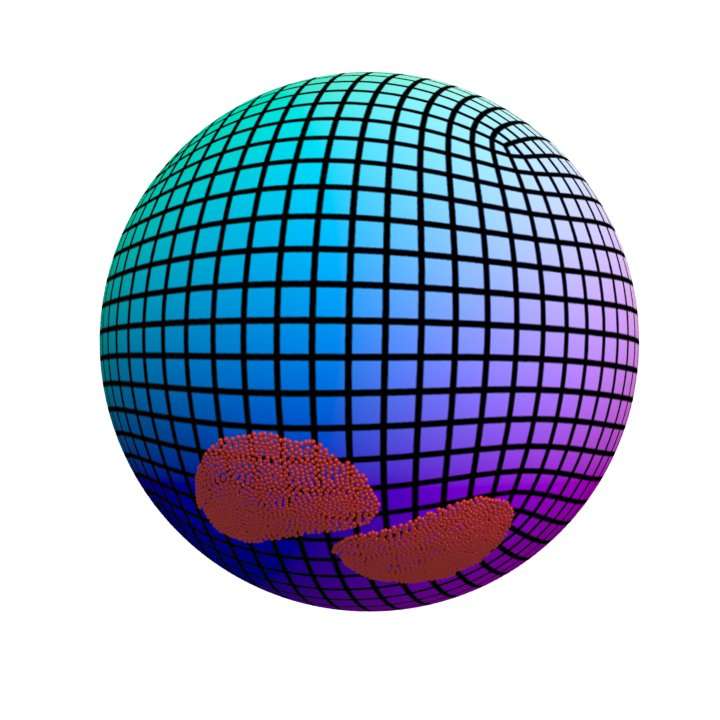}
\includegraphics[width=\locallength]{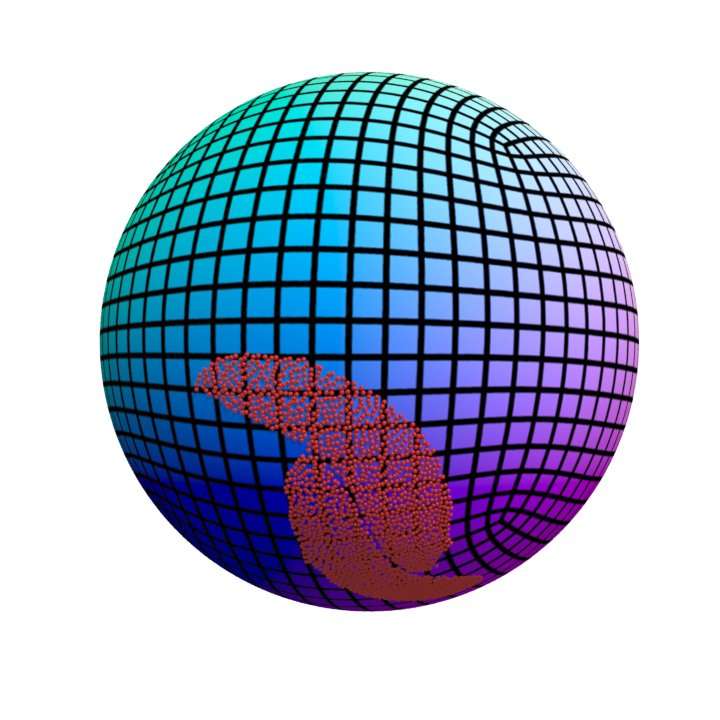}
\includegraphics[width=\locallength]{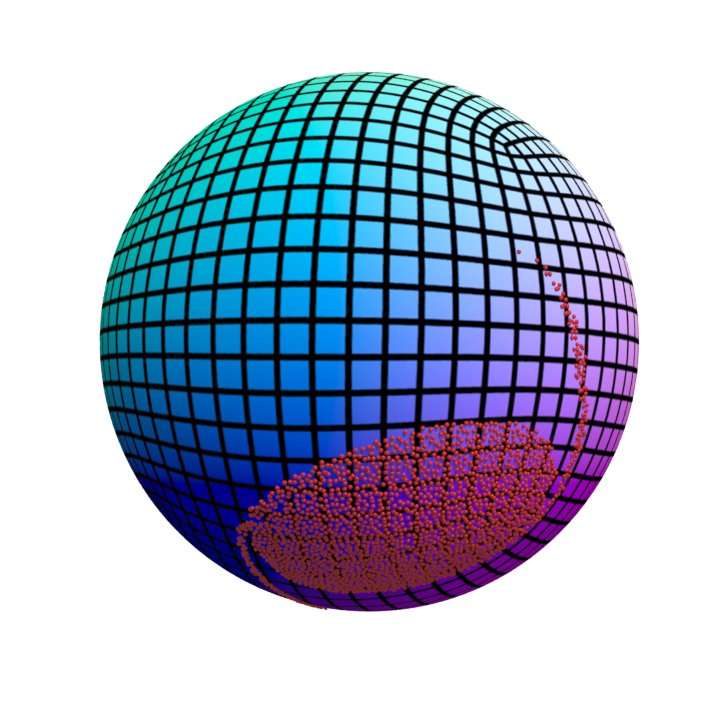}
\caption{Taylor vortices seen on the symbol of Berlin $M$ and on the sphere. (from left to right)}
\label{fig:general taylor bear}
\end{figure}

%=============================================================================
%\section{\generalvd Discussion}
%\label{sec:\generalvd Discussion}

%=============================================================================
%\import{chapters/chapter07/}{chap07.tex}
%\clearemptydoublepage

%+ + + + + + + + + + + + + + + + + + + + + + + + + + + + + + + + + + + + + + + + + + + + + + + + + + + + + + + + + + + + + + + + + + + + + + + + + + + + + + + + + + + + + + + + + + + + + + + + + + + + + + + + + + + + + + + + + + + + + + + + +    NEW CHAPTER     + + + + + + + + + + + + + + + + + + + + + + + + + + + + + + + + + + + + + + + + + + + + + + + + + + + + + + + + + + + + + + + + + + + +

\chapter{Implementation}
\label{chap:Implementation}

Until now we have focused on the smooth theory of point vortex motion only, meaning that we have always assumed our plane, sphere and closed surface to be smooth objects with infinite detail. This next chapter feeds on all the theory we have developed so far and discretizes it, outmaneuvering several dangerously deep and scary pitfalls. We will first mention the software used for our implementation and how you can use it and then proceed to implement the planar and spherical case. For the final case, the closed surface vortex motions, we need to introduce more details about surface to sphere mappings before the implementation.

All of the geometry cases will have similarities. Each one will use the main theorems from the chapters \ref{chap:\planarvd}, \ref{chap:\sphericalvd} and \ref{chap:\generalvd} in two ways:

\begin{enumerate}
\item To make use of the equation of motion for the point vortices.
\item To determine the velocity field away from the points.
\end{enumerate}

Both objectives will make use of the 4th order Runge-Kutta method when solving the differential equations at each time step.

%=============================================================================
\section{About using the Supplement Material}
\label{sec:About the Software}

\begin{figure}
\centering
\includegraphics[width=8cm]{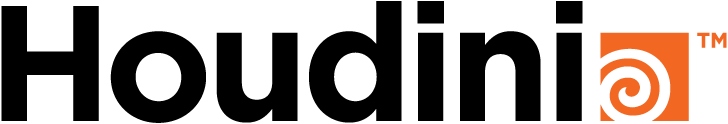}\\
\includegraphics[width=8cm]{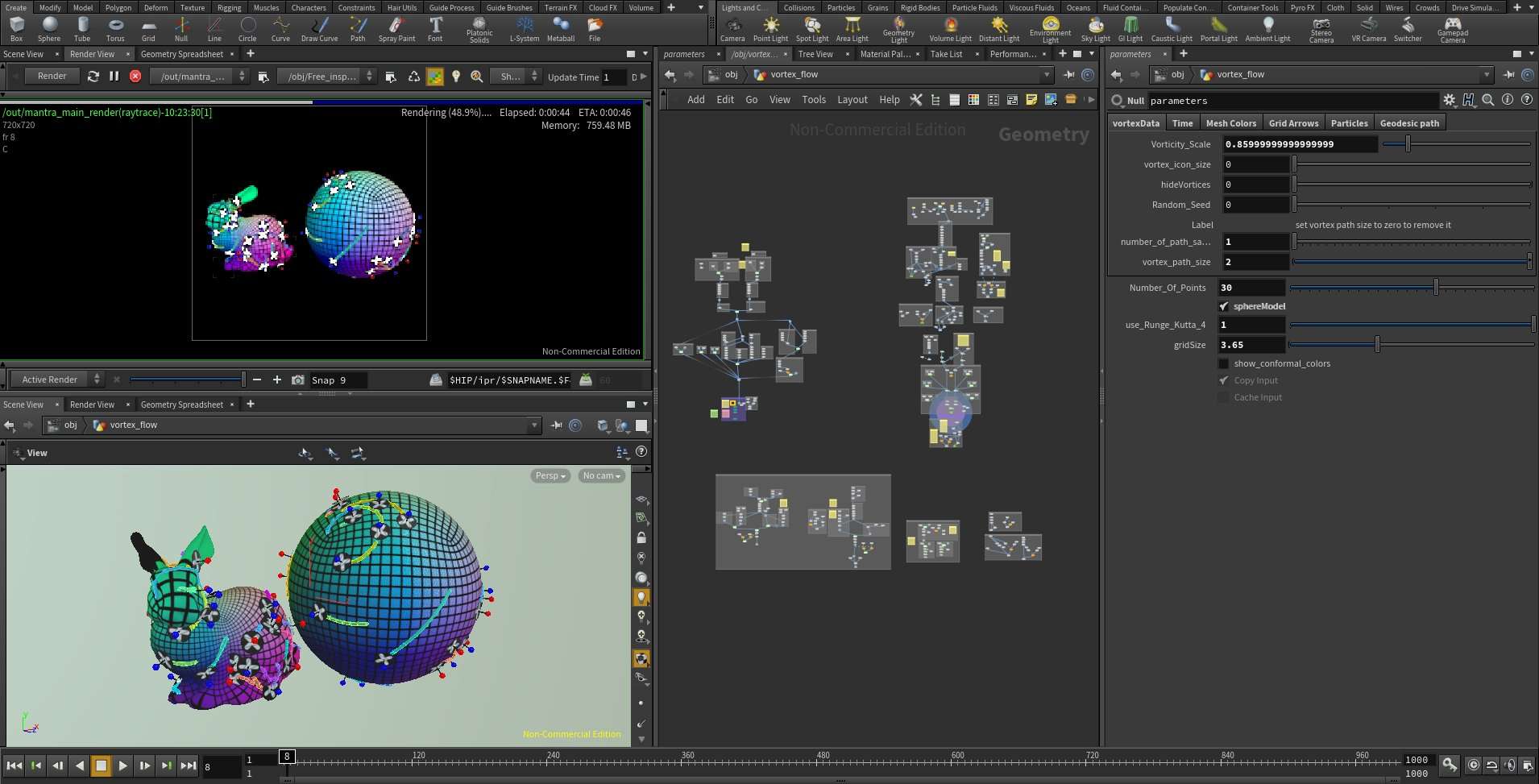}
\caption{The official Houdini logo with an image of the workspace.}
\label{fig:houdini workspace}
\end{figure}

Don't be afraid if you have never used Houdini before. The implementation as describe in this chapter will be independent of the software. Also you do not need to know anything about it in order to run the results on your computer using the documented supplement material.

The animation software we used has it's main focus on the creation of 3D designs and animations for games and movies. It comes with numerous useful tools for the testing of our algorithms. As of today, an email registration is required when downloading the software at

\begin{center}
\href{https://www.sidefx.com/products/houdini-apprentice/}{\color{blue}https://www.sidefx.com/products/houdini-apprentice/}.
\end{center}

The free version does not hinder the creation or execution of anything we present here. If you install Houdini, simply open our supplement files with it and follow the documentation written into the projects to know what to do. Pressing the play button at the bottom of the screen should execute the default simulation. An extensive tutorial on the use of Houdini for mathematicians can be found online under \href{wordpress.discretization.de/houdini/}{\color{blue}wordpress.discretization.de/houdini/}. The supplement material only requires the use of Houdini with Scipy which is trivial to use with Linux and Mac but requires an \href{http://wordpress.discretization.de/houdini/home/advanced-2/installing-and-using-scipy-in-houdini/}{\color{blue} additional set up guide for windows} found on the tutorial website.

For all differential equations used in here we used the 4th order Runge-Kutta method on all point vortices at once, meaning that we gathered all coordinates of all points in one large vector and advected them together. The advection and Runge-Kutta step run in parallel threads. Since Houdini is not the usual type of rendering software it's computation times are skewed and not displayed here. We keep in mind that all vortex dynamics formulas run with $\bigO(n^2)$ calculations were $n\in\N$ is the number of point vortices.

For computational efficiency there are of course problems. At every time step each point vortex want's to interact with each other point vortex, thus increasing the amount of computations quadratically with the number of points $n\in\N$, namely $\bigO(n^2)$. There are of course parallelization options available in the sum computations of the dynamics and often clusters of far away vortices can be considered as one. The most significant contributions to a single point vortex comes from the closest vortices.

%=============================================================================
\section{\planarvd Implementation}
\label{sec:\planarvd Implementation}

%\begin{figure}[H]
%\centering
%\import{images/vectorgraphics/}{plane.tex}
%\caption{The infinite plane $\R^2$ embedded in $\R^3$.}
%\label{fig:plane}
%\end{figure}

\begin{figure}
\centering
\def\locallength{5cm}
\includegraphics[height=\locallength]{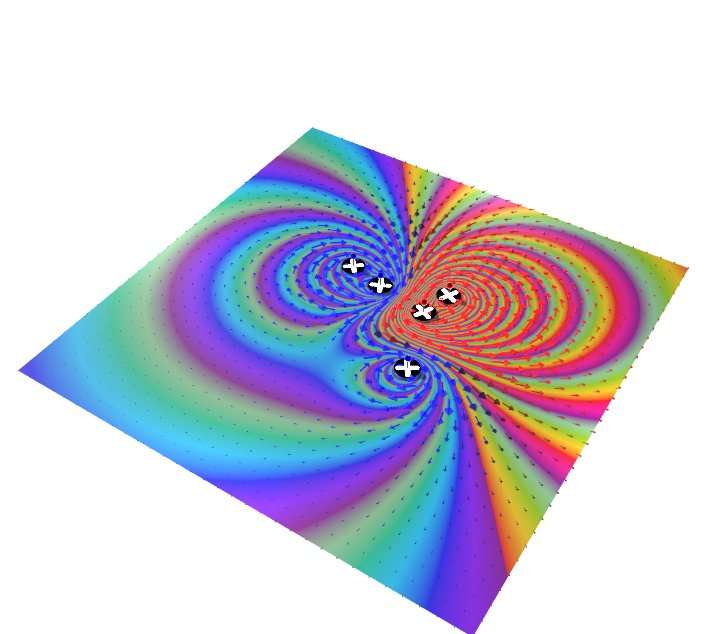}
\includegraphics[height=\locallength]{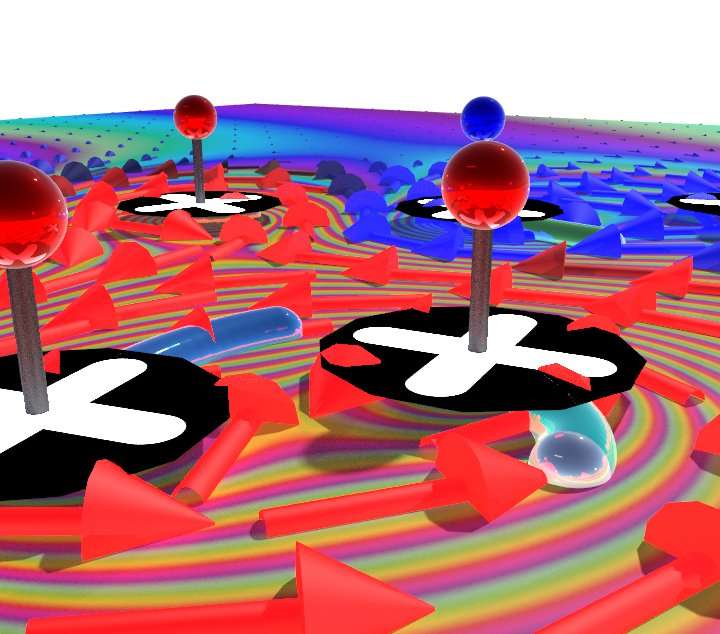}
\caption{Implementation set up seen from another perspective.}
\label{fig:planar set up}
\end{figure}

This case is fairly straight forward. The point vortex dynamics are independent from the resolution of the mesh of the plain so we only have to spawn some points and move them using the differential equation obtained in section \ref{sec:\planarvd Derivation} proposition \ref{thr:\planarvd}. See figure \ref{fig:planar set up} to better understand the set up.

Note that the color scheme used on the planar examples where created using the HSV color wheel to represent the level lines of the following stream function:

\begin{equation}
\psi = -\frac{1}{2\pi} \sum_{ \substack{i=1 \\ i\neq j} }^n \w_i \ln(|x-y|)
\end{equation}

We added little geometries and colored paths to each point vortex to identify them better.
% Note that the computation of $\psi$ to receive a high resolution set of vortex lines is slow and unsuitable for real time computations.

%=============================================================================
\section{\sphericalvd Implementation}
\label{sec:\sphericalvd Implementation}

%\begin{figure}[H]
%\centering
%\def\svgwidth{0.8\textwidth}
%\import{images/vectorgraphics/}{sphere}
%\caption{The $\Sp$ sphere embedded in $\R^3$.}
%%\label{fig:sphere}
%\end{figure}

\begin{figure}
\centering
\def\locallength{5cm}
\includegraphics[height=\locallength]{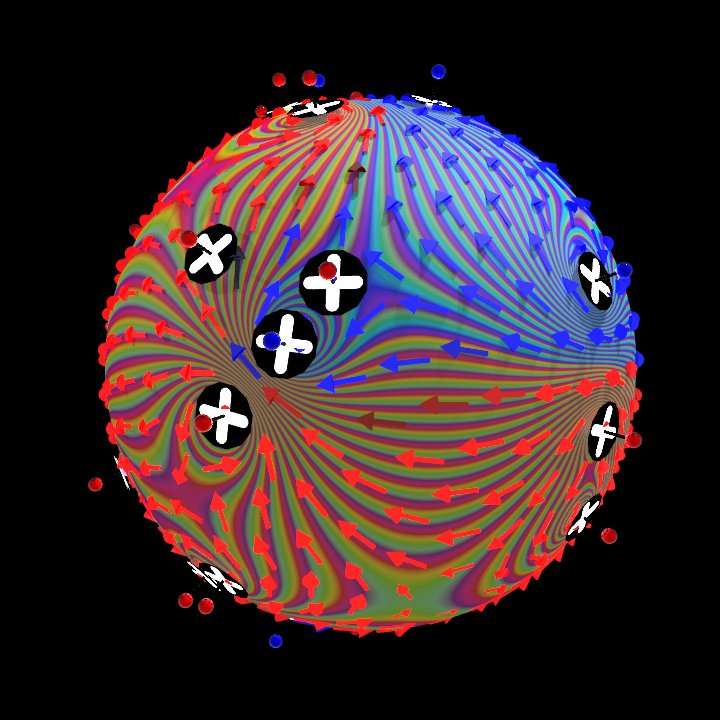}
\includegraphics[height=\locallength]{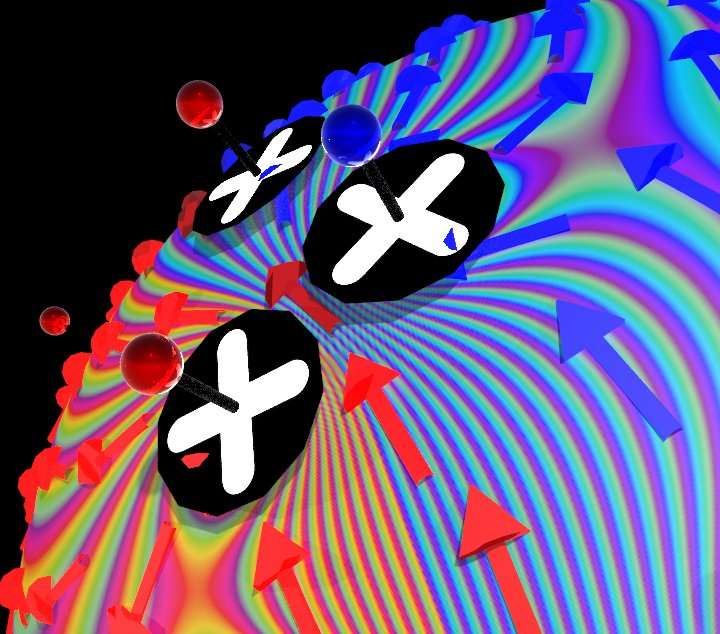}
\caption{Implementation for the $\Sp$ set up seen from another perspective.}
\label{fig:sphere set up}
\end{figure}

This case is implemented almost as straight forward as in the planar case as section \ref{sec:\sphericalvd Derivation} proposition \ref{thr:\sphericalvd} also provides us with clear differential equations for the point vortices. The resolution of the mesh to build our sphere has no effect either.

However, we must make sure that our advected point vortices remain on the sphere. Our approach to ensure this is to modify the Runge-Kutta method in the sense that the curvature of the unit sphere is taken into account. Instead of sending point vortices on straight lines we take the velocity vector and compute where the point vortex would arrive if we moved along the same distance but in the curved path.

Effectively, we rotate the point vortices instead of moving them on straight lines. Figure \ref{fig:Spherical advection} illustrates this on one arc of the sphere. Note that the curved vector and the straight vector have the same length. We do this to avoid a loss of motion through the normalization (projection back to the sphere) that would happen in classical numerical settings.

\begin{figure}
\centering
\def\svgwidth{0.8\textwidth}
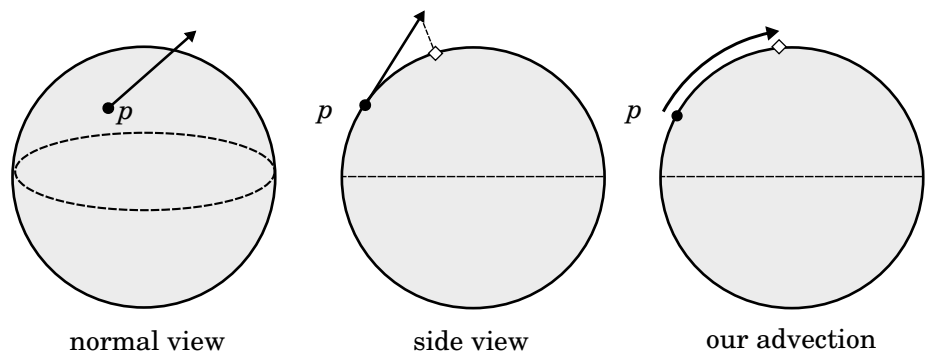
\caption{On $\Sp$ we see a tangent vector (left). An advection step with normalization (center). An advection step on spherical motion (right).}
\label{fig:Spherical advection}
\end{figure}

To understand the coloring in figure \ref{fig:sphere set up} please refer to the planar case implementation but with the following stream function:

\begin{equation}
\psi = -\frac{1}{2\pi} \sum_{ \substack{i=1 \\ i\neq j} }^n \w_i \ln \left(\ \sin\left(\frac{1}{2}\acos(x\cdot y) \right)\ \right )
\end{equation}

%=============================================================================
\section{\generalvd Implementation}
\label{sec:\generalvd Implementation}

Here comes all the fun now. We will now provide crucial steps needed in order to successfully implement closed surface (genus 0) point vortex dynamics.

We assume that we cannot find a smooth parametrization of the closed surface $M$ who's Green's function we would like to compute. Unlike with the planar and spherical case, we are now dependent on the resolution of the mesh of a surface that we are given and can only work with the limited amount of information. The key to implementing closed surfaces point vortex dynamics is to implement discrete conformal- maps, factors and transports between the triangular surface of the mesh and the sphere in order to use the result from section \ref{sec:\generalvd Derivation} theorem \ref{prop:\generalvd}. Also note that we only work with triangulated meshes (see figure \ref{fig:triangulated mesh}).

\begin{figure}
\centering
\includegraphics[width=4cm]{images/houdini/mesh_bunny_low}\includegraphics[width=4cm]{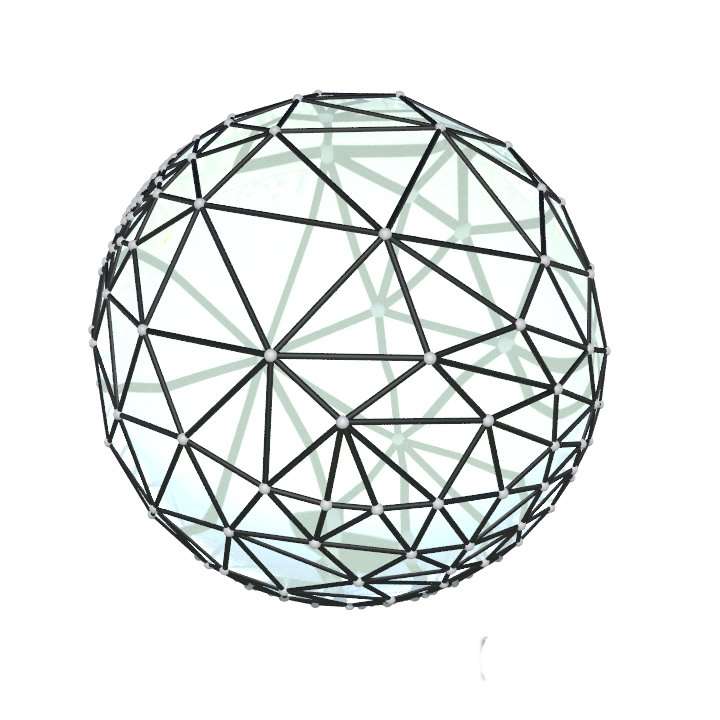}
\caption{A low resolution mesh just to display the general structure. We later use higher resolutions.}
\label{fig:triangulated mesh}
\end{figure}

\subsection{Discrete Conformal Mapping}

The most challenging task is to grab a \emph{discrete conformal map}. This word's definition is not self explanatory and there are in general multiple ways to discretize the concept of conformality using different properties of conformal maps \cite{Crane:2017:GID}. A good way to understand \emph{discrete conformal maps} is that we want to map each point of the mesh to a point on the sphere in such a way that the distortion of the triangles connecting the points is minimal. The individual triangles will change in scale which we will capture with the conformal factor (figure \ref{fig:discrete conformal map}).

\begin{figure}
\centering
\def\locallength{6cm}
\includegraphics[height=\locallength]{images/houdini/sphere_conformal}
\includegraphics[height=\locallength]{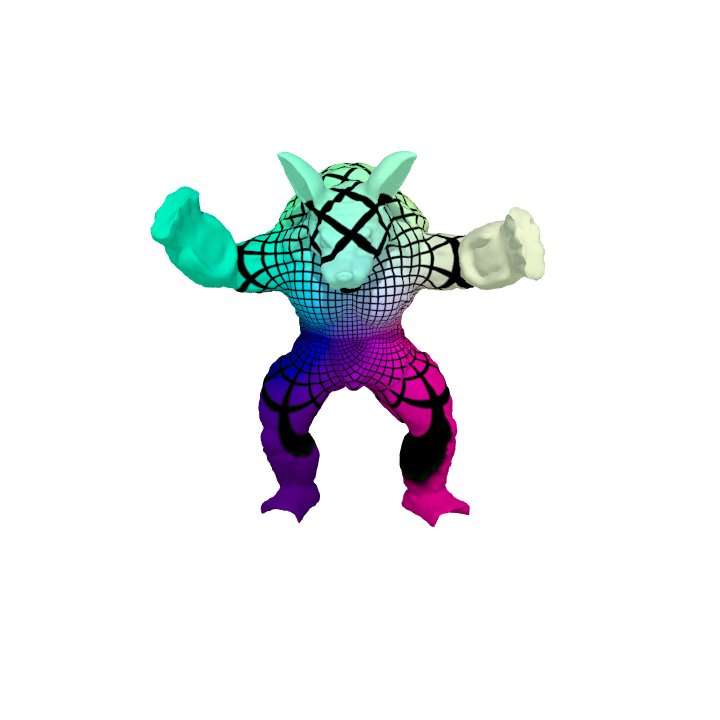}\\
\includegraphics[height=\locallength]{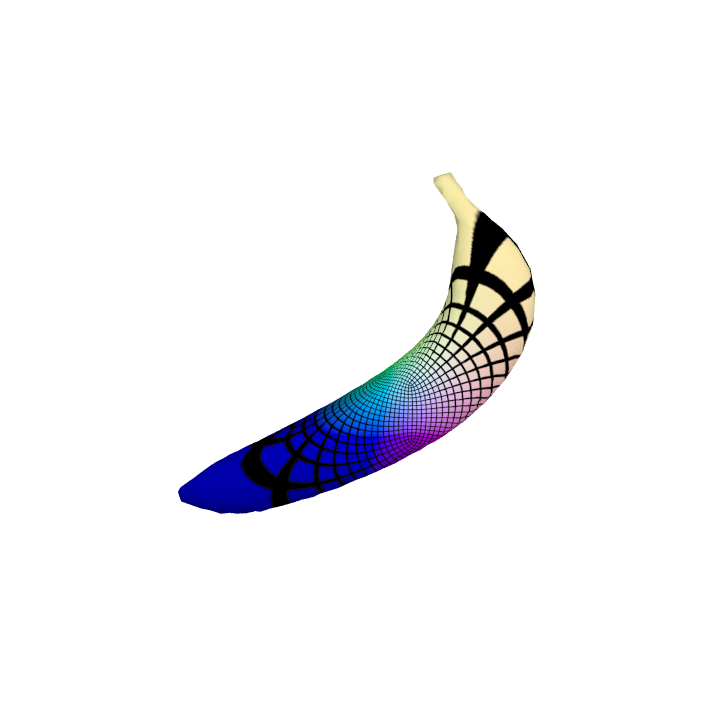}
\includegraphics[height=\locallength]{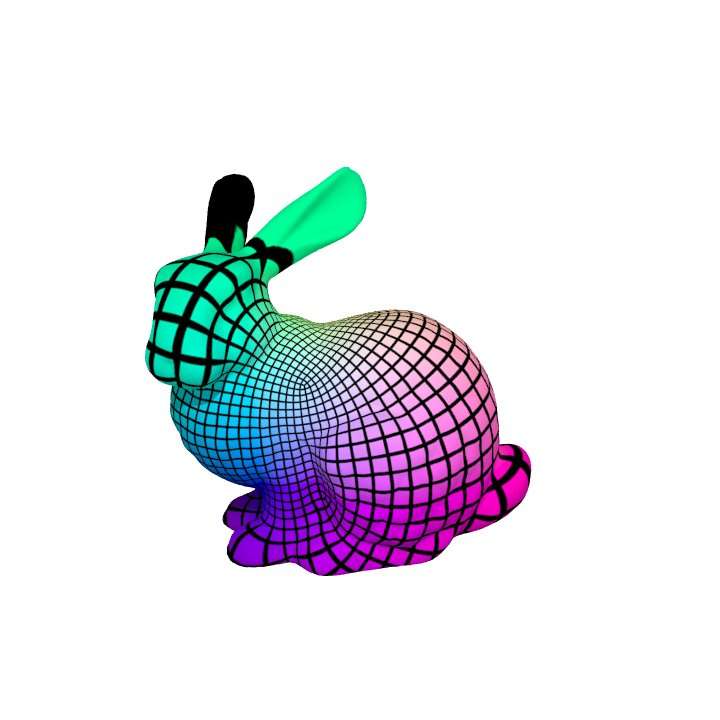}
\caption{An example of a discrete conformal maps. The sphere is mapped to different closed surfaces of genus zero, thus each surfaces can be conformally mapped to any other.}
\label{fig:discrete conformal map}
\end{figure}

We compute the discrete conformal mapping using discrete differential geometry and discrete exterior calculus which has provided handy results in \cite{pinkall1993computing}, \cite{Springborn:2008:CET:1360612.1360676}, \cite{Crane:2013:GH}, \cite{Knoppel:2013:GOD}, \cite{azencot2014functional}, \cite{WeiBmann:2014:SRS:2601097.2601171}, \cite{Chern:2016:SS:2897824.2925868} and many more. One of such DEC results from the \emph{Can Mean-Curvature Flow Be Modified to Be Non-singular?} paper by Michael Kazhdan et al. published in 2012  \cite{Kazhdan:2012:MFM:2346796.2346809} provides us a method to create the conformal map we need through the \emph{modified mean curvature flow}. In essence, it is an iterative method that repeatedly solves

\begin{equation}
\label{eq:conformal iteration}
(\Id - \delta\Delta_0)f_{k+1} = f_k
\end{equation}

where $\Delta_0$ represents the DEC Laplacian of the initial surface mesh $M$, $\delta\in\R$ the time step, $f_0\in\R^{3N}$ the coordinates of all $N\in\N$ points in the mesh and $k\in\N$ the iteration counter. The exact construction of the Laplace operator in this iteration scheme can be read inside the paper. Applying this backward euler method  often enough and normalizing the scale every time results in a new surface where every point lies on $\Sp$ with minimal distortion in the triangles. 

The resulting mesh of a sphere provides us with the conformal map that we have been looking for. As the iterative process of equation \eqref{eq:conformal iteration} reshaped the original mesh $M$ to a sphere mesh $\Sp$ we can now directly identify each initial point on $M$ with a point on $\Sp$. See figure \ref{fig:mcf} to follow this process.

\begin{figure}
\centering
\def\locallength{0.3\textwidth}
\includegraphics[width=\locallength]{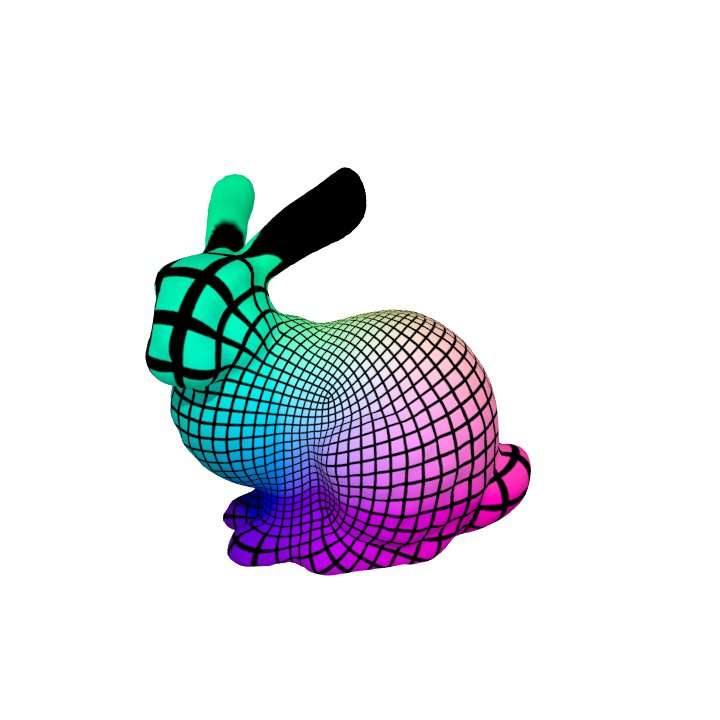}
\includegraphics[width=\locallength]{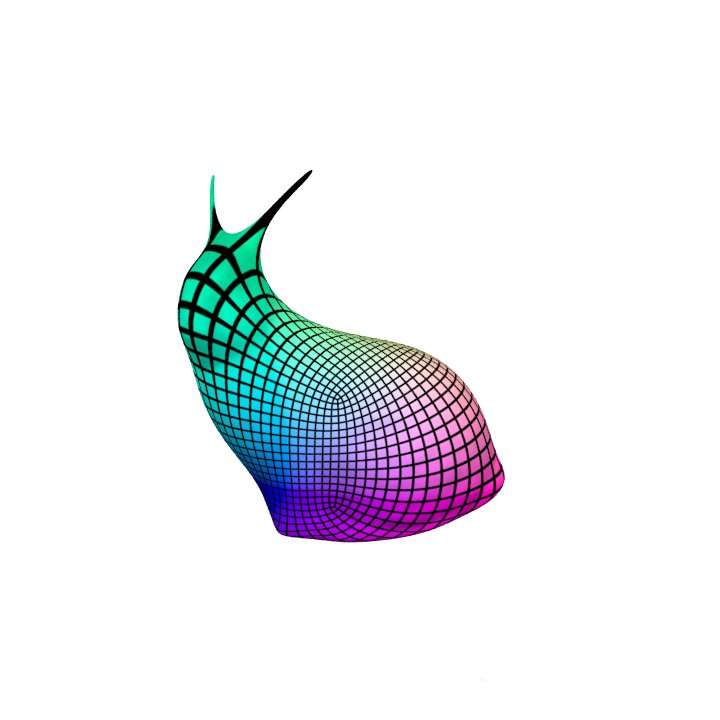}
\includegraphics[width=\locallength]{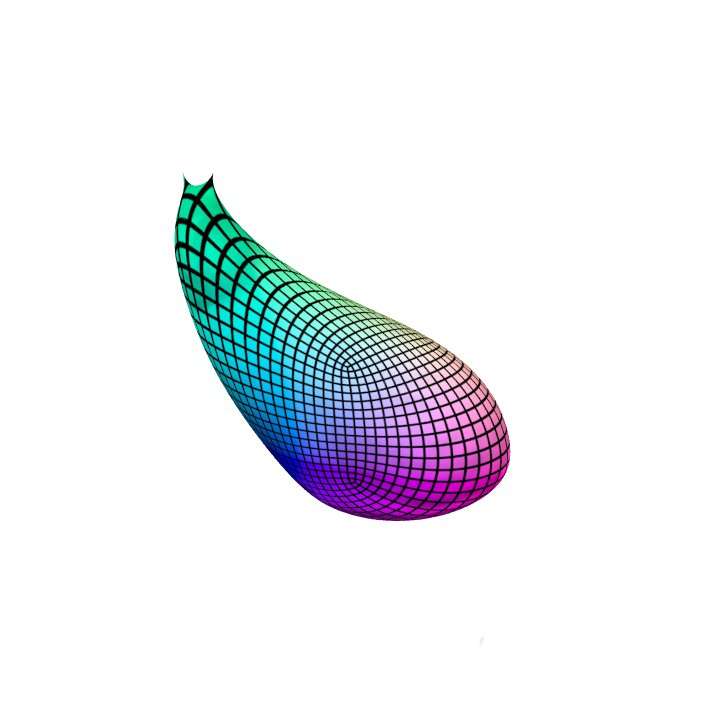}\\
\includegraphics[width=\locallength]{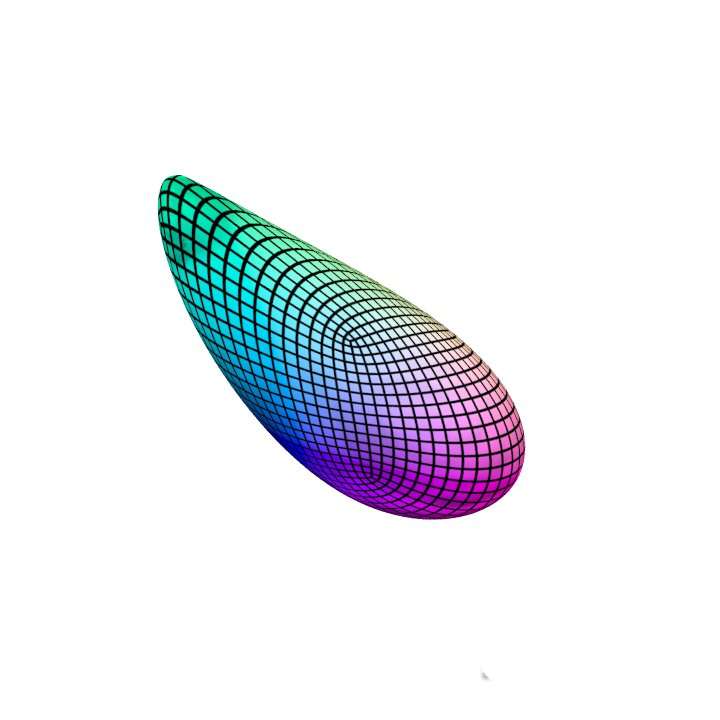}
\includegraphics[width=\locallength]{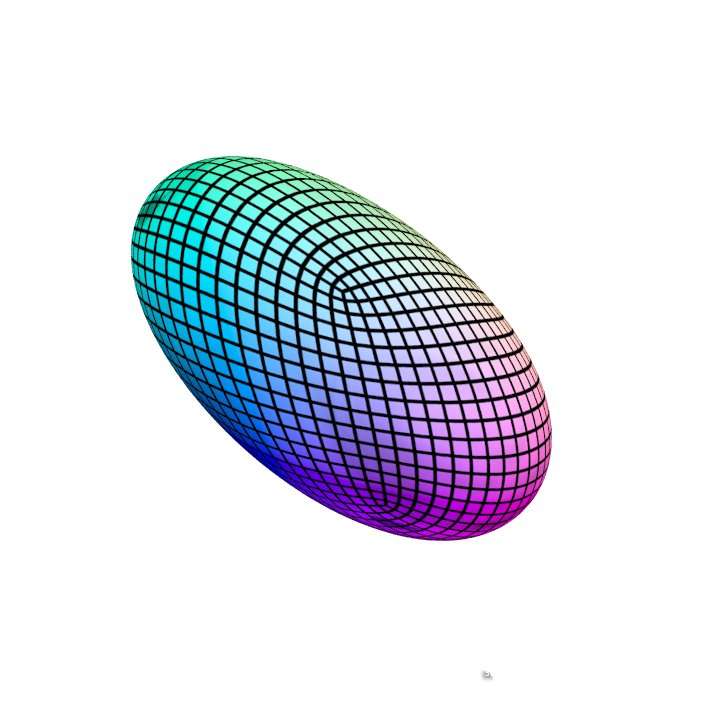}
\includegraphics[width=\locallength]{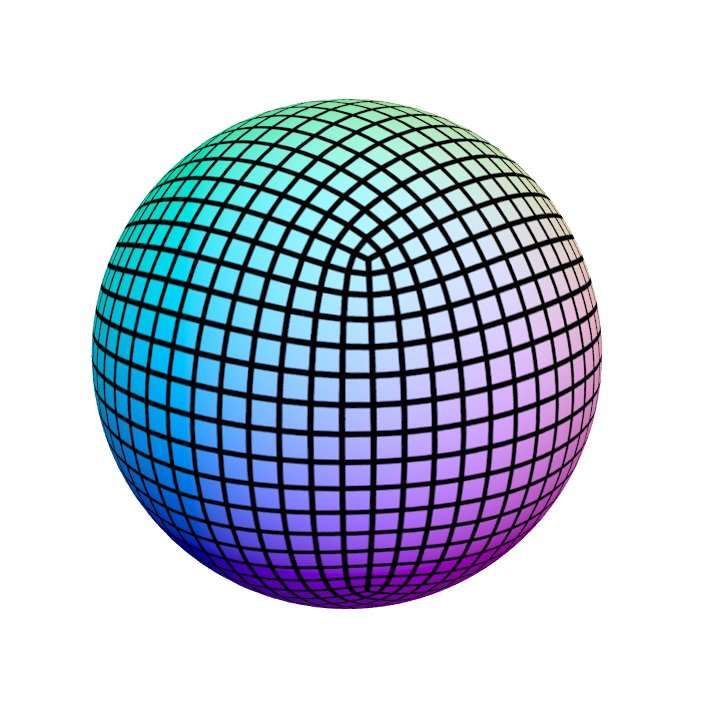}
\caption{Multiple iteration steps of Kazhdan et al's algorithm for the creation of a conformal map. The texture was applied to the final sphere and then projected back to the previous steps.}
\label{fig:mcf}
\end{figure}

Please note that this discrete conformal mapping represents a conformal mapping between smooth surfaces that is bijective! This is evident as the creation of this map emerged through the distortion of the original mesh.

\subsection{Discrete Conformal Factor}
\label{sec:Discrete Conformal Factor}

The discrete conformal map preserves the angles at each triangle but not the triangle scale. The conformal factor can be intuitivley understood as the rescaling factor of edges and triangles, but for our cause we want to know these rescaling factors for each point. To define the conformal factor on points we use the discrete differential geoemtry (DDG) styled definition established by Springborn et al. \cite{Springborn:2008:CET:1360612.1360676}. In a nutshell it calculates the conformal factor on each point through the rescaling ratios in the surrounding of the point. Figure \ref{fig:conformal factor} shows triangles being mapped to each other.

\begin{figure}
\centering
\includegraphics[width=5cm]{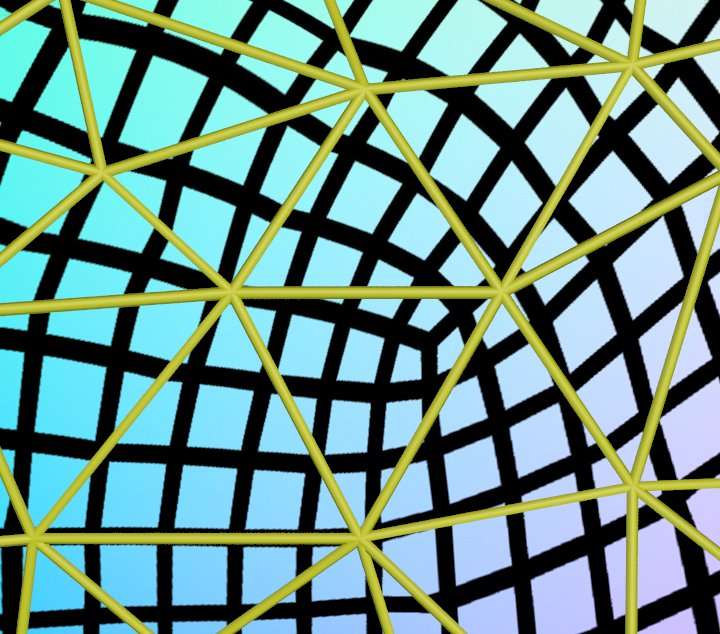}
\includegraphics[width=5cm]{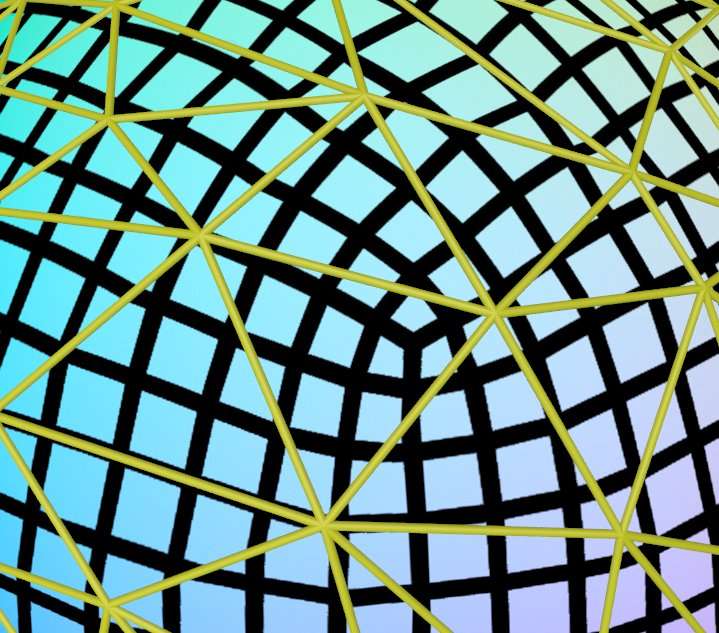}
\caption{Conformal mapping in action seen from up close. Triangles on the bunny (left). Corresponding triangles on the sphere (right). The triangles look similar and the conformal factor represent the size scaling of these triangles.}
\label{fig:conformal factor}
\end{figure}

The discrete conformal factor at each point $p_i$ of the $N\in\N$ points will be store as a scalar value $h_i$. It is a common practice to represent conformal factors between two metrics $g,\tg$ using the equation $\tg=h^2g=e^ug$. In the DDG theory proposes that the discretized values of $u$ on the points have to fulfill this equation:

\begin{equation}
\tl_{ij}=e^{(u_i+u_j)/2}l_{ij}
\end{equation}

where $\tl_{ij},l_{ij}$ are the edge lengths between two points in $M$ and $\Sp$ respectively. We can thus solve for the individual $u_i$ values through

\begin{equation}
e^{u_i}=\frac{\tl_{ij} \ \tl_{ki} \ l_{jk} }{l_{ij} \ \tl_{ki} \ \tl_{jk} }
\end{equation} 

using the triangle $ijk$ and then set $h_i:= \sqrt{e^{u_i}}$. In order to catch the conformal factors within each triangle we rely again on linear interpolations of the values between triangle corners (as later specified in section \ref{sec:Making Use of the Conformal Mapping}. Figure \ref{fig:conformal color} shows examples of colored conformal factors on surfaces. Note that in many example the conformal factor ranges between numbers from near zero up to many thousands depending on the surface.

\begin{figure}
\centering
\def\lll{5cm}
\includegraphics[height=\lll]{images/houdini/sphere_conformal}
\includegraphics[height=\lll]{images/houdini/armadillo_conformal}
\includegraphics[height=\lll]{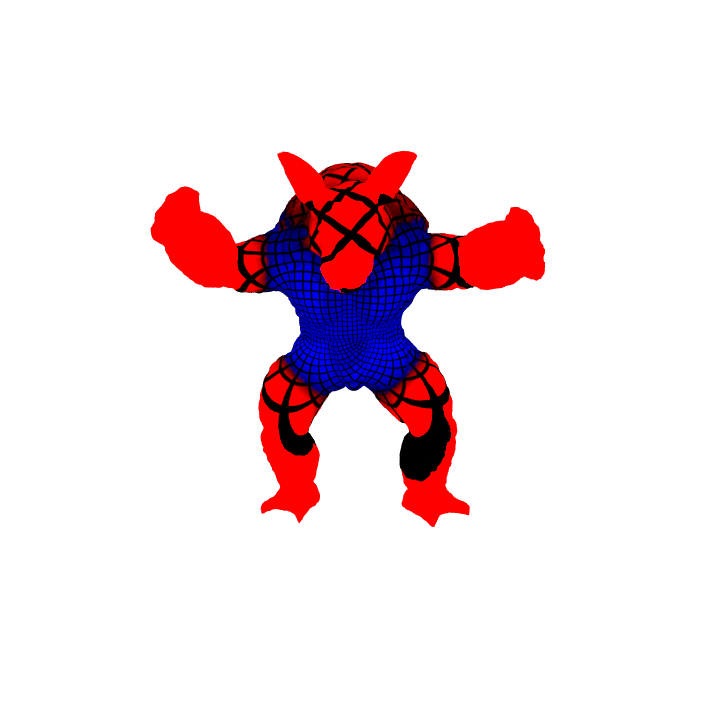}
\caption{Conformal mapping from the sphere (left) to the armadillo (center) with highlighted conformal factor values (right). Red mean high conformal factor and thus more texture stretching.}
\label{fig:conformal color}
\end{figure}

All these steps mentioned until here are the preprocessing steps to smoothly run the closed surface vortex dynamics. Next we will explain how the conformal maps are used in our implementation. 

\subsection{Making Use of the Conformal Mapping}
\label{sec:Making Use of the Conformal Mapping}

Our point vortices want to live anywhere on the mesh, on points, edges and faces, but right now our discrete conformal mapping and factors are point-to-point mappings. We now quickly explain how to get a surface-to-surface mapping from each corresponding triangle from the mesh $M$ to the corresponding triangle in the mesh $\Sp$.

One way to archive this is through the use of barycenter coordinates of each triangle. This is motivated by the success of the  texture mapping that really looks as if the angles on the texture remain preserved. An analogous term for this would be \emph{piece wise linear mapping} from each triangle to another.

The way to perform this mapping for a point $p\in M$ on the arbitrary mesh is to first determine the triangle it lies on (a triangle if $p$ lands on an edge or point) together with the barycenter coordinates $(s,t)$ of $p$ within that triangle with the corners $p_1,p_2,p_3$. Through the point corners of this triangle we can then find the corresponding triangle on $\Sp$ with corners $f(p_1),f(p_2),f(p_3)$ using the point to point mapping $f$. At last we use the same barycenter coordinates $(s,t)$ again to define $f(p)$.

This scheme is general for point-point mappings $f:A\rightarrow B$ representing bijective mappings on smooth surfaces and is shown on figure \ref{fig:Barycenter}. In short we do these 3 steps for $p\in A$:

\begin{enumerate}
\item Locate the triangle $p$ is part of and collect it's corners $p_1,p_2,p_3$
\item Find $(s,t)$ such that  $p = p_1 + s(p_2-p_1) + t(p_3-p_1)$
\item Define $f(p):= f(p_1) + s(f(p_2)-f(p_1)) + t(f(p_3)-f(p_1))$
\end{enumerate}

\begin{figure}[H]
\centering
\def\svgwidth{0.8\textwidth}
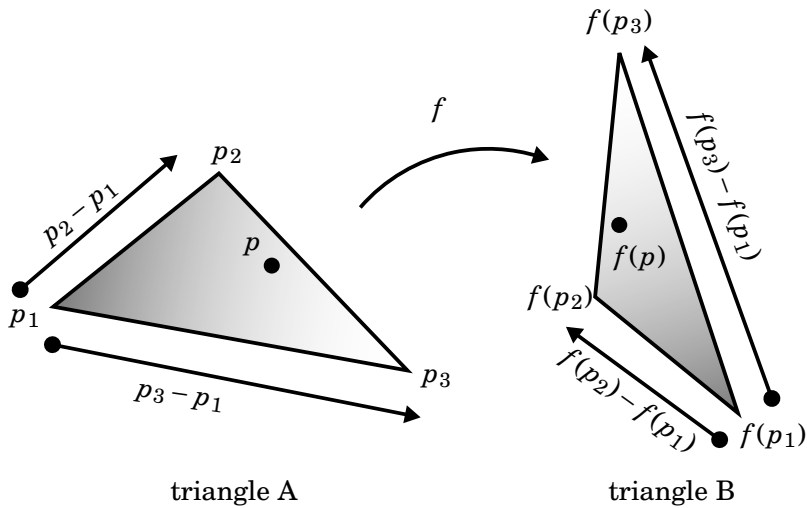
\caption{An example of how two triangles are mapped to each other using barycenter coordinates.}
\label{fig:Barycenter}
\end{figure}

Keep in mind that now every triangle is mapped to another triangle between two closed surfaces. This means that our map is \emph{bijective}, and that we automatically posses an inverse map that is also a conformal map using the exact same method.

Note that we use this scheme not only for the conformal map $f$ but also for the conformal factor $h$. Of course, this is the part where the dependence on the resolution kicks in. With more vertices we can more accurately make conformal mappings which then make the mappings more precise.

\subsection{Point Sampling}

The consequeces of using a new metric become very evident when observing the sampling of points from the surfaces $M$ and $\Sp$ after mapping with the conformal bijective map $f$.

In our implementation we work most of the time on $\Sp$. If we where to just uniformly sample the points on $\Sp$ and map them back to $M$ via $f^{-1}$ we'd end up with a very unbalanced distribution of points. This happens mainly because areas in $M$ that get mapped by $f$ with high conformal factor $h$ arrive at much tinier areas and are thus underrepresented when uniformly sampled.

We can circument this problem by the use of the conformal factor when sampling or by weigthing the sample process by the areas of the triangles in $M$ if we were to sample directly on $\Sp$. So the probabilty for a point $p\in\Sp$ to appear in the triangle $ijk$ of the sphere will now be computed as

$$P( \ p\in \text{triangle}_{\subset\Sp}(ijk) \ ) = \frac{ A(\text{triangel}_M(ijk) }{ A(M) }$$

Figure \ref{fig:area sample} shows the differences that this makes. It evidently shows the expansion and contraction of area by the conformal mapping.

\begin{figure}
\centering
\def\locallength{0.24\textwidth}
\includegraphics[width=\locallength]{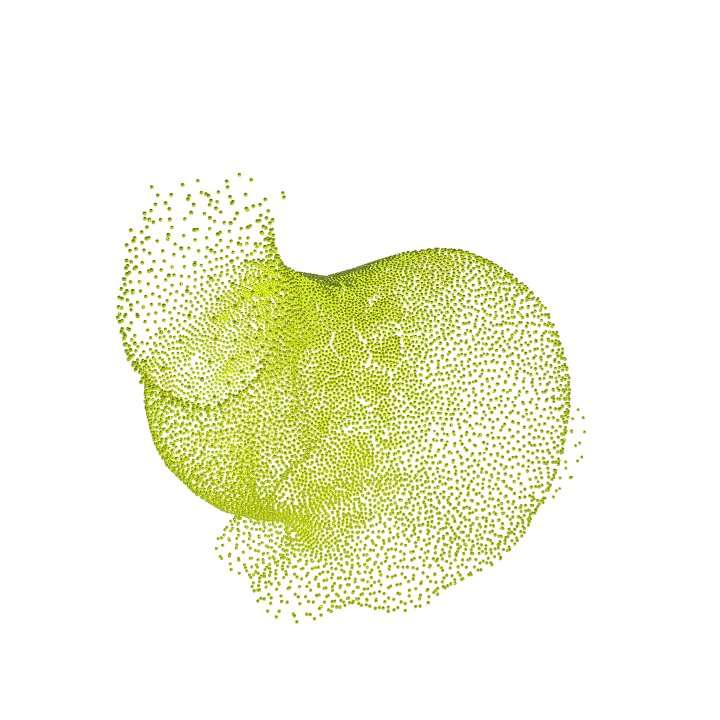}
\includegraphics[width=\locallength]{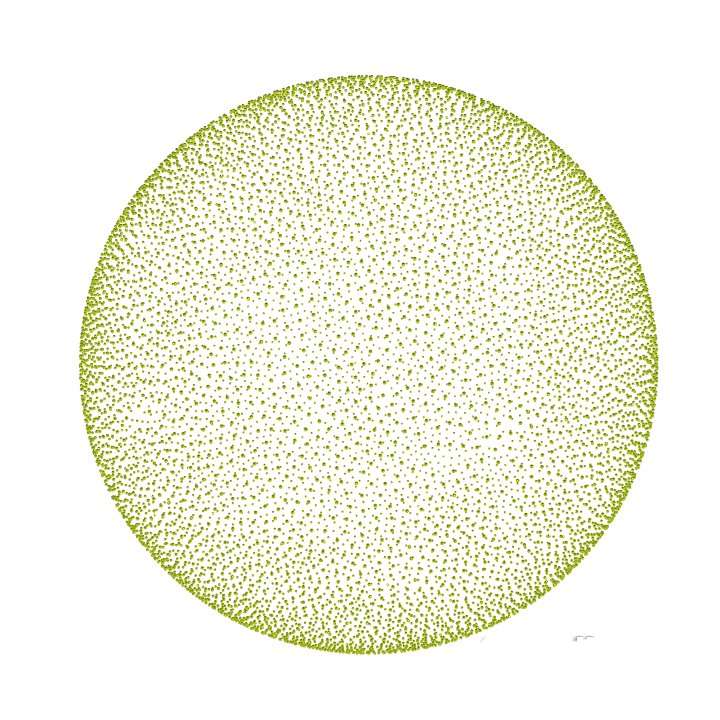}
\includegraphics[width=\locallength]{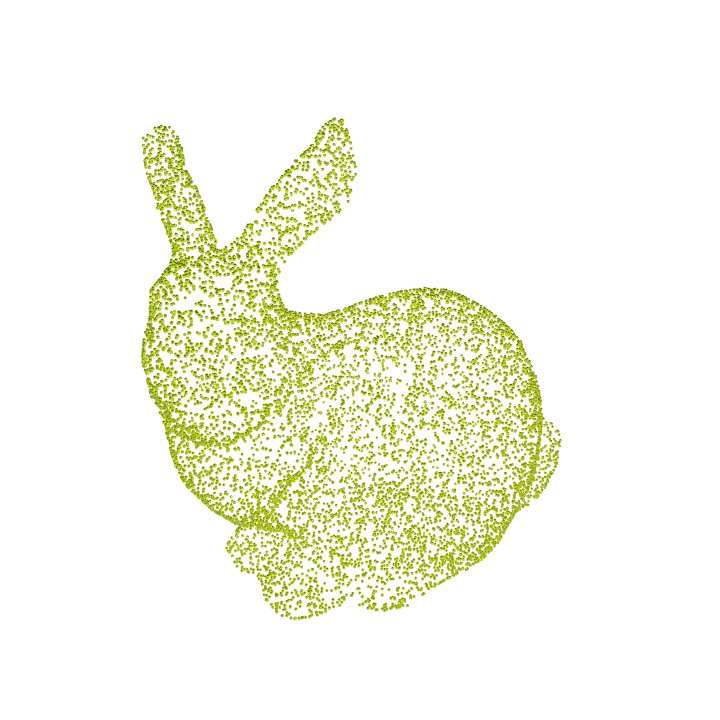}
\includegraphics[width=\locallength]{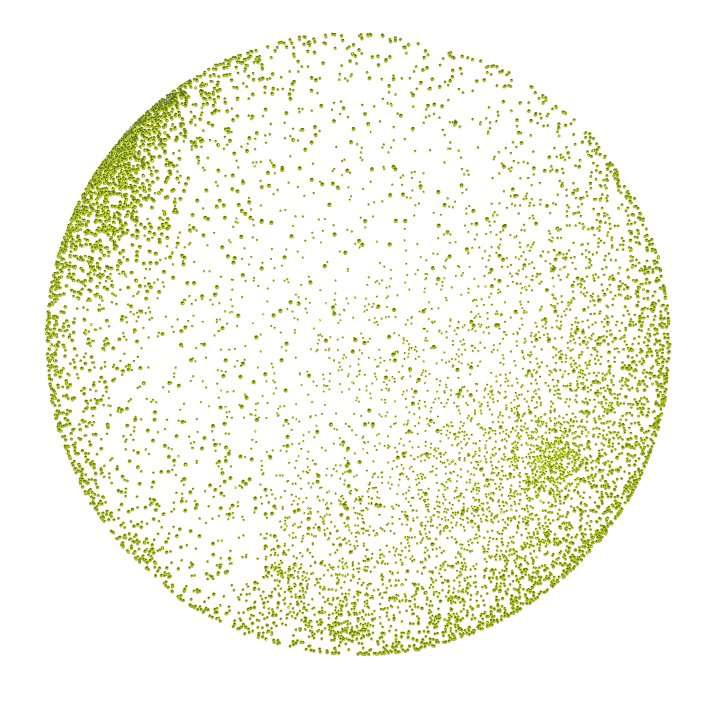}
\caption{The bunny and the sphere sampled on $\Sp$ uniformly (left pair) and using the weights (right pair).}
\label{fig:area sample}
\end{figure}

\subsection{Discrete Gradient}

For the discrete gradient we grab the piece wise linear established definition for scalar functions on points on triangle meshes as covered in section \ref{sec:Making Use of the Conformal Mapping}. The gradient of a scalar function $h$ on the points will thus be constant on each triangle. Given the set of triangles $F$ where each triangle is denoted by three point indices $ijk$ we define the gradient as

%$$ \grado: F\times\R \longrightarrow \R $$
\begin{equation}
\grado_{ijk}h=\frac{1}{2A} \left ( \ (p_3-p_2)h(p_1) + (p_1-p_3)h(p_2) + (p_2-p_1)h(p_3) \ \right )
\end{equation}

Where $A$ is the triangle area.Figure \ref{fig:gradient definition} displays the formula. This definition will be important for the use of the gradient of the conformal factor.

\begin{figure}[H]
\centering
\def\svgwidth{0.5\textwidth}
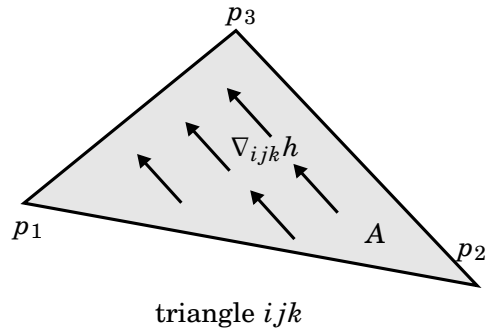
\caption{A constant gradient on a triangle.}
\label{fig:gradient definition}
\end{figure}

\subsection{Vortex Dynamics}

Initially, our vortices are arbitrarily placed on the surface mesh $M$, meaning that they can lie anywhere on any triangle. Using the discrete conformal map $f:M \rightarrow \Sp$ we then project the point vortices on $\Sp$. From then on we only apply the closed surface vortex dynamics on the vortices on $\Sp$ and the corresponding vortices on $M$ only become puppets that are completely dependent on the dynamics of the vortices on $\Sp$ through the inverse conformal map $f^{-1}:S\rightarrow M$. Figure \ref{fig:closed_surfaces_mappings} illustrate this relation.

\begin{figure}[H]
\centering
\def\svgwidth{0.49\textwidth}
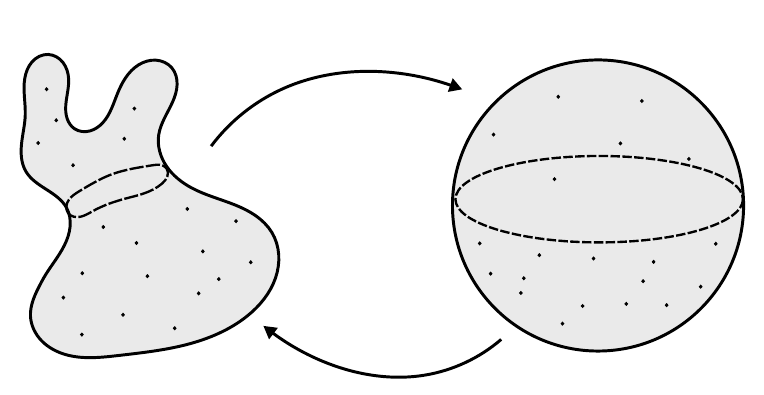
\includegraphics[width=0.49\textwidth]{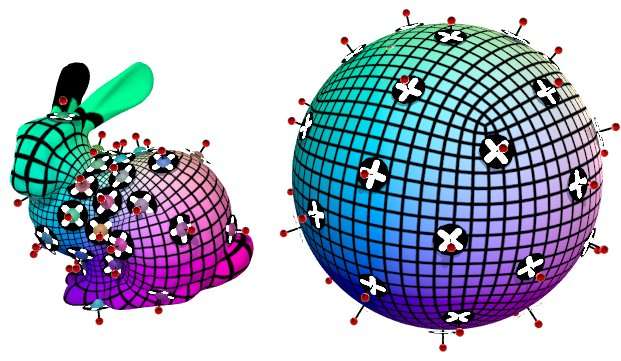}
\caption{Conformally mapping the surface to $M$ to the sphere $\Sp$ (above). Mapping vortices between $M$ and $\Sp$ (below).}
\label{fig:closed_surfaces_mappings}
\end{figure}

At last we can finally focus on the motion of the point vortices alone. Now that we have the conformal maps and factors all implemented and ready to use we will directly apply the closed surface vortex dynamics only on the point vortices on the sphere $\Sp$ through the result of theorem \ref{prop:\generalvd}.

$$
\vu(p_j) = \frac{1}{4\pi h(p_j)^2} \left( \sum_{  \substack{i=1 \\ i\neq j}}^n \left( \w_i \frac{p_j \times p_i}{ 1 -  p_j\cdot p_i} \right)  -   \frac{1}{h(p_j)} \w_jp_j\times\grado_{p_j}h(p_j) \right)
$$

Analogously to the spherical case we advect the point vortices in a curved fashion instead of the primitive straight way as explained in section \ref{sec:\sphericalvd Implementation}. %The only difference now is that the differential equation has changed and that it uses the conformal factor from within each triangle through linear interpolation of the conformal factors on the points of each triangle.

And after finishing the dynamics on $\Sp$ we of course have to transport the results back onto $M$ using the inverse of our conformal map. The procedure is summed up in algorithm \ref{alg:it} and an overview can be seen in figure \ref{fig:overview_of_algorithm}.

\begin{figure}
\centering
\def\svgwidth{0.999\textwidth}
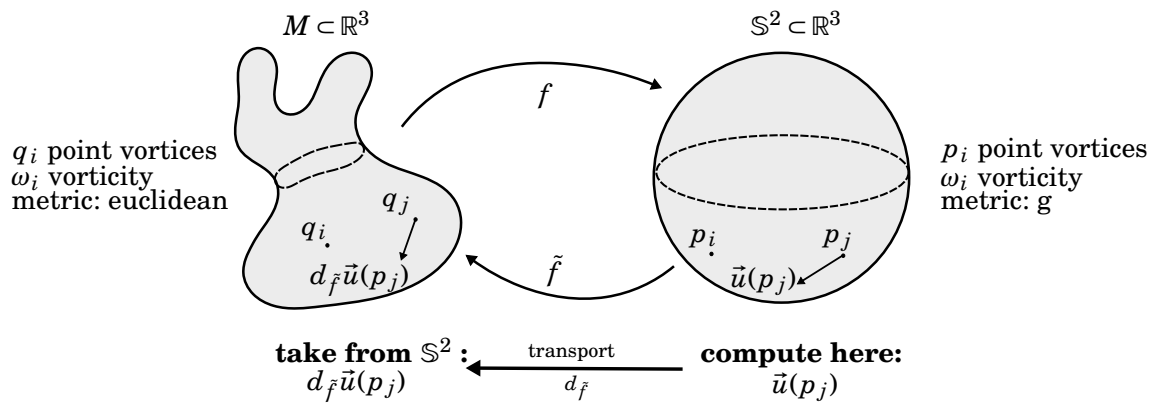
\caption{Our algorithm in a nutshell. $\Sp$ is responsible for the dynamics while $M$ does the visualization. The conformal maps $f,\tf=f^{-1}$ provide the connection.}
\label{fig:overview_of_algorithm}
\end{figure}

%\begin{algorithm}[H]
%\SetAlgoLined
%\KwResult{Write here the result }
% initialization\;
% \While{While condition}{
%  instructions\;
%  \eIf{condition}{
%   instructions1\;
%   instructions2\;
%   }{
%   instructions3\;
%  }
% }
% \caption{How to write algorithms}
%\end{algorithm}

\begin{algorithm}
\label{alg:it}
  \caption{Point Vortex Dynamics on Closed Surfaces}
  \begin{algorithmic}[1]
    \State \textbf{Given:}
    \State triangulated closed surface mesh $M$ of genus zero
    \State $n\in\N$ point vortices $q_i\in M$ with vorticity $\w_i\in\R$
    \State $\sum \w_i = 0$
    \State \textbf{Init:}
    \State create bijective conformal map $f:M\longrightarrow \Sp$
	\State create points $p_i:=f(q_i)\in\Sp$
    \State give every point a conformal factor $f:\Sp\longrightarrow \R$
    \State \textbf{Loop:}
	\State \quad compute all $\vu(p_i)$ using theorem \ref{prop:\generalvd}
	\State \quad update all $p_i$ by advection on $\Sp$
	\State \quad update all $q_i \mapsto f^{-1}(q_i)$
  \end{algorithmic}
\end{algorithm}

That's it.

%laplace intro nofree lunch and minimal surf. paper cite
%=============================================================================
%\import{chapters/chapter08/}{chap08.tex}
%\clearemptydoublepage

%+ + + + + + + + + + + + + + + + + + + + + + + + + + + + + + + + + + + + + + + + + + + + + + + + + + + + + + + + + + + + + + + + + + + + + + + + + + + + + + + + + + + + + + + + + + + + + + + + + + + + + + + + + + + + + + + + + + + + + + + + +    NEW CHAPTER     + + + + + + + + + + + + + + + + + + + + + + + + + + + + + + + + + + + + + + + + + + + + + + + + + + + + + + + + + + + + + + + + + + + +

\chapter{Conclusion}
\label{chap:Conclusion}

In this thesis we presented a complete guide from the theoretical basis of point vortices with all of its justifications up to the implementation on closed surfaces of genus zeros. Starting with very general statements about fluid dynamics we provided all mathematical tools necessary to deal with vorticity on surfaces. We then rigorously explained point vortices and their motivation in a completely general setting to later apply to distinct surfaces. This allowed us to summarize previous results about point vortices for the plane and the sphere as well.

In addition we applied the point vortex theory results from Boatto and Koiller\citep{Boatto2015} in an easy to follow manner with the conformal map creation algorithm from Kazhdan, Solomon, Ben-Chen and Mirela \cite{Kazhdan:2012:MFM:2346796.2346809} while explaining the crucial implementation steps for the whole project.

The main aim was to establish a comprehensive basis for the theory of point vortices. %While their usefulness for fluid simulations with today's demands can be questioned, their contribution to the understanding of vortex methods makes them worth to study.

\textbf{Acknowledgements}: I thank Prof. Dr. Ulrich Pinkall for the supervision of this master thesis and Dr. Felix Knöppel and Dr. Albert Chern for all of their contributions during the discussions.

%=============================================================================
%\import{chapters/chapter09/}{chap09.tex}
%\clearemptydoublepage

%+ + + + + + + + + + + + + + + + + + + + + + + + + + + + + + + + + + + + + + + + + + + + + + + + + + + + + + + + + + + + + + + + + + + + + + + + + + + + + + + + + + + + + + + + + + + + + + + + + + + + + + + + + + + + + + + + + + + + + + + + +    NEW CHAPTER     + + + + + + + + + + + + + + + + + + + + + + + + + + + + + + + + + + + + + + + + + + + + + + + + + + + + + + + + + + + + + + + + + + + +

%==========================================================

% And the appendix goes here
%\appendix

%========================================================================================================================================================================================================================================

% FINISHING STUFF

% Apparently the guidelines don't say anything about citations or
% bibliography styles so I guess we can use anything.
\backmatter
% For arXiv: use pre-generated .bbl file instead of \bibliography
\refstepcounter{chapter}
\newcommand{\etalchar}[1]{$^{#1}$}

\clearemptydoublepage
%
% Add index
%\printindex
%   

%sources of some images
% wooden log https://commons.wikimedia.org/wiki/File:K%C5%82oda.png

\end{document}